\documentclass[opre,sglanonrev]{informs4}

\RequirePackage{tgtermes}
\RequirePackage{newtxtext}
\RequirePackage{newtxmath}
\RequirePackage{bm}
\RequirePackage{endnotes}
\usepackage{float}

\OneAndAHalfSpacedXII

\usepackage{graphicx}
\usepackage{comment}
\usepackage{url}
\usepackage{mathtools}
\usepackage{amsfonts}
\usepackage{amssymb}
\usepackage[dvipsnames,svgnames,x11names]{xcolor}
\usepackage{enumitem}
\usepackage[ruled,vlined]{algorithm2e}

\usepackage{natbib}
\bibpunct[, ]{(}{)}{,}{a}{}{,}%
\def\bibfont{\small}%
\usepackage[colorlinks,citecolor=black,linkcolor=black]{hyperref}

\EquationsNumberedThrough

\TheoremsNumberedThrough
\ECRepeatTheorems

\MANUSCRIPTNO{}

\begin{document}



\RUNTITLE{Online Bipartite Matching with Reusable Capacity under Non-Stationary Rewards}

\TITLE{Online Bipartite Matching with Reusable Capacity under Non-Stationary Rewards}

\ARTICLEAUTHORS{%
 \AUTHOR{Xi Chen\footnotemark[1]}
 \AFF{Leonard N. Stern School of Business, New York University, New York, NY 10012, USA, \EMAIL{xc13@stern.nyu.edu}}  \AUTHOR{Shixin Wang\footnotemark[1]}
 \AFF{H. Milton Stewart School of Industrial and Systems Engineering, Georgia Institute of Technology, Atlanta, GA 30332, USA, \EMAIL{shixin.wang@isye.gatech.edu}} \AUTHOR{Bingkun Zhou\footnotemark[1]}
 \AFF{Qiuzhen College, Tsinghua University, Beijing 100084, China, \EMAIL{zbk23@mails.tsinghua.edu.cn}}
 \AUTHOR{Yuan Zhou\footnotemark[1]}
 \AFF{Yau Mathematical Sciences Center \& Department of Mathematical  Sciences, Tsinghua University, Beijing 100084, China; Beijing Institute of Mathematical Sciences and Applications, Beijing 101408, China \EMAIL{yuan-zhou@tsinghua.edu.cn}}
}
\renewcommand{\thefootnote}{\fnsymbol{footnote}}
\footnotetext[1]{Author names listed in alphabetical order.}

\ABSTRACT{%
We study online bipartite matching with reusable server capacity and non-stationary rewards. Jobs arrive sequentially, reveal compatible servers, reward rates, and processing durations, and must be accepted or rejected irrevocably. An accepted job occupies one unit of server capacity only during its processing interval, so an assignment may displace an unknown sequence of future jobs. Existing guarantees are typically calibrated by a global reward range, which can become arbitrarily large when rewards drift over a long horizon. We instead impose a locally bounded reward condition: reward rates of jobs that can compete for the same server within a relevant time window differ by at most a factor $\delta$. Under this condition, we develop two BALANCE-type algorithms with time-aware opportunity-cost losses. TS-BAL maximizes cumulative blocking losses over feasible reuse schedules and achieves a competitive ratio of $2\ln(\delta D)+\mathcal O(\ln\ln(\delta\vee D))$. GR-BAL uses a greedy relaxation of this loss and achieves $\ln(\delta D)+\mathcal O(\ln\ln(\delta\vee D))$, matching a lower bound of $\ln(\delta D)$ in the leading term. Numerical experiments demonstrate robust performance under substantial global reward drift and favorable finite-capacity performance.
}%




\maketitle


\section{Introduction}\label{sec:intro}
The real-time allocation of limited supply to incoming demand is central to the operation of many online platforms. Classical applications, including display advertising and online retail, often model an accepted request as permanently consuming the corresponding inventory or budget. In many modern service systems, however, capacity is reusable. A cloud server, vehicle, or other service unit is occupied only for the duration of an accepted job and becomes available again after the service is completed. The allocation decision must therefore account not only for current capacity scarcity, but also for how capacity will be released and reused over time.

Online job assignment, also known as online bipartite matching with reusable servers, provides a general model for allocating reusable resource capacity. A platform manages a fixed collection of capacitated servers, while jobs arrive sequentially over time. Upon a job's arrival, the platform observes its compatible servers and the reward rate and processing duration associated with each compatible assignment. The platform must then immediately and irrevocably assign the job to an available compatible server or reject it. An accepted job occupies one unit of capacity for its processing duration and releases that capacity upon completion. The platform seeks to maximize its total reward without knowing future arrivals.

A motivating application arises in managed GPU-cloud systems. A centralized scheduler manages multiple types of GPU servers and receives jobs such as fine-tuning, batch-inference, and embedding-generation requests. Compatibility may depend on GPU architecture, memory requirements, software environments, data locality, and other operational constraints. A job's processing duration may vary across compatible servers, while its reward rate may reflect its priority, service tier, latency requirement, or the prevailing economic value of GPU capacity. These assignment characteristics are revealed only when the job arrives, and the scheduler must decide immediately whether and where to process it.

\citet{feng2025online} study a general online job-assignment model that captures reusable server capacity together with assignment-specific rewards and processing durations. They assume that all reward rates over the entire horizon lie within a common global range, whose maximum-to-minimum ratio is denoted by \(R\), and that the ratio between the maximum and minimum processing durations is bounded by \(D\). Under these assumptions, their Forward-Looking BALANCE (FLB) algorithm achieves the optimal leading-order competitive ratio \(\ln(RD)\) in the large-capacity regime.

In this paper, we consider settings in which the economic conditions of a long-lived platform evolve over time. In such settings, the global reward range may be too large to yield meaningful performance guarantees. Jobs arriving close in time and competing for the same server capacity may have comparable reward rates, while the prevailing reward scale may drift substantially across temporally separated operating periods. Even modest changes across successive periods can accumulate into an arbitrarily large global reward range. Consequently, a competitive guarantee calibrated by this global range may deteriorate with the operating horizon, even when reward variation among jobs that directly compete for capacity remains limited.

We address this issue through a locally bounded reward condition. Informally, processing durations are bounded by \(D\), and the reward rates of two jobs compatible with the same server differ by at most a factor \(\delta\) whenever their arrival times are within \(D\) of each other. Because no job occupies capacity for more than \(D\) time units, this condition controls reward variation within the temporal windows in which direct capacity conflicts can arise. On the other hand, the local parameter \(\delta\) can remain small even when the global reward range becomes arbitrarily large.

In this paper, we study the online job assignment problem under locally bounded reward variation. Our goal is to design online algorithms with asymptotically optimal competitive ratios, even when rewards vary substantially over the full horizon.

\subsection{Our Results}

We establish asymptotically tight competitive-ratio bounds for online job assignment under local reward variation. For integer-valued processing durations and in the large-capacity regime, our main algorithm achieves an asymptotic competitive ratio of $\ln(\delta D)+\ln\ln(\delta\vee D)+\mathcal{O}(1)$ (Theorem~\ref{tm:main_theorem_refine}).
Complementing this upper bound, we prove that every online algorithm has competitive ratio at least $\ln(\delta D)$ (Theorem~\ref{tm:lb_main}). Thus, the upper and lower bounds have the same leading term \(\ln(\delta D)\), and our algorithm attains the optimal leading coefficient.  We also extend our algorithm to real-valued processing durations and obtain an asymptotic competitive ratio of
$\ln(\delta D)+2\ln\ln(\delta\vee D)+\mathcal{O}(1)$, which also attains the optimal leading coefficient (Theorem~\ref{thm:real-duration}). Our analysis additionally provides finite-capacity guarantees, and the stated asymptotic bounds are recovered whenever the minimum server capacity $c_{\min}\gg\ln^2(\delta D)$. Our numerical experiments further show that our algorithms perform robustly under substantial global reward drift and achieve favorable performance at finite capacities relative to the benchmark policies we consider.

Comparing our results with those of \cite{feng2025online}, we note that, under the global reward bound \(R\), their Forward-Looking BALANCE algorithm achieves
$\ln(RD)+\ln\ln(R\vee D)+\mathcal{O}(1)$
for integer-valued durations and
$\ln(RD)+3\ln\ln(R\vee D)+\mathcal{O}(1)$
for real-valued durations. Their global condition implies our local condition with \(\delta=R\). Consequently, when specialized to globally bounded rewards, our guarantee recovers their asymptotic upper bound and slightly improves it for real-valued durations. The main distinction arises when rewards drift over time. The local parameter \(\delta\) may remain small while the global reward ratio \(R\) becomes arbitrarily large, causing a guarantee calibrated by $R$ to deteriorate with the operating horizon. In contrast, our guarantee depends only logarithmically on the local reward variation and remains independent of the reward variation accumulated over the full horizon.

\subsection{Technical Challenges}

Our algorithms build on the BALANCE framework, which has been widely used in prior works that studied online bipartite matching with adversarial arrivals \citep{mehta2007adwords,golrezaei2014real,goyal2025asymptotically,huang2024online,feng2025online}. Upon the arrival of a job, a BALANCE-type algorithm evaluates each compatible server by subtracting an assignment-specific loss from the immediate reward. It selects the server with the largest positive adjusted reward and rejects the job if every adjusted reward is non-positive. The loss represents the opportunity cost of occupying capacity, namely, the future reward that may be forgone by accepting the current job.

The main challenge is to construct an opportunity-cost loss that simultaneously captures capacity reusability and local reward variation. Because capacity is reusable, preserving one unit of capacity may enable a sequence of future assignments rather than a single assignment. A current job may therefore block multiple future jobs along an unknown reuse schedule, and the loss must aggregate the corresponding blocking losses. Moreover, the reward of each potentially blocked job depends on its arrival time, while the relevant capacity state also changes over time as previously accepted jobs complete. The loss must consequently combine time-aware estimates of future rewards and capacity availability and aggregate them over an unknown feasible reuse schedule. More detailed discussion on these challenges can be found in Section~\ref{sec:algorithm-balance-framework-challenge}.

\citet{feng2025online} address reusable capacity under globally bounded rewards by evaluating projected congestion at a prescribed collection of future reuse times. A direct extension of this approach is insufficient under local reward variation. We show that incorporating time-aware reward estimates while retaining the unit-spaced reuse schedule used by FLB algorithm can still result in an \(\Omega(\delta)\) competitive ratio (see the discussion in Remark~\ref{rem:comparison_feng} in Section~\ref{subsec:TS-BALANCE}). Thus, achieving a logarithmic dependence on \(\delta\) requires a new opportunity-cost loss construction rather than merely a local recalibration of the existing loss.

\subsection{Technical Contributions}

\noindent\underline{\bf A time-aware schedule-based loss.}
We first develop the \emph{Time-aware Schedule-based BALANCE} algorithm (TS-BAL). For each candidate assignment, TS-BAL constructs a blocking loss at every future time during the job's processing interval. This loss combines the history-based time-aware estimates of the reward rate and the server capacity available at that time. TS-BAL then maximizes the cumulative blocking loss over all feasible reuse schedules within the processing interval. The resulting loss directly captures the largest opportunity cost that the current assignment may impose on a reusable unit of capacity and can be computed through dynamic programming.

We prove that TS-BAL achieves an asymptotic competitive ratio of $2\ln(\delta D)+2\ln\ln(\delta\vee D)+\mathcal{O}(1)$ in Theorem~\ref{tm:main_theorem}.
This guarantee establishes the desired logarithmic dependence on local reward variation bound, but it contains an additional factor of two in the leading term.

\noindent\underline{\bf A greedy relaxation with the optimal leading coefficient.}
To remove this factor of two, we develop the \emph{Greedy-Relaxation BALANCE} algorithm (GR-BAL). GR-BAL partitions each processing interval into unit-length subintervals and independently selects the largest blocking loss from each subinterval. Because a feasible reuse schedule contains at most one reuse time in each subinterval, the resulting greedy-relaxation loss upper-bounds the schedule-based loss. It also admits a simpler computation based on the reward order statistics of previously accepted jobs that remain active. In Section~\ref{subsec:GR-BAL}, we prove that GR-BAL eliminates the extra factor of two and achieves the optimal leading coefficient in its asymptotic competitive ratio.

\noindent\underline{\bf A primal-dual charging argument based on loss increments.}
Building on the configuration-LP benchmark of \cite{feng2025online}, we develop new dual constructions and charging arguments tailored to our opportunity-cost losses. For each offline configuration, we charge its blocking losses to loss increments generated by online assignments whose processing intervals overlap the configuration. The main analytical task is to aggregate these local charges across disjoint time intervals and thereby establish dual feasibility. 

This aggregation step is precisely where the greedy relaxation provides its key advantage. The loss increments associated with TS-BAL satisfy only weak subadditivity: locally worst reuse schedules over disjoint intervals cannot necessarily be combined into a single feasible reuse schedule, and partitioning them into two feasible schedules introduces the factor of two. In contrast, the loss increments associated with GR-BAL are fully subadditive, allowing charges from different portions of an offline schedule to be aggregated without the factor-of-two overhead. 

\noindent\underline{\bf An additive lower-bound construction.}
Our lower-bound construction captures the joint difficulty created by reward and duration heterogeneity. It combines increasing reward rates, which produce the \(\ln\delta\) term, with heterogeneous processing durations, which produce the harmonic term \(H_D\). A common prefix-instance argument yields the additive lower bound
$H_D+\ln\delta\geq \ln(\delta D)$,
showing that these two sources of uncertainty accumulate.

\subsection{Operations Insights}
Our algorithm design and analysis highlights how a platform should price a reusable resource with changing rewards: a reusable resource should be priced by the value of the feasible future service trajectory that a current decision may displace, rather than by its current scarcity or any single anticipated future use. Each potential future use should be valued according to the reward and capacity conditions prevailing when that use would occur, and these values should be aggregated only along operationally feasible reuse trajectories. Because the future trajectory is unknown at the time of decision, a robust capacity price should reflect the most costly feasible continuation that may be disrupted.

\section{Related Works}
Our work is closely related to the online bipartite matching literature. In this section, we review the relevant work and position our contribution within this literature. Throughout, we adopt the standard terminology: offline vertices are referred to as \emph{resources} or \emph{servers}, corresponding to the servers in our model, whereas online vertices represent arriving jobs or requests.

\noindent\underline{\textbf{Non-reusable resources.}}
Online allocation of non-reusable resources under adversarial arrivals is a foundational model for sequential resource management. In classical unweighted online bipartite matching, \citet{karp1990optimal} establish the optimal competitive ratio of $e/(e-1)$. Subsequent work extends this framework to vertex-weighted matching \citep{aggarwal2011online}, edge-weighted display advertising with free disposal \citep{feldman2009displayads}, online matching, assortment, and pricing with multiple prices \citep{ma2020algorithms}, and AdWords with unknown budgets \citep{udwani2024adwords}. For a broader survey of online matching and its generalizations, we refer readers to \citet{huang2024online}.

Among them, a particularly relevant line develops BALANCE-type algorithms. \citet{kalyanasundaram2000optimal} introduce BALANCE for online bipartite $b$-matching in the large-capacity regime; \citet{mehta2007adwords} extend this principle to AdWords under the small-bid assumption; and \citet{golrezaei2014real} develop an inventory-balancing policy for real-time assortment optimization with large capacities. \citet{ma2020algorithms} further introduce Multi-price Balance, whereas \citet{feng2025batching} develop regularized BALANCE algorithms for fractional allocations with batch arrivals. These studies motivate our general design principle of evaluating each assignment by its immediate reward net of a capacity-dependent loss. Their loss terms price irreversible inventory or budget depletion through fill levels, bid prices, or stage-dependent regularizers. In contrast, our loss must capture the time-aware opportunity cost of occupying reusable capacity over potential future reuse schedule under non-stationary rewards.

Complementary methodological work develops general primal-dual approaches to online optimization. \citet{buchbinder2009online} establish a unified primal-dual framework for online covering and packing with linear objectives, and \citet{azar2016online} extend this framework to convex objectives. Although these models do not explicitly represent reusable capacity-time constraints, their methodologies motivate our primal-dual analysis.

Reward heterogeneity constitutes another closely related theme. \citet{ball2009toward} study single-resource online booking with multiple fare classes and characterize protection-level and order-quantity-control policies whose guarantees depend logarithmically on the global fare ratio $R$. Their model captures globally specified reward heterogeneity for non-reusable inventory rather than temporal reward drift. Nevertheless, it motivates the study of heterogeneity-dependent guarantees. Our model instead permits substantial global drift while controlling only the local reward variation among jobs.

\noindent\underline{\bf Reusable resources and stationary rewards or reward rates.} 
A substantial literature studies online allocation of reusable resources under stochastic arrivals and stationary reward structures \citep{levi2010provably,owen2018price,dickerson2021allocation,rusmevichientong2020DARP,baek2022bifurcating}. More recent work derives sharp large-capacity guarantees \citep{feng2022near,aminian2026bayesian} and studies learning under unknown, time-varying arrival distributions \citep{zhang2025nonstasto}. These models share the reusable-capacity feature of ours but differ in their arrival models and information assumptions. In particular, their nonstationarity, when present, concerns demand rather than assignment-specific reward rates.

The literature under adversarial arrivals is more closely related, but largely assumes stationary rewards or reward rates. Two especially relevant large-capacity results are \cite{feng2019linear} and \cite{goyal2025asymptotically}. The former analyze a BALANCE-type policy and establish a competitive ratio of $\frac{e^2}{(e-1)^2}$, whereas the latter develop a primal-dual-inspired framework for reusable online matching and attain the optimal asymptotic ratio of $\frac{e}{e-1}$ under known, resource-dependent usage-time distributions. These works motivate our use of balancing and primal-dual ideas. Their capacity adjustments, however, are calibrated to stationary reward scales and resource-dependent usage laws, and therefore do not capture the reuse schedule and the time-varying blocking values induced by reusable resources with nonstationary reward rates.

A complementary stream derives capacity-independent guarantees under stationary rewards or reward rates. \cite{gong2022online,baek2026leveraging} analyze the Greedy policy, whereas \cite{delong2022online} develop reranking policies for fixed deterministic reuse times. \cite{gong2022online} further note that their analysis extends to nonstationary rewards, yielding an $\mathcal{O}(R)$ competitive ratio for arbitrary capacities, where $R$ denotes the global maximum-to-minimum reward ratio. Such guarantees depend linearly on reward heterogeneity and are therefore far from optimal in our setting. In contrast, we seek a near-optimal algorithm tailored to locally nonstationary reward rates.

\noindent\underline{\bf Reusable resources and heterogeneous reward rates and durations.}
The closest work to ours is \citet{feng2025online}, which has been discussed in Section~\ref{sec:intro}. Other related work examines more structured models. \citet{huo2022onlineReusableMultiClass} consider a single pool of identical reusable units and customer classes with class-specific reward rates and durations, obtaining protection-level policies under divisibility and monotonicity conditions. \citet{nekouyan2026online} develop relax-and-round algorithms for $k$ identical units, achieving an optimal guarantee when durations are identical and reward rates vary, and an order-optimal guaranty when durations vary but reward rates are identical. \citet{simchi2025greedy} analyze greedy-like policies in a unified adversarial framework with reusable network resources and decaying rewards, whereas \citet{liu2026unified} allow combinatorial resource bundles whose rewards, activation probabilities, and service times deteriorate with cumulative usage. \citet{aminian2026OJS} allow non-resumable preemption with partial rewards and develop index policies that balance an active request's reward rate against its remaining value. These papers study special cases, like decay or endogenous deterioration, while our non-stationarity model arises from those special cases.

\noindent\underline{\bf Reusable resources with other extensions.}
One extension incorporates exogenous replenishment in addition to endogenous resource returns. \citet{feng2025robustness} study online assortment optimization with fixed usage durations and external replenishment, whereas \citet{kang2025replenishment} develop a black-box reduction that transfers guarantees for online allocation algorithms to settings with exogenous replenishment. Related work studies dynamic pricing and control for reusable resources through dynamic programming \citep{besbes2019static,jia2024pricing,willMa2026dynamic}. More broadly, \citet{udwani2024submodular} introduce a unified framework for online submodular and submodular-order welfare maximization with stochastic outcomes, which accommodates reusable-resource models whose induced objectives are generally non-submodular.

\section{Problem Formulation}

We study an online job assignment problem with adversarial arrival, reusable server capacity and non-stationary rewards. There is a finite set of servers, indexed by $[n] := \{1,2,\ldots,n\}$.
Each server $i\in[n]$ has an integral capacity $c_i\in\mathbb{N}$, which represents the maximum number of jobs that can be processed simultaneously on server $i$. Jobs arrive sequentially over a continuous-time horizon. We index jobs by their order of revelation, and denote their arrival times by $0 \le t_1 \le t_2 \le \dots \le t_m$.

The number of jobs, their arrival times, and their characteristics are not known to the online algorithm in advance (that could be set by an adversary). If multiple jobs arrive at the same time, they are still revealed one after another according to their indices. In particular, when the algorithm makes the decision for job $j$, it has not yet observed the information of job $j+1$, even if $t_j=t_{j+1}$. The compatibility between servers and jobs is represented by an edge set $ E \subseteq [n]\times[m] $. Upon the arrival of job $j$, the algorithm observes the set of compatible servers $\mathcal N(j)\coloneqq \{i\in[n]:(i,j)\in E\}$, as well as the reward rate and processing duration associated with each compatible assignment:$\{(r_{ij},d_{ij})\}_{i\in \mathcal N(j)}$.
Here $r_{ij}>0$ denotes the reward rate earned per unit time if job $j$ is processed on server $i$, and $d_{ij}>0$ denotes the corresponding processing duration.

At time $t_j$, the algorithm must immediately and irrevocably either assigns job $j$ to one of its compatible servers with available capacity, or makes no assignment. If job $j$ is assigned to server $i\in \mathcal N(j)$, then it occupies one unit of capacity on server $i$ over the interval $[t_j,t_j+d_{ij})$ and generates total reward $r_{ij}d_{ij}$. An online policy is feasible if, for every server $i\in[n]$ and every time $t$,
\[
\sum_{j:\, j \text{ is assigned to } i}
\mathbb{I}[t_j\le t<t_j+d_{ij}]
\le c_i ,\qquad \forall i\in[n],\qquad \forall t\ge 0.
\]

We next impose a local regularity condition that captures the non-stationary reward structure studied in this paper. We normalize the minimum processing duration to one.

\begin{assumption}[Locally  bounded reward]
\label{assumption:local-bounded-heterogeneity}
There exist parameters $D\ge 1$ and $\delta\ge 1$ such that the following conditions hold.
First, every compatible assignment $(i, j) \in E$ has processing duration bounded by
$1 \le d_{ij} \le D$.
Second, for every server $i\in[n]$ and every pair of jobs $(j,j^{\prime})$ that are both compatible with server $i$, if their arrival times differ by at most $D$, then their reward rates on server $i$ differ by at most a multiplicative factor $\delta$:
\begin{equation}
\frac{1}{\delta}
\le
\frac{r_{ij}}{r_{ij'}}
\le
\delta,
\qquad
\forall i\in \mathcal N(j)\cap \mathcal N(j')
, |t_j-t_{j'}|\le D .
\label{eq:local-reward-ratio}
\end{equation}
We assume throughout the whole procedure the online algorithm knows the parameters $\delta$ and $D$.
\end{assumption}

Assumption~\ref{assumption:local-bounded-heterogeneity} allows reward rates to drift substantially over the full horizon, but restricts reward variation among jobs that can locally compete for the same server capacity. The condition is both server-specific and time-local: jobs that are far apart in time, or jobs that do not share a compatible server, are not required to have comparable rewards.

We evaluate online algorithms through the worst-case competitive analysis. For a problem instance $I$, let $\mathrm{ALG}(I)$ denote the total reward collected by an online algorithm ALG:
\begin{equation}
\mathrm{ALG}(I)
:=
\sum_{j:\ j \text{ is assigned by ALG to some } i}
r_{ij}d_{ij}.
\label{eq:alg-reward}
\end{equation}
Let $\mathrm{OPT}(I)$ denote the maximum total reward attainable by an offline decision maker that knows the entire arrival sequence and all job characteristics in advance, while being subject to the same compatibility and capacity constraints. For fixed $\delta,D\ge 1$ and capacity level $\{c_i\}_{i\in[n]}$, let $\mathcal{I}(\delta,D,\{c_i\})$ denote the set of all finite problem instances satisfying Assumption~\ref{assumption:local-bounded-heterogeneity} with server capacities $\{c_i\}$:
\begin{equation}
\mathcal{I}(\delta,D,\{c_i\})
\coloneqq
\left\{
I:
I \text{ satisfies Assumption~\ref{assumption:local-bounded-heterogeneity}}
\text{ and server } i \text{ has capacity } c_i
\right\}.
\label{eq:instance-class}
\end{equation}

\begin{definition}[Asymptotic competitive ratio]
\label{definition:asymptotic-competitive-ratio}
A possibly randomized online algorithm ALG is asymptotically $\Gamma(\delta,D)$-competitive if
\begin{equation}
\limsup_{\min c_i\to\infty}
\
\sup_{I\in\mathcal{I}(\delta,D,\{c_i\})}
\frac{\mathrm{OPT}(I)}
{\mathbb{E}[\mathrm{ALG}(I)]}
\le
\Gamma(\delta,D),
\label{eq:asymptotic-cr}
\end{equation}
where the expectation is taken over the internal randomness of ALG.
\end{definition}

Our objective is to design online algorithms with near-optimal asymptotic competitive ratios. In the algorithmic analysis, we fix an arbitrary instance $I$ and establish competitive-ratio guarantees that are independent of $I$. When the instance is clear from context, we write \(\mathrm{OPT}\coloneqq \mathrm{OPT}(I)\) and \(\mathrm{ALG}\coloneqq \mathrm{ALG}(I)\).

\subsection{A configuration LP benchmark}\label{se:configuration_LP}
Following \cite{feng2025online}, we use a configuration linear program (LP) that relaxes the offline benchmark. The key idea is to represent the complete schedule followed by a single unit of server capacity as a configuration. For a server $i$, a configuration is a set of compatible jobs whose processing intervals on server $i$ are pairwise disjoint.

Formally, a subset $S\subseteq [m]$ is a feasible configuration for server $i$ if $(i,j)\in E$ for every job $j\in S$, and for every pair of distinct jobs $j,j^{\prime}\in S$, it holds that $
[t_{j^{\prime}},t_{j^{\prime}}+d_{ij^{\prime}})\cap[t_j,t_j+d_{ij})=\emptyset$. Let $\mathcal S_i\subseteq 2^{[m]}$ denote the collection of all feasible configurations for server $i\in[n]$. Consider the following linear program and its dual:

\begin{tabular}{p{0.5\textwidth}|p{0.5\textwidth}}
{Primal}:~~~~$
P_{\mathrm{OPT}}\coloneqq$ & {Dual}:~~~~$
D_{\mathrm{OPT}}\coloneqq$\\
$
\begin{aligned}
\refstepcounter{equation} \label{eq:config-lp}
\max_{\boldsymbol x\ge 0}\quad
& \sum_{i\in[n]}\sum_{S\in\mathcal S_i}
    \left(
        \sum_{j\in S} r_{ij}d_{ij}
    \right)x(i,S) \nonumber \\
\text{s.t.}\quad
& \sum_{i\in[n]}\sum_{\substack{S\in\mathcal S_i:\\ j\in S}}
    x(i,S)
    \le 1,
     ~~~\forall j\in[m], &\text{(\ref{eq:config-lp}a)} \\
& \sum_{S\in\mathcal S_i} x(i,S)
    \le c_i,~~~~~~~~~
     \forall i\in[n]. & \text{(\ref{eq:config-lp}b)} 
\end{aligned}
$
&
$
\begin{aligned}
\min_{\boldsymbol\lambda,\boldsymbol\theta\ge 0}\quad
& \sum_{j\in[m]}\lambda_j
  +
  \sum_{i\in[n]} c_i\theta_i  \\
\text{s.t.}\quad
& \sum_{j\in S}\lambda_j+\theta_i
    \ge
    \sum_{j\in S} r_{ij}d_{ij},\\
    &\qquad\qquad\qquad\qquad \forall i\in[n],\; S\in\mathcal S_i .
\end{aligned}
$
\end{tabular}

In the primal program, the variable $x(i,S)$ represents the weight assigned to configuration $S$ on server $i$. In an integral solution, $x(i,S)$ can be interpreted as the number of capacity units of server $i$ that follow schedule $S$. The objective maximizes the total reward generated by the selected configurations. Constraint~(\ref{eq:config-lp}a) ensures that each job is assigned at most once across all servers and configurations, while Constraint~(\ref{eq:config-lp}b) ensures that no more than $c_i$ capacity units of server $i$ are used. The next claim states that this configuration LP is a valid relaxation of the offline benchmark; its proof is deferred to Section~\ref{se:app_pf_claim_benchmark}.
\begin{claim}\label{claim:benchmark} 
    The optimum of the configuration LP upper bounds the offline benchmark: 
    $
        P_{\mathrm{OPT}}\ge \operatorname{OPT}$. 
\end{claim}

\section{Our Algorithms and the Main Theorems}
\label{sec:algorithm_main_theorem}

\subsection{The BALANCE Algorithm Framework}\label{sec:algorithm-balance-framework-challenge}

Our work follows the BALANCE framework, which has been widely used in prior works that studied online bipartite matching with adversarial arrivals \citep{mehta2007adwords,golrezaei2014real,goyal2025asymptotically,huang2024online,feng2025online}. As specified in Line~\ref{line:alg-balance-i-star} of Algorithm~\ref{alg:BALANCE}, upon the arrival of job $j$, a BALANCE-type policy selects a compatible server $i$ that
maximizes the job's immediate reward $r_{ij}d_{ij}$ net of an assignment-specific loss $\mathcal{L}_{ij}$,
provided that the largest adjusted reward $(r_{ij}d_{ij} - \mathcal{L}_{ij})$ is positive. Otherwise, job $j$ is
left unassigned.

\LinesNumbered
\begin{algorithm}[t]
\caption{The BALANCE Algorithm Framework}\label{alg:BALANCE}
\For{each job $j = 1$ \KwTo $m$}{
    observe the arrival of job $j$ with type $\bigl(\mathcal N(j),\{r_{ij},d_{ij}\}_{i\in \mathcal N(j)}\bigr)$\;
    \If{
    $\displaystyle
    \max_{i\in \mathcal N(j)}
    \left\{
        r_{ij}d_{ij}
        -
        \mathcal{L}_{ij}
    \right\}
    >0$
    }{
        choose $\displaystyle
        i_j^*
        \leftarrow \underset{i\in \mathcal N(j)}{\arg\max}
        \left\{
            r_{ij}d_{ij}
            -
            \mathcal{L}_{ij}
        \right\}$ and 
        assign job $j$ to server $i_j^*$\; \label{line:alg-balance-i-star}
    }
    \Else{
        leave job $j$ unassigned\;
    }
}
\end{algorithm}

The key of designing a BALANCE-type policy is to construct the assignment-specific loss $\mathcal{L}_{ij}$. In standard non-reusable models with stationary server-side rewards (i.e., every assignment to server $i$ yields reward $r_i$), the loss often takes the form
\begin{align}\label{eq:loss-ordinary-balance}
    \mathcal{L}_{ij}=r_i(e^{1-\alpha_{i,j}}-1)/(e-1),
\end{align}
where $\alpha_{i,j}$ denotes the fraction of server $i$'s capacity remaining when job $j$ arrives~\citep{mehta2007adwords,golrezaei2014real,huang2024online}. At a high level, this loss admits an opportunity-cost interpretation: assigning job $j$ to server $i$ consumes one unit of server $i$'s capacity and may prevent that unit from serving a future job that it could otherwise accommodate. We say that a future assignment is \emph{blocked} by assignment $(i,j)$ , and results a \emph{blocking loss}. In
Eq.~\eqref{eq:loss-ordinary-balance}, the blocking loss is represented by the
reward $r_i$ of the potentially foregone assignment, multiplied by a congestion-dependent discount factor. This factor decreases with the
remaining-capacity fraction: when more capacity remains, a capacity shortage, and hence a blocking event, is less likely, so the potential future reward is discounted more heavily.

\medskip
\noindent\underline{\bf Challenges in extending BALANCE to our setting.} Extending the BALANCE framework to reusable server capacity and non-stationary rewards raises two challenges.

\emph{First, reusability turns a single blocking loss into an aggregate of blocking losses along an unknown reuse schedule.} In a non-reusable model, if the platform refrains from making assignment $(i,j)$, the preserved capacity unit can serve at most one future job. Thus, making assignment $(i,j)$ can block at most one future assignment. In our model, however, the preserved unit may be \emph{reused} by several future
jobs in sequence. Assignment $(i,j)$ may therefore block some or all of the assignments along such a \emph{reuse schedule}, and $\mathcal{L}_{ij}$ must aggregate the resulting blocking losses. This aggregation is particularly challenging because the reuse schedule selected by the offline optimum is unknown to the online policy.

\emph{Second, each blocking loss becomes time dependent.} In a non-reusable model with stationary server-side rewards, once $i$ and $j$
are fixed, the blocking loss in Eq.~\eqref{eq:loss-ordinary-balance} is
determined by a fixed reward scale and remaining-capacity level and therefore does not depend on when the blocked future job arrives. In our model with non-stationary rewards, the reward rate of a future job may vary with its arrival time. Reusability also makes the projected available capacity time dependent, because jobs accepted
before $j$ may complete and release capacity at different future times. Consequently, constructing the blocking loss associated with a future assignment requires time-aware estimates of both its reward rate and the projected capacity state at its arrival time.

\subsection{The Time-Aware Schedule-Based BALANCE Algorithm}
\label{subsec:TS-BALANCE}

We now introduce our first algorithm, which addresses the two challenges described above.

\medskip
\noindent\underline{\bf Aggregating blocking losses over reuse schedules.} To address the first challenge, i.e., aggregating multiple blocking losses along an unknown reuse schedule, we define $\mathcal{L}_{ij}$ as the \emph{maximum} total blocking loss incurred by future jobs over all possible reuse schedules. In other words, the ``worst'' reuse schedule is the one that maximizes the cumulative blocking loss. To formalize this idea, we first define the set of possible reuse schedules.
\begin{definition}[Reuse Schedule] \label{def:feasible-dp-path}
Fix $T_1<T_2$. A \emph{reuse schedule} over the time interval $[T_1,T_2)$ is a finite ordered sequence $\boldsymbol{\tau}= (\tau_1,\ldots,\tau_s)$ such that
$T_1\le\tau_1<\cdots<\tau_s<T_2$, and $\tau_{z+1}-\tau_z\ge1$ holds for every $z=1,\ldots,s-1$.
By convention, the empty sequence is also a valid reuse schedule We denote by $\mathfrak{S}[T_1,T_2)$ the set of all reuse schedules over $[T_1,T_2)$.
\end{definition}
In words, a reuse schedule, as defined in  Definition~\ref{def:feasible-dp-path}, represents a possible sequence of starting times at which a single unit of a server's capacity is used to serve future jobs. In particular, the separation condition $\tau_{z+1}-\tau_z\ge1$ follows from the lower bound $d_{ih}\ge 1$ in Assumption~\ref{assumption:local-bounded-heterogeneity}: once a capacity unit begins serving a job at time $\tau_z$, it cannot be reused before time $\tau_z+1$.

For a candidate assignment $(i,j)$, we define its associated loss $\mathcal{L}_{ij}$ as the total blocking losses incurred by future jobs under the worst reuse schedule in $\mathfrak{S}[t_j, t_j + d_{ij})$. The interval \([t_j,t_j+d_{ij})\) naturally corresponds to the processing period of job \(j\) on server \(i\). More generally, we may define the blocking loss over any subinterval of this processing period. In the definition below, \(L_{i,j,\tau}\) denotes the blocking loss incurred when assigning job \(j\) to server \(i\) prevents a future job arriving at time \(\tau\) from being assigned to server \(i\). The precise form of \(L_{i,j,\tau}\) will be specified shortly.
\begin{definition}[Time-Aware Schedule-Based Loss]
\label{df:dp_price}
Fix a compatible server-job pair $(i,j)\in E$, a time interval $[T_1, T_2) \subseteq [t_j, t_j + d_{ij})$, and a collection of nonnegative blocking losses $L= \{L_{i,j,\tau}\}$. The \emph{time-aware schedule-based loss} over \([T_1,T_2)\) is defined as
\[
 \mathcal{L}_{ij}^{\operatorname{TS}}([T_1,T_2);L)
 \coloneqq
 \max_{\boldsymbol{\tau}\in\mathfrak S[T_1,T_2)}
 \left\{\sum_{\tau\in\boldsymbol{\tau}}L_{i,j,\tau}
 \right\}.
\]
\end{definition}

\noindent\underline{\bf Designing time-aware blocking losses.} We now specify the precise form of the blocking loss \(L_{i,j,\tau}\), thereby addressing the second challenge described in the preceding subsection. Our construction has the same high-level structure as that in Eq.~\eqref{eq:loss-ordinary-balance}, but both of its components are time-aware:
\begin{align}\label{eq:time-aware-blocking-loss}
L_{i,j,\tau} \coloneqq \hat r_{i,j,\tau}\,\Psi\!\left(\alpha_{i,j-1\to \tau}\right),
\end{align}
where $\hat r_{i,j,t}$ is a time-aware estimate of the reward associated with a future job arriving at time $\tau$, and $\alpha_{i,j-1\to \tau}$ is a time-aware estimate of the fraction of server $i$'s capacity available at time $\tau$, based on the system state before the arrival of job $j$. We use the exponential discount function introduced by \citet{feng2025online}:
\[
\Psi(x)\coloneqq \eta\bigl(\beta^{1-x}-1\bigr), \qquad x \in [0,1],
\]
where \(\eta>0\) and \(\beta>1\) are tuning parameters to be calibrated in the performance analysis. It is straightforward to verify that \(\Psi\) is nonnegative and decreasing. Thus, a larger estimated available-capacity fraction results in a smaller blocking loss.

It remains to define the time-aware reward estimate $\hat r_{i,j,t}$ and capacity estimate $\alpha_{i,j-1\to \tau}$ in Eq.~\eqref{eq:time-aware-blocking-loss}. We first define the time-aware reward estimate.

\begin{definition}[Time-Aware Reward Estimate]\label{def:future_reward_estimator}
Fix a compatible server-job pair $(i,j)\in E$. For each $t\in[t_j,t_j+D)$, define \[
\hat r_{i,j,t} \coloneqq \min_{k\in\mathcal K_{i,j}(t)} r_{ik}, \qquad \text{where}\qquad 
    \mathcal K_{i,j}(t)\coloneqq \left\{ k\le j:i\in N(k),\; t_k\ge t-D\right\}.
\]
At a high level, the set $\mathcal K_{i,j}(t) \supseteq \{j\}$ collects the jobs revealed through job $j$ whose reward rates on server $i$ remain locally informative about a job arriving at time $t$ (based on Eq.~\eqref{eq:local-reward-ratio} in Assumption~\ref{assumption:local-bounded-heterogeneity}). 

According to Definition~\ref{def:future_reward_estimator}, if a future job $h$ compatible with server $i$ arrives at time $t_h=t$, then Assumption~\ref{assumption:local-bounded-heterogeneity} implies $r_{ih}\le \delta\hat r_{i,j,t}$. Moreover, the adversary may choose a job \(h\) satisfying \(r_{ih}=\delta\hat r_{i,j,t}\) without violating the assumption. Therefore, \(\hat r_{i,j,t}\) provides a history-based estimate of the future reward rate up to the factor \(\delta\), and the resulting upper bound is tight under adversarial arrivals.
\end{definition}

We finally define the time-aware capacity estimate, which is also used in \cite{feng2025online}.
\begin{definition}[Time-Aware Capacity Estimate]
\label{df:available_capa}
Fix a server $i$, a job $k$, and a future time $\tau\ge t_k$. The estimated fraction of server $i$'s capacity available at time $\tau$, based on the assignment decisions through job $k$, is defined as
\[
    \alpha_{i,k\to \tau}\coloneqq
    1-\frac{1}{c_i}
    \sum_{h:\,h\le k,\ \mathcal N(h)\ni i}
    \mathbb{I}[ h~\text{is assigned to}~i  ~\text{and}~  t_h+d_{ih}>\tau].
\]
\end{definition}
In words, \(\alpha_{i,k\to\tau}\) is the projected fraction of server \(i\)'s capacity available at time \(\tau\), computed solely from the assignment decisions made through job \(k\), under the assumption that no additional jobs are assigned to server \(i\) after job \(k\). Additionally, define $\alpha_{i,0\to t} \coloneqq 1$ for any non-negative $t$.

\medskip
\noindent\underline{\bf The time-aware schedule-based BALANCE algorithm.} We now formally define our time-aware schedule-based BALANCE algorithm (TS-BAL for short), as follows.

\LinesNotNumbered
\begin{algorithm}[H]
\caption{The Time-aware Schedule-based BALance Algorithm  (TS-BAL)}\label{alg:main_d_integer_accurate}
Instantiate the BALANCE framework (Algorithm~\ref{alg:BALANCE}) with
\[
\mathcal{L}_{ij} = \mathcal{L}_{ij}^{\operatorname{TS}}\coloneqq \mathcal{L}_{ij}^{\operatorname{TS}}([t_j,t_j+d_{ij});L)=\max_{\boldsymbol{\tau}\in\mathfrak S[t_j,t_j+d_{ij})}
 \left\{\sum_{\tau\in\boldsymbol{\tau}}\hat r_{i,j,\tau}\,\Psi\!\left(\alpha_{i,j-1\to \tau}\right)
 \right\} .
\]
\end{algorithm}

\begin{remark}\label{rem:oc_dp}
The quantity $\mathcal{L}^{\operatorname{TS}}_{ij}$ can be computed efficiently. Observe that \(\hat r_{i,j,\cdot}\) and \(\alpha_{i,j\to\cdot}\) change at only finitely many points over the interval \([t_j,t_j+d_{ij})\). Let \({\tau_1,\ldots,\tau_q}\) denote all points in this interval at which either quantity changes. We define the dynamic-programming state space as \[
        S^{\operatorname{DP}}_{ij}=\{\{t_j+k:k\in\mathbb N\}\cup\{\tau_1+k:k\in\mathbb{N}\}\cup\dots\cup\{\tau_q+k:k\in\mathbb{N}\}\}\cap [t_j,t_j+d_{ij}) .
    \]
Using this state space, \(\mathcal{L}^{\operatorname{TS}}_{ij}\) can be computed by the following dynamic program:
\[
V(t) = \max\left\{0,\max_{\substack{\tilde\tau\in S^{\operatorname{DP}}_{ij},\
\tilde\tau\ge t}}
\left\{
V(\tilde\tau+1)
+
\hat r_{i,j,\tilde\tau}
\Psi\left(\alpha_{i,j-1\to\tilde\tau}\right)
\right\}\right\},
\qquad
t\in S^{\operatorname{DP}}_{ij},
\]
with terminal condition \(V(t)=0\) for \(t\ge t_j+d_{ij}\). The desired time-aware schedule-based loss is then given by
$
\mathcal{L}^{\operatorname{TS}}_{ij}=V(t_j)$.
\end{remark}

\begin{remark}\label{rem:comparison_feng}
\citet{feng2025online} study a related online matching problem with reusable servers and non-stationary rewards. However, they impose a global reward bound, rather than the locally bounded reward condition in Assumption~\ref{assumption:local-bounded-heterogeneity}. In particular, their Forward-Looking BALANCE (FLB) algorithm assumes that \(r_{ij}\in[1,R]\) for every \((i,j)\in E\) and uses the assignment-specific loss
\(
\mathcal{L}^{\mathrm{FLB}}_{ij}
=
\sum_{\tau\in\{t_j+k:k\in\mathbb N,\ k<d_{ij}\}}
\Psi\!\left(\alpha_{i,j-1\to\tau}\right)
\).
Compared with our time-aware schedule-based loss (TS loss), the FLB loss can be viewed as a special case in two respects. First, the FLB loss does not incorporate a time-aware reward estimate. Under their globally bounded reward assumption, the uniform lower bound \(r_{ij}\ge 1\) provides a natural reward estimate, which corresponds to setting \(\hat r_{i,j,\tau}\equiv 1\) in our TS-loss construction. Second, the FLB loss evaluates the blocking loss along the specific reuse schedule
\(
\boldsymbol{\tau}
=
\{t_j+k:k\in\mathbb N,\ k<d_{ij}\}
\),
whereas our TS loss maximizes over all feasible reuse schedules and therefore accounts for the worst possible sequence of reuse times.
We further emphasize that maximizing over all feasible reuse schedules is essential for achieving an \(\mathcal{O}(\ln(\delta D))\) competitive ratio (as proved in Theorem~\ref{tm:main_theorem}). Consider instead the assignment-specific loss
\(
\mathcal{L}_{ij}
=
\sum_{\tau\in\{t_j+k:k\in\mathbb N,\ k<d_{ij}\}}
\hat r_{i,j,\tau}
\Psi\!\left(\alpha_{i,j-1\to\tau}\right)
\),
which incorporates the time-aware reward estimate while retaining the specific reuse schedule used in the FLB loss. As shown in Section~\ref{app:failed_candidates}, the resulting policy can have a competitive ratio of \(\Omega(\delta)\).
\end{remark}

\medskip
\noindent\underline{\bf Competitive ratio guarantee.} The following guarantee for Algorithm~\ref{alg:main_d_integer_accurate} is established in Section~\ref{se:pf_main_tm}, which also derives the parameter choices $\eta^*$ and $\beta^*$ that attain the stated competitive ratio.

\begin{theorem}[Competitive ratio of the TS-BAL algorithm]
\label{tm:main_theorem}
Suppose that Assumption~\ref{assumption:local-bounded-heterogeneity} holds. Then there exist parameters
$(\eta^*,\beta^*)$ such that, as $c_{\min}:=\min_i c_i\to\infty$, 
Algorithm~\ref{alg:main_d_integer_accurate} achieves an asymptotic competitive ratio of 
$2\ln(\delta D)+2\ln\ln(\delta\vee D)+\mathcal O(1)$.
\end{theorem}

To complement the upper bound in Theorem~\ref{tm:main_theorem}, the following result establishes a competitive-ratio lower bound for any online algorithm.
\begin{theorem}[Competitive ratio lower bound]\label{tm:lb_main}
Under Assumption~\ref{assumption:local-bounded-heterogeneity}, every online algorithm has a competitive ratio of at least
$H_D+\ln\delta \geq \ln (\delta D)$, where $ H_D\coloneqq\sum_{k=1}^{\lfloor D\rfloor} (1/k)$. This lower bound holds even when fractional assignments are permitted and processing durations \(\{d_{ij}\}\) are integers.
\end{theorem}

Comparing Theorem~\ref{tm:main_theorem} and Theorem~\ref{tm:lb_main}, we conclude that the competitive ratio achieved by Algorithm~\ref{alg:main_d_integer_accurate}  is asymptotically optimal up to a factor of $2$.

\subsection{Tight Competitive Ratio via Greedy Relaxation Loss}\label{subsec:GR-BAL}

In this subsection, we present a refined algorithm that removes the factor \(2\) from the leading term of the competitive-ratio bound in Theorem~\ref{tm:main_theorem}. The algorithm continues to follow the BALANCE framework, but its assignment-specific loss \(\mathcal{L}_{ij}\) is constructed through a greedy relaxation of the time-aware schedule-based (TS) loss introduced in Section~\ref{subsec:TS-BALANCE}. Specifically, rather than optimizing over feasible reuse schedules as in Definition~\ref{df:dp_price}, we partition \([T_1,T_2)\) using a unit-length grid and independently select the reuse time with the largest blocking loss from each subinterval. Although the resulting collection of reuse times need not constitute a feasible reuse schedule, its total blocking loss provides an upper bound on the TS loss. This construction yields what we call the \emph{greedy relaxation loss}.

For ease of exposition, we restrict attention in this subsection to integer-valued processing durations. The extension to real-valued processing durations follows the approach of \citet{feng2025online}; the corresponding modification of the algorithm is deferred to Section~\ref{app:real-duration}, where Theorem~\ref{thm:real-duration} establishes that the modified algorithm, $\gamma$GR-BAL (Algorithm~\ref{alg:main_d_realvalue}), achieves an asymptotic competitive ratio of $\ln(\delta D) + 2\ln \ln (\delta \vee D) + \mathcal{O}(1) $ as $c_{\min} \to \infty$.

\medskip
\noindent\underline{\bf The greedy relaxation loss.} We now derive an upper bound on the TS loss via greedy relaxation. For any time interval $[T_1, T_2)$, We define the unit-length grid 
\[
\mathcal{T}[T_1,T_2)\coloneqq \{T_1,T_1+1,\dots,T_1+\lceil T_2-T_1-1\rceil\}.
\]
Then $\{[\tau, \tau + 1\wedge T_2)\}_{\tau \in \mathcal{T}[T_1, T_2)}$ forms a partition of $[T_1, T_2)$. Fix a candidate assignment $(i,j)$, a time interval $[T_1, T_2) \subseteq [t_j, t_j + d_{ij})$, and a reuse schedule $\boldsymbol{\tau} \in \mathfrak S[T_1, T_2)$. For any two distinct elements $\tau', \tau'' \in \boldsymbol{\tau}$, Definition~\ref{def:feasible-dp-path} implies that $|\tau' - \tau''| \geq 1$. Consequently, $\tau'$ and $\tau''$ belong to distinct intervals in the above partition. Each partition interval therefore contains at most one element of \(\boldsymbol{\tau}\), and hence $\sum_{\tau \in \boldsymbol{\tau}} L_{i,j,\tau} \leq \sum_{\tau \in \mathcal{T}[T_1, T_2)} \max_{t\in[\tau,\tau+1\wedge T_2)} L_{i,j,t}$, Taking the maximum over all reuse schedules yields
\begin{align}
\mathcal{L}_{ij}^{\operatorname{TS}}([T_1,T_2);L) \leq \sum_{\tau \in \mathcal{T}[T_1, T_2)} \max_{t\in[\tau,\tau+1\wedge T_2)} L_{i,j,t}. \label{eq:blocking-loss-upper-bound-1}
\end{align}
On the right-hand side of Eq~\eqref{eq:blocking-loss-upper-bound-1}, we greedily select the largest blocking loss from each subinterval in the partition of $[T_1, T_2)$. The sum of these selected losses provides an upper bound, or relaxation, of the TS loss.

To derive the final form of the greedy relaxation loss, we further relax the right-hand side of Eq~\eqref{eq:blocking-loss-upper-bound-1}. For any time \(t \in [t_j, t_j + d_{ij})\), let \(\tilde{\boldsymbol r}_{i,j,t}\) denote the multiset of reward rates of the jobs that were assigned to server \(i\) before the arrival of job \(j\) and remain active at time \(t\). We use \(\tilde r_{i,j,t}^{(\ell)}\) to denote the \(\ell\)-th largest element of this multiset, and define \(\tilde r_{i,j,t}^{(0)} \coloneqq +\infty\) for convenience. Fix any $t \in [\tau, \tau+1\wedge T_2)$, and write $\alpha_{i,j-1\to t} = 1 - \ell/{c_i}$, where $\ell = |\tilde{\boldsymbol r}_{i,j,t}|$ is the number of jobs that were assigned to server \(i\) before the arrival of job \(j\) and remain active at time \(t\). We then have
\begin{align*}
L_{i,j,t} = \hat{r}_{i,j,t} \Psi(\alpha_{i,j-1\to t}) \leq \left(\tilde{r}_{i,j,t}^{(\ell)} \wedge r_{ij}\right) \Psi(1 - {\ell}/{c_i}) \leq \left(\tilde{r}_{i,j,\tau}^{(\ell)} \wedge r_{ij}\right)  \Psi(1 - {\ell}/{c_i}) .
\end{align*}
The last inequality follows because every job that arrives before $\tau$ and remains active at time \(t\ge \tau\) is also active at time \(\tau\). Consequently,
\begin{align}
\max_{t\in[\tau,\tau+1\wedge T_2)} L_{i,j,t} \leq  \max_{\ell \in \{0\} \cup [|\tilde{\boldsymbol{r}}_{i,j,\tau}|]} \left(\tilde{r}_{i,j,\tau}^{(\ell)} \wedge r_{ij}\right) \Psi(1 - \ell / c_i) .  \label{eq:blocking-loss-upper-bound-2}
\end{align}
Combining Eq.~\eqref{eq:blocking-loss-upper-bound-1} and Eq.~\eqref{eq:blocking-loss-upper-bound-2}, we derive
\begin{align}
\mathcal{L}_{ij}^{\operatorname{TS}}([T_1,T_2);L) \leq \mathcal{L}_{ij}^{\operatorname{GR}}[T_1,T_2)  \coloneqq \sum_{\tau \in \mathcal{T}[T_1, T_2)} \max_{\ell \in \{0\} \cup [|\tilde{\boldsymbol{r}}_{i,j,\tau}|]} \left(\tilde{r}_{i,j,\tau}^{(\ell)} \wedge r_{ij}\right) \Psi(1 - \ell / c_i) . \label{eq:def-greedy-relaxation-loss}
\end{align}
We refer to the quantity on the right-hand side of Eq.~\eqref{eq:def-greedy-relaxation-loss} as the \emph{greedy relaxation loss}.

\medskip
\noindent\underline{\bf The greedy-relaxation BALANCE algorithm and its competitive ratio guarantee.} We formally define our refined greedy-relaxation BALANCE algorithm (GR-BAL for short), as follows.

\LinesNotNumbered
\begin{algorithm}[H]
\caption{The Greedy-Relaxation BALance Algorithm  (GR-BAL)}\label{alg:main_d_integer_active}
Instantiate the BALANCE framework (Algorithm~\ref{alg:BALANCE}) with
\[
\mathcal{L}_{ij} = \mathcal{L}_{ij}^{\operatorname{GR}}\coloneqq L_{ij}^{\operatorname{GR}}
 [t_j,t_j+d_{ij}) .
\]
\end{algorithm}
The following theorem presenting the guarantee for Algorithm~\ref{alg:main_d_integer_active} is established in Section~\ref{se:pf_main_tm_refine}, which also derives the parameter choices $\eta^*$ and $\beta^*$ that attain the stated competitive ratio.

\begin{theorem}[Competitive ratio of the GR-BAL algorithm under integer-valued durations]
\label{tm:main_theorem_refine}
Suppose Assumption~\ref{assumption:local-bounded-heterogeneity} holds and all durations are integer-valued. There exist parameters
$(\eta^*,\beta^*)$ such that, as $c_{\min}:=\min_i c_i\to\infty$,
Algorithm~\ref{alg:main_d_integer_active} achieves an asymptotic competitive ratio of
\(
\ln(\delta D)+\ln\ln(\delta\vee D)+\mathcal \mathcal{O}(1) \).
\end{theorem}

Compared with Theorem~\ref{tm:main_theorem}, the key analytical reason why Theorem~\ref{tm:main_theorem_refine}  achieves an asymptotically tight competitive-ratio bound is a stronger structural subadditivity property. The TS-loss increments associated with Algorithm~\ref{alg:main_d_integer_accurate}, formally defined in Definition~\ref{def:pathoc-increment}, can only be shown to satisfy a weak subadditivity inequality over disjoint time intervals; see Lemma~\ref{le:pf_dp_additive} in Section~\ref{se:pf_main_tm}. This weaker inequality introduces the factor \(2\). By contrast, the greedy-relaxation-loss increments are fully subadditive over disjoint time intervals; see Lemma~\ref{le:pf_dp_additive_refine} in Section~\ref{se:pf_main_tm_refine}. This stronger property is the key to removing the factor \(2\) from the leading term of the competitive-ratio bound. Nevertheless, it remains unclear whether a more refined analysis could establish an asymptotically tight competitive ratio for Algorithm~\ref{alg:main_d_integer_accurate}. We leave this as an interesting open question.

\section{Proof of Theorem~\ref{tm:main_theorem}}\label{se:pf_main_tm}
The proof is divided into three parts. First, in Section~\ref{se:pf_capa_fea}, we explore the capacity availability of Algorithm~\ref{alg:main_d_integer_accurate}. Specifically, we define an admissible parameter set and illustrate that when $\beta,\eta$ lie in this set, for every $i$ and $j$, the remaining capacity of server $i$ is larger than $0$ if Algorithm~\ref{alg:main_d_integer_accurate} assigns $i$ to $j$. The main result of this part is Proposition~\ref{prop:capa_fea}. Second, in Section~\ref{se:pf_value_LP}, we give a competitive ratio of Algorithm~\ref{alg:main_d_integer_accurate} for any admissible parameter pair $(\beta,\eta)$ through a primal-dual analysis of the configuration LP. The derived competitive ratio relies on $(\beta,\eta)$, and the main result of this part is Proposition~\ref{prop:pf_value_LP}.
Finally, in Section~\ref{se:pf_choice_beta_eta}, we combine the results in Section~\ref{se:pf_capa_fea} and Section~\ref{se:pf_value_LP}. Specifically, we select explicit admissible parameters and substitutes them into the bound derived in Proposition~\ref{prop:pf_value_LP}, thereby completing the proof.

\subsection{Capacity Feasibility of Algorithm~\ref{alg:main_d_integer_accurate}}\label{se:pf_capa_fea}

Recall that the first step in proving Theorem 1 is to establish the capacity feasibility of Algorithm~\ref{alg:main_d_integer_accurate}. To this end, Definition~\ref{df:capa_fea} introduces an admissible parameter set $\Theta$, and Proposition~\ref{prop:capa_fea} subsequently proves that every parameter pair $(\beta,\eta)\in\Theta$ guarantees capacity feasibility.

\begin{definition}[Admissible Parameter Set]\label{df:capa_fea}
For the function $\Psi(x)=\eta(\beta^{1-x}-1)$, define
\[
    \Theta \coloneqq
    \left\{
    (\beta,\eta):
    \delta D\le \eta\left(\beta-1\right),\;
    \beta>1,\;
    \eta>0
    \right\}.
\]
\end{definition}

We next establish that this parameter condition is sufficient for capacity feasibility.

\begin{proposition}[Feasibility of Algorithm~\ref{alg:main_d_integer_accurate}]\label{prop:capa_fea}
Fix $\eta$ and $\beta$. If $(\beta,\eta)\in\Theta$, then Algorithm~\ref{alg:main_d_integer_accurate} never violates the server-capacity constraints. Equivalently, whenever the algorithm assigns job $j$ to server $i$, server $i$ has strictly positive remaining capacity immediately before the assignment.
\end{proposition}

\proof{Proof.}
We first state a key inequality implied by the assignment rule and used in the feasibility argument. Specifically, Lemma~\ref{le:capa_fea_lb_integral} compares the duration scale $\delta d_{ij}$ with the scarcity cost at the arrival time $t_j$ when Algorithm~\ref{alg:main_d_integer_accurate} makes an assignment. Its proof is deferred to Section~\ref{sec:app_pf_le_capa_fea_integral}.

\begin{lemma}\label{le:capa_fea_lb_integral}
    When Algorithm~\ref{alg:main_d_integer_accurate} assigns server $i$ to job $j$, we have that $\delta d_{ij}> \Psi\!(\alpha_{i,j-1\to t_j})$.
\end{lemma}

We prove Proposition~\ref{prop:capa_fea} by induction over arrivals. Suppose the process is feasible before job $j$ arrives. It remains to show that, when $(\beta,\eta)\in\Theta$, Algorithm~\ref{alg:main_d_integer_accurate} can not assign job $j$ to a server $i$ with no available capacity at time $t_j$, that is, with $\alpha_{i,j-1\to t_j}=0$. Assume, for contradiction, that such an assignment is made. By Lemma~\ref{le:capa_fea_lb_integral}, we have that \(
    \delta D\ge\delta d_{ij}>  \Psi\!\left(\alpha_{i,j-1\to t_j}\right)=\Psi\!\left(0\right)=\eta\left(\beta-1\right) \),
which contradicts with the choice that $(\beta,\eta)\in\Theta$. Therefore, the induction closes, and the proposition follows.
\Halmos\endproof

\subsection{A Primal-Dual Bound for Fixed Admissible Parameters}\label{se:pf_value_LP}

Having established capacity feasibility for every $(\beta,\eta)\in\Theta$, we now bound the competitive ratio of Algorithm~\ref{alg:main_d_integer_accurate} as a function of these parameters.
\begin{proposition}\label{prop:pf_value_LP}
    Suppose that $(\beta,\eta)\in\Theta$. Then the competitive ratio of Algorithm~\ref{alg:main_d_integer_accurate} is upper bounded by $[1+2(1+2\eta)\beta^{1/c_{\min}}\ln \beta]$.
\end{proposition}

Fix $(\beta,\eta)\in\Theta$. By Proposition~\ref{prop:capa_fea}, Algorithm~\ref{alg:main_d_integer_accurate} is capacity feasible, and hence $\alpha_{i,j\to t}\in[0,1]$ throughout the process. With this fact, it remains to prove the claimed competitive-ratio bound. To this end, we construct a feasible dual solution to the configuration LP and compare its objective value with the reward collected by Algorithm~\ref{alg:main_d_integer_accurate}. Specifically, we carry out this primal-dual analysis in three steps. First, we construct dual variables $\{\hat\lambda_j,\hat\theta_i\}$ from the trajectory of Algorithm~\ref{alg:main_d_integer_accurate}. Second, in Proposition~\ref{prop:pf_value_LP_relate_objectivedual_to_alg}, we compare the dual objective value induced by $\{\hat\lambda_j,\hat\theta_i\}$ to the total reward collected by the algorithm. Finally, in Proposition~\ref{prop:pf_value_LP_dual_fea}, we prove that the constructed dual variables $\{\hat\lambda_j,\hat\theta_i\}$ are feasible; weak duality then yields the desired competitive-ratio bound.

\noindent\underline{\bf Construction of dual variables.}
We first construct dual variables
$\{\hat{\lambda}_j,\hat{\theta}_i\}$ from the trajectory of
Algorithm~\ref{alg:main_d_integer_accurate}. For each job $j$, set the job-side variable to be 
\[
    \hat\lambda_j\coloneqq\max\left(0,\max_{i\in \mathcal N(j)}
    \left\{
        r_{ij}d_{ij}
        -
       \mathcal{L}_{ij}^{\operatorname{TS}}\right\}\right).
\]
At a high level, $\hat\lambda_j$ is the positive part of the largest reduced reward available to job $j$.

We next construct the server-side variables $\hat \theta_i$. For each server $i$, let 
\[
\hat{\theta}_i:=\sum_{j=1}^{m}\hat\vartheta_{ij},
\]
where $\hat\vartheta_{ij}$ is the incremental component corresponding to the duration when the algorithm processes job $j$. To specify these incremental components, we first introduce the following definition.
\begin{definition}[Capacity-Induced Loss Increment]
\label{def:pathoc-increment}
Fix a compatible pair $(i,j)\in E$ and a time interval  $[T_1, T_2) \subseteq [t_j, t_j+D)$, define
\begin{align}\label{eq:pf_df_delta_relate_increment_1}
\Delta_{ij}\left[T_1,T_2\right)\coloneqq\max_{\boldsymbol{\tau}\in\mathfrak{S}[T_1,T_2)}\sum_{\tau\in\boldsymbol{\tau}} \hat r_{i,j,\tau}\left[\Psi(\alpha_{i,j\to\tau})-\Psi(\alpha_{i,j-1\to\tau})\right].
\end{align}
{Additionally, for convenience, we define $\Delta_{ij}[T_1,T_2) \coloneqq 0$ for $(i,j)\notin E$ and any time points $T_1<T_2$}. At a high level, $\Delta_{ij}$ captures the increase in the {schedule-based loss} caused by updating the projected-capacity trajectory from
$\alpha_{i,j-1\to\cdot}$ to $\alpha_{i,j\to\cdot}$, while holding the reward rate estimates fixed.
\end{definition}

We then define the incremental components $\hat\vartheta_{ij}$ based on the capacity-induced loss increment:
\[
    \hat\vartheta_{ij}\coloneqq 2\Delta_{ij}[t_j,t_j+d_{ij})
        \cdot \mathbb{I}[\text{Algorithm~\ref{alg:main_d_integer_accurate} assigns job $j$ to server $i$}]=2\Delta_{ij}[t_j,t_j+d_{ij}),
\]
where the last equality due to the fact that $\alpha_{i,j-1\to\cdot}=\alpha_{i,j\to\cdot}$ and $\Delta_{ij}=0$ when Algorithm~\ref{alg:main_d_integer_accurate} does not assign job $j$ to server $i$. Since $\alpha_{i,j\to t}\le \alpha_{i,j-1\to t}$ and $\Psi$ is decreasing, every $\hat\vartheta_{ij}$ is nonnegative. Hence,
$\hat{\lambda}_j\ge 0$ and $\hat{\theta}_i\ge 0$ for all $j$ and $i$.
The resulting variables $\{\hat{\lambda}_j,\hat{\theta}_i\}$ form the candidate dual solution induced by the trajectory of Algorithm~\ref{alg:main_d_integer_accurate}.

\noindent\underline{\bf Objective bound and feasibility of $\{\hat\lambda_j,\hat\theta_i\}$.}
Let $\mathrm{ALG}$ denote the total reward collected by Algorithm~\ref{alg:main_d_integer_accurate}, and let $\mathrm{Dual}(\boldsymbol{\hat{\lambda}},\boldsymbol{\hat{\theta}})$ denote the dual objective evaluated at the constructed variables $\boldsymbol{\hat \lambda}=\{\hat \lambda_j\}_j$ and $\boldsymbol{\hat \theta}=\{\hat \theta_i\}_i$. The following Proposition~\ref{prop:pf_value_LP_relate_objectivedual_to_alg} bounds $\mathrm{Dual}(\boldsymbol{\hat{\lambda}},\boldsymbol{\hat{\theta}})$ in terms of $\mathrm{ALG}$, and Proposition~\ref{prop:pf_value_LP_dual_fea} establishes the feasibility for $\{\hat\lambda_j,\hat\theta_i\}$. Their proofs are provided in Sections~\ref{subse:pf_pf_value_LP_1} and Section~\ref{subse:pf_pf_value_LP_2}, respectively.
\begin{proposition}[Objective value]
\label{prop:pf_value_LP_relate_objectivedual_to_alg}
$\mathrm{Dual}(\boldsymbol{\hat{\lambda}},\boldsymbol{\hat{\theta}})
\le
\left[1+2(1+2\eta)\beta^{1/c_{\min}}\ln\beta\right]\mathrm{ALG}$.
\end{proposition}

\begin{proposition}[Dual feasibility]\label{prop:pf_value_LP_dual_fea}
The constructed dual variables $\{\hat\lambda_j,\hat\theta_i\}$ are feasible.
\end{proposition}

Combining Claim~\ref{claim:benchmark}, Propositions~\ref{prop:pf_value_LP_relate_objectivedual_to_alg} and ~\ref{prop:pf_value_LP_dual_fea}, we establish the desired competitive-ratio bound: \begin{align*}
    \left[1+2(1+2\eta)\beta^{1/c_{\min}}\ln\beta\right]\cdot \operatorname{ALG}\ge\mathrm{Dual}(\boldsymbol{\hat{\lambda}},\boldsymbol{\hat{\theta}})\ge P_{\mathrm{OPT}}\ge \operatorname{OPT}.
\end{align*}
This finishes the proof of Proposition~\ref{prop:pf_value_LP}.

\subsubsection{Proof of Proposition~\ref{prop:pf_value_LP_relate_objectivedual_to_alg}} \label{subse:pf_pf_value_LP_1}
To prove Proposition~\ref{prop:pf_value_LP_relate_objectivedual_to_alg}, we compare, arrival by arrival, the increase in the dual objective with the reward collected by Algorithm~\ref{alg:main_d_integer_accurate}. Specifically, for each job $j$, let $\mathrm{Dual}_j=\hat\lambda_j+\sum_i c_i\hat\vartheta_{ij}$ denote the dual-objective increment associated with processing arrival $j$, let $\operatorname{ALG}_j=\sum_{i\in \mathcal N(j)}r_{ij}d_{ij}\cdot\mathbb{I}[\text{Algorithm~\ref{alg:main_d_integer_accurate} assigns $j$ to $i$}]$ denote the increase of the total reward caused by the assignment of job $j$. Let $\Gamma=1+2(1+2\eta)\beta^{1/c_{\min}}\ln\beta$. Since $\operatorname{Dual}(\boldsymbol{\hat \lambda},\boldsymbol{\hat \theta})=\sum_j\operatorname{Dual}_j$ and $\operatorname{ALG}=\sum_j\operatorname{ALG}_j$, it is sufficient to show that, for every arriving job $j$, it holds that \begin{equation}\label{eq:pf_competitive_ratio_dual_alg}
    \operatorname{Dual}_j\le \Gamma\cdot \operatorname{ALG}_j.
\end{equation}

The rest of this subsection is devoted to the proof of Eq.~\eqref{eq:pf_competitive_ratio_dual_alg}. Note that if $j$ is rejected, we have $\operatorname{Dual}_j=\operatorname{ALG}_j=0$ and Eq.~\eqref{eq:pf_competitive_ratio_dual_alg} holds directly. Hence, it remains to consider the case in which Algorithm~\ref{alg:main_d_integer_accurate} assigns job $j$ to some server $i$. In this case, since it holds directly that $\hat\lambda_j\le r_{ij}d_{ij}$, the only nontrivial task is to control the server-side increment $\hat \vartheta_{ij}$, for which we follow a two-step argument. First, using the convexity of $\Psi$, we bound $\hat \vartheta_{ij}$ in terms of {$\mathcal{L}_{ij}^{\operatorname{TS}}$ and a simple residual term. Second, we use $\mathcal{L}_{ij}^{\operatorname{TS}}<r_{ij}d_{ij}$ to complete the per-arrival comparison}.

\noindent\underline{\bf Step 1: Bounding the server-side increment.} For every $\alpha\in[1/c_i,1]$, the convexity and monotonicity of $\Psi$ imply that \begin{align}\label{eq:pf_psi_increment_bound}
    \Psi\!\left(\alpha-\frac{1}{c_i}\right)-\Psi\!\left(\alpha\right)\le \left|\Psi^{\prime}\!\left(\alpha-\frac{1}{c_i}\right)\right|\frac{1}{c_i} =\beta^{1/c_i}\left[\Psi\!\left(\alpha\right)+\eta\right]\frac{\ln\beta}{c_i}.
\end{align}

Since job $j$ is assigned to server $i$, for every $\tau\in[t_j,t_j+d_{ij})$, we have that $\alpha_{i,j\to\tau}=\alpha_{i,j-1\to\tau}-\frac{1}{c_i}$. {For each $\tau\in[t_j,t_j+d_{ij})$, recall that $(\beta,\eta)\in\Theta$ implies that $\alpha_{i,j\to\tau}\ge 0$, which indicates that $\alpha_{i,j-1\to\tau}\ge\frac{1}{c_i}$ in this case}. Therefore, applying Eq.~\eqref{eq:pf_psi_increment_bound} to the definition of $\hat{\vartheta}_{ij}$ gives \begin{align}\label{eq:pf_value_LP_dual_increment_bound}
    \nonumber c_i\hat{\vartheta}_{ij}& = 2c_i\underset{\boldsymbol{\tau}\in\mathfrak{S}[t_j,t_j+d_{ij})}{\max}\left\{\sum_{\tau\in\boldsymbol{\tau}}\hat r_{i,j,\tau}\left[\Psi\!\left(\alpha_{i,j-1\to\tau}-\frac{1}{c_i}\right)-\Psi\!(\alpha_{i,j-1\to\tau})\right]\right\}\\
        \nonumber& \leq 2c_i\underset{\boldsymbol{\tau}\in\mathfrak{S}[t_j,t_j+d_{ij})}{\max}\left\{\sum_{\tau\in\boldsymbol{\tau}}\hat r_{i,j,\tau}\left[\beta^{1/c_i}\left(\Psi\!\left(\alpha_{i,j-1\to\tau}\right)+\eta\right)\frac{\ln\beta}{c_i}\right]\right\}\\
        & \le 2\beta^{1/c_i}\ln\beta\cdot\mathcal{L}_{ij}^{\operatorname{TS}}+2\eta\beta^{1/c_i}\ln\beta\underset{\boldsymbol{\tau}\in\mathfrak{S}[t_j,t_j+d_{ij})}{\max}\left\{\sum_{\tau\in\boldsymbol{\tau}}\hat r_{i,j,\tau}\right\}.
\end{align}

Then, we upper bound the residual term $\underset{\boldsymbol{\tau}\in\mathfrak{S}[t_j,t_j+d_{ij})}{\max}\left\{\sum_{\tau\in\boldsymbol{\tau}}\hat r_{i,j,\tau}\right\}$. By Definition~\ref{def:future_reward_estimator}, we have $\hat r_{i,j,\tau}\le r_{ij}$ for every $\tau\in[t_j,t_j+d_{ij})$. Moreover, every feasible path in this interval contains at most $\lceil d_{ij}\rceil<d_{ij}+1$ points. Since
$d_{ij}\ge 1$, we have that
\[
    \max_{\boldsymbol{\tau}\in\mathfrak{S}[t_j,t_j+d_{ij})}\sum_{\tau\in\boldsymbol{\tau}}\hat r_{i,j,\tau}
    \le r_{ij}\max_{\boldsymbol{\tau}\in\mathfrak{S}[t_j,t_j+d_{ij})}|\boldsymbol{\tau}|
    \le r_{ij}(d_{ij}+1)\le2r_{ij}d_{ij}.
\]
Substituting this bound into Eq.~\eqref{eq:pf_value_LP_dual_increment_bound} yields that \begin{equation}\label{eq:pf_vartheta_pathoc_bound}
    c_i\hat{\vartheta}_{ij} \le 2\beta^{1/c_i}\ln\beta
    \left(\mathcal{L}_{ij}^{\operatorname{TS}}+2\eta r_{ij}d_{ij}
\right).
\end{equation}

\noindent\underline{\bf Step 2: Completing the per-arrival bound.}
Next, we use Eq.~\eqref{eq:pf_vartheta_pathoc_bound} to prove Eq.~\eqref{eq:pf_competitive_ratio_dual_alg} by deriving an upper bound for $\mathcal{L}_{ij}^{\operatorname{TS}}$. Since job $j$ is assigned to server $i$, the assignment rule in Algorithm~\ref{alg:main_d_integer_accurate} implies that $r_{ij}d_{ij}\ge \mathcal{L}_{ij}^{\operatorname{TS}}$. Combining this inequality with Eq.~\eqref{eq:pf_vartheta_pathoc_bound}, we obtain that
\begin{align*}
    \operatorname{Dual}_j=\hat{\lambda}_j+c_i\hat{\vartheta}_{ij} &\le
    r_{ij}d_{ij}+2\beta^{1/c_i}\ln\beta\left(\mathcal{L}_{ij}^{\operatorname{TS}}+2\eta r_{ij}d_{ij}\right)\\
    &\le \left[1+2(1+2\eta)\beta^{1/c_i}\ln\beta\right]r_{ij}d_{ij}=\Gamma r_{ij}d_{ij}=\Gamma\cdot\operatorname{ALG}_j,
\end{align*}
where the last inequality follows from $c_i\ge c_{\min}$ and $\beta>1$. This proves Eq.~\eqref{eq:pf_competitive_ratio_dual_alg} holds for every job $j$, and therefore concludes the proof of Proposition~\ref{prop:pf_value_LP_relate_objectivedual_to_alg}. \Halmos

\subsubsection{Proof of Proposition~\ref{prop:pf_value_LP_dual_fea}} \label{subse:pf_pf_value_LP_2}
To establish the feasibility of $\{\hat\lambda_j,\hat\theta_i\}$, we only need to prove that, for every server $i$ and every $S\in\mathcal S_i$, we have \begin{equation}\label{eq:dual_fea_thetai_lb_total}
    \sum_{j\in S}\hat\lambda_j+\hat\theta_i\ge \sum_{j\in S} r_{ij}d_{ij}.
\end{equation}

We first present a sufficient condition for Eq.~\eqref{eq:dual_fea_thetai_lb_total}. By the definition of $\hat\lambda_j$, for every $j\in[m]$ and every $i\in \mathcal N(j)$, we have that $\hat\lambda_j\geq r_{ij}d_{ij}-\mathcal{L}_{ij}^{\operatorname{TS}}$.
Therefore, it suffices to prove that, for every server $i$ and every feasible configuration $S\in\mathcal S_i$, it holds that \begin{equation}\label{eq:dual_fea_thetai_lb}
    \hat\theta_i\geq \sum_{j\in S}\mathcal{L}_{ij}^{\operatorname{TS}}.
\end{equation}

Fix a server $i$ and a feasible configuration $S\in\mathcal S_i$. We now prove Eq.~\eqref{eq:dual_fea_thetai_lb} in two steps. First, for each $j\in S$, we charge $\mathcal{L}_{ij}^{\operatorname{TS}}$ to local increments contributing to $\hat{\theta}_i$. Second, we aggregate these charges and show that the summary of them does not exceed $\hat{\theta}_i$.

\noindent\underline{\bf Step 1: Local charging.}
We first associate each $\mathcal{L}_{ij}^{\operatorname{TS}}$ with loss increments generated by jobs that were previously assigned to server $i$ and remain active at time $t_j$. Formally, we define \[
    \mathcal C_i(j)\coloneqq \left\{ h:h<j,\; \text{ Algorithm~\ref{alg:main_d_integer_accurate} assigns }h \text{ to } i,\;t_j<t_h+d_{ih}\right\}.
\]
For each $h\in\mathcal C_i(j)$, we consider only the increment accumulated over the overlap between the service intervals of jobs $j$ and $h$. The following lemma shows that these local increments collectively cover $\mathcal{L}_{ij}^{\operatorname{TS}}$.

\begin{lemma}[Local charging]\label{le:pf_dp_charging}
    Fix a job $j\in S$. We have $\displaystyle
        \sum_{h\in \mathcal C_i(j)}\Delta_{ih}[t_j,\left(t_j+d_{ij}\right)\wedge \left(t_{h}+d_{ih}\right))\ge \mathcal{L}_{ij}^{\operatorname{TS}}$.
\end{lemma}
The proof of Lemma~\ref{le:pf_dp_charging} is provided in Section~\ref{se:app_pf_dp_charging}. Applying Lemma~\ref{le:pf_dp_charging} to every $j\in S$ and summing over $j\in S$ yields 
\begin{align}
    \sum_{j\in S}\mathcal{L}_{ij}^{\operatorname{TS}}&\le\sum_{j\in S}\sum_{h\in\mathcal C_i(j)}\Delta_{ih}\left[t_j,\left(t_j+d_{ij}\right)\wedge \left(t_{h}+d_{ih}\right)\right) 
    \le\sum_{h=1}^m\sum_{j\in S:\mathcal C_i(j)\ni h} \Delta_{ih}\left[t_j,\left(t_j+d_{ij}\right)\wedge \left(t_{h}+d_{ih}\right)\right). \label{eq:local-charging-sum}
\end{align}

\noindent\underline{\bf Step 2: Aggregation via weak sub-additivity.}
It remains to aggregate the loss increments over multiple time intervals on the
right-hand side of Eq.~\eqref{eq:local-charging-sum}. We show that,
for every $h\in[m]$, we have
\begin{equation}\label{eq:pf_dp_additive_111}
    \sum_{j\in S:\mathcal C_i(j)\ni h} \Delta_{ih}\left[t_j,\left(t_j+d_{ij}\right)\wedge \left(t_{h}+d_{ih}\right)\right)\le \hat\vartheta_{ih}. 
\end{equation}
Once this bound is established, Eq.~\eqref{eq:local-charging-sum} and
the definition of $\hat\theta_i$ imply
\begin{align*}
    \hat \theta_i=\sum_{h=1}^m\hat\vartheta_{ih}
    \ge \sum_{h=1}^m\sum_{j\in S:\mathcal C_i(j)\ni h}\Delta_{ih}[t_j,\left(t_j+d_{ij}\right)\wedge \left(t_{h}+d_{ih}\right))
    \ge \sum_{j\in S}\mathcal{L}_{ij}^{\operatorname{TS}}.
\end{align*}
This proves Eq.~\eqref{eq:dual_fea_thetai_lb}, which implies that the constructed dual variables $\{\hat\lambda_j,\hat\theta_i\}$ are feasible.

It therefore remains to establish Eq.~\eqref{eq:pf_dp_additive_111}, for which we introduce the following weak sub-additivity property, the proof of which is provided in Section~\ref{se:app_pf_dp_additive}.
\begin{lemma}[Weak sub-additivity of cost increments]\label{le:pf_dp_additive}
    Fix $h\in [m]$ and any sequence $
        t_h\le \xi_1<\xi_2\le \xi_3<\xi_4\le\dots\le\xi_{2k-1}<\xi_{2k}\le t_h+d_{ih}$.
    Suppose that $\xi_{2u}-\xi_{2u-1}\ge 1$ for any $u\in[k-1]$, we have that $
        \hat\vartheta_{ih}=2\Delta_{ih}[t_h,t_h+d_{ih}) \ge \sum_{u=1}^{k}\Delta_{ih}[\xi_{2u-1},\xi_{2u})$.
\end{lemma}

Next, we apply Lemma~\ref{le:pf_dp_additive} to prove Eq.~\eqref{eq:pf_dp_additive_111}. For notational convenience, we fix $h\in[m]$, and write \[
        \{j_1,\dots,j_w\}=\{j\in S:h\in\mathcal C_i(j)\},\qquad j_1<j_2<\dots<j_w.
\]
Noting that when $w=0$, Eq.~\eqref{eq:pf_dp_additive_111} holds directly. Suppose that $w\ge 1$, we define $b_u\coloneqq \left(t_{j_u}+d_{ij_u}\right)\wedge \left(t_{h}+d_{ih}\right)$ for $u\in[w]$, and show that $\{t_{j_1},b_{j_1},\dots,t_{j_w},b_{j_w}\}$ satisfies the conditions in Lemma~\ref{le:pf_dp_additive}.

By the definition of the set $\{j_1,\dots,j_w\}$, for each $u\in[w-1]$, we have that $
    b_{u}=\left(t_{j_u}+d_{ij_u}\right)\wedge \left(t_{h}+d_{ih}\right)\le t_{j_w}<t_h+d_{ih}$,
which yields that $
    b_{u}=t_{j_u}+d_{i{j_u}}$, for all $u\in[w-1]$.
Since $S$ is a feasible configuration, the processing time intervals of
its jobs on server $i$ are pairwise disjoint. Hence, for every $u\in[w-1]$, we have that 
\[
    t_h\le t_{j_1}<b_{1}\le\dots\le t_{j_w}<b_{w}\le t_h+d_{ih}.
\]
Since Assumption~\ref{assumption:local-bounded-heterogeneity} guaranties that $d_{ij}\ge 1$ for any $i,j$ and any $u \in [w-1]$, it holds that 
\[
    b_{u}-t_{j_u}=t_{j_u}+d_{ij_{u}}-t_{j_u}\ge 1.
\]
Therefore, we can apply Lemma~\ref{le:pf_dp_additive} to the time points $\{t_{j_1},b_{1},\dots,t_{j_w},b_{w}\}$,
which yields that \begin{align*}
    \hat\vartheta_{ih}=2\Delta_{ih}[t_h,t_h+d_{ih})\ge \sum_{j\in S:\mathcal C_i(j)\ni h}\Delta_{ih}[t_j,\left(t_{j}+d_{ij}\right)\wedge \left(t_{h}+d_{ih}\right)).
\end{align*}
This proves Eq.~\eqref{eq:pf_dp_additive_111}, and we finish the proof of Proposition~\ref{prop:pf_value_LP_dual_fea}. \Halmos

\subsection{Putting Things Together}\label{se:pf_choice_beta_eta}

We now complete the proof of Theorem~\ref{tm:main_theorem}.
When $\ln(\max\{\delta,D\})\geq e-1$, we take
\[
    \eta=\eta^*=\frac{1}{\ln(\delta\vee D)}, \qquad \text{and} \qquad 
    \beta=\beta^*=1+\delta D\ln(\delta\vee D),
\]
so that $(\beta,\eta)\in\Theta$. Hence, Proposition~\ref{prop:capa_fea} guarantees capacity feasibility. By Proposition~\ref{prop:pf_value_LP}, the competitive ratio  of Algorithm~\ref{alg:main_d_integer_accurate}  is upper bounded by
\begin{align}\nonumber
&1+2(1+2\eta)\beta^{1/c_{\min}}\ln\beta =1+2\left(1+\frac{2}{\ln(\delta\vee D)}\right)\ln\left(1+\delta D\ln(\delta \vee D)\right)\beta^{1/c_{\min}} \\ \nonumber
&\qquad \le 1+2\left(1+\frac{2}{\ln(\delta\vee D)}\right)\ln\left(\frac{10}{9}\delta D\ln(\delta \vee D)\right)\beta^{1/c_{\min}} \\ 
&\qquad \leq 1+ 2\left(\mathcal O(1)+\ln \left(\delta D\right)+\ln\ln\left(\delta \vee D\right)\right)\beta^{1/c_{\min}} ,\label{eq:com_ratio_fin}
\end{align}
where the first inequality is due to $e^{e-1}(e-1)>9$.

Equation~\eqref{eq:com_ratio_fin} gives a finite-capacity guarantee and therefore does not require taking the large-capacity limit $c_{\min}\to\infty$. This bound also recovers the asymptotic guarantee under the mild scaling $c_{\min}\gg \ln^2(\delta D)$. Indeed, in this regime, $\ln^2(\delta D)/c_{\min}=o(1)$. Since $\ln\beta\le 2\ln(\delta D)$, we have \[
    \beta^{1/c_{\min}}=\exp\!\left(\frac{\ln\beta}{c_{\min}}\right) \le \exp\!\left(\frac{2\ln(\delta D)}{c_{\min}}\right)\le 1+ \frac{4\ln(\delta D)}{c_{\min}}=1+o\left(\frac{1}{\ln(\delta D)}\right).
\]
Substituting this estimate into Eq.~\eqref{eq:com_ratio_fin} shows that Algorithm~\ref{alg:main_d_integer_accurate} achieves a competitive ratio of $
    (2\ln(\delta D) + 2 \ln\ln\left(\delta \vee D\right) + \mathcal{O}(1))$.

\begin{remark}\label{rem:factor_2_main_alg}
The extra factor $2$ in the leading term $2\ln(\delta D)$ in the above competitive ratio bound stems from Lemma~\ref{le:pf_dp_additive}, where the {``worst'' reuse schedule} of disjoint time intervals cannot combine into a {reuse schedule} of the whole time interval. Instead, we can only obtain a weak sub-additivity property of $\Delta_{ih}[a,b)$ over disjoint time intervals. In Algorithm~\ref{alg:main_d_integer_active}, we design a refined {loss} that guarantees the sub-additivity of $\Delta_{ih}[a,b)$ over disjoint time intervals.
\end{remark}

We next consider the remaining case where $L\coloneqq\ln(\delta\vee D)<e-1$. In this regime, we use a constant-regime parameter choice. We set
$\beta=\beta^*=2(\delta D+1)$ and
$\eta=\eta^*=1$
so that $(\beta,\eta)\in\Theta$.
By Proposition~\ref{prop:pf_value_LP}, the competitive ratio of Algorithm~\ref{alg:main_d_integer_accurate} is upper bounded by
$
1+2(1+2\eta)\beta^{1/c_{\min}}\ln\beta
=
1+2\left[\ln 2+2\ln(\delta D+1)\right]\beta^{1/c_{\min}}$.
Letting $c_{\min}\to\infty$, since $\ln(\delta\vee D)<e-1$, we can conclude that Algorithm~\ref{alg:main_d_integer_accurate} is asymptotically $\mathcal O(1)$-competitive in this case. \Halmos

\section{Proof of Theorem~\ref{tm:main_theorem_refine}}\label{se:pf_main_tm_refine}

The proof of Theorem~\ref{tm:main_theorem_refine} mirrors the primal-dual proof of Theorem~\ref{tm:main_theorem}. Since these two arguments are largely parallel, we omit steps that carry over unchanged and focus on the modifications required by Algorithm~\ref{alg:main_d_integer_active}. 
The only substantive departure lies in the refined loss increment as defined in Definition~\ref{def:pathoc-increment_refine}. As discussed in Remark~\ref{rem:factor_2_main_alg}, Lemma~\ref{le:pf_dp_additive} provides only a weak sub-additivity bound for the increment associated with Algorithm~\ref{alg:main_d_integer_accurate}. By contrast, the refined increment $\tilde \Delta_{ih}[a,b)$ fits the form of greedy relaxation loss in Algorithm~\ref{alg:main_d_integer_active} and is sub-additive across disjoint time intervals, which removes the factor-of-two loss in the dual-feasibility argument.

\subsection{Capacity Feasibility of Algorithm~\ref{alg:main_d_integer_active}}\label{se:pf_capa_fea_refine}

We first establish the corresponding capacity-feasibility result for Algorithm~\ref{alg:main_d_integer_active}.

\begin{proposition}[Feasibility of Algorithm~\ref{alg:main_d_integer_active}]\label{prop:capa_fea_refine}
Recall that $\Theta$ is the admissible parameter set in Definition~\ref{df:capa_fea}. For any $(\beta,\eta)\in\Theta$,  Algorithm~\ref{alg:main_d_integer_active} never violates the server-capacity constraints.
\end{proposition}
The only ingredient in the proof of Proposition~\ref{prop:capa_fea_refine} that differs from the proof of Proposition~\ref{prop:capa_fea} is the following analogue of Lemma~\ref{le:capa_fea_lb_integral}, the proof of which is provided in Section~\ref{se:app_pf_capa_fea_lb_integral_refine}. Given Lemma~\ref{le:capa_fea_lb_integral_refine}, the induction argument in the proof of Proposition~\ref{prop:capa_fea} applies verbatim, and Proposition~\ref{prop:capa_fea_refine} holds accordingly.

\begin{lemma}\label{le:capa_fea_lb_integral_refine}
    When Algorithm~\ref{alg:main_d_integer_active} assigns server $i$ to job $j$, we have that
 $ \delta d_{ij}> \Psi\!(\alpha_{i,j-1\to t_j})$.
\end{lemma}

\subsection{A Primal-Dual Bound for Algorithm~\ref{alg:main_d_integer_active} under Fixed Admissible Parameters}\label{se:pf_value_LP_refine}

We next derive a competitive-ratio bound for Algorithm~\ref{alg:main_d_integer_active} under fixed admissible parameters.
\begin{proposition}\label{prop:pf_value_LP_refine}
    Suppose that $(\beta,\eta)\in\Theta$ and all processing durations $d_{ij}$ are integer-valued. The competitive ratio of Algorithm~\ref{alg:main_d_integer_active} is upper bounded by $1+(1+\eta)\beta^{1/c_{\min}}\ln \beta$.
\end{proposition}

Noting that $(\beta,\eta)\in\Theta$ yields that $\alpha_{i,j\to t}\in[0,1]$ throughout the process, we now prove the claimed competitive-ratio bound, following the similar steps in Section~\ref{se:pf_value_LP}.

\noindent\underline{\bf Construction of dual variables.} 
The dual variables $\{\tilde\lambda_j,\tilde\theta_i\}$ are constructed in a manner similar to those in Section~\ref{se:pf_value_LP}, with several components modified to reflect {the greedy relaxation loss}. For each job $j$ and each server $i$, we set
\[\tilde\lambda_j:=\max\left(0,\max_{i\in\mathcal N(j)}
    \left\{
        r_{ij}d_{ij}
        -
       \mathcal{L}_{ij}^{\operatorname{GR}}\right\}\right), \qquad \text{and} \qquad 
\tilde{\theta}_i:=\sum_{j=1}^{m}\tilde\vartheta_{ij},
\]
where $\tilde\vartheta_{ij}$ is the incremental component corresponding to the duration when the algorithm processes job $j$. To specify these incremental components, we now introduce the following definition.
\begin{definition}[Refined Capacity-Induced Loss Increment]
\label{def:pathoc-increment_refine}
Fix a compatible pair $(i,j)\in E$ and a time interval satisfying $ t_j\le T_1<T_2$, define $\tilde\Delta_{ij}[T_1,T_2)\coloneqq$
\begin{align*}
\sum_{\tau\in\mathcal T[T_1,T_2)}\max_{\ell \in \{0\} \cup [|\tilde{\boldsymbol{r}}_{i,j,\tau}|]}\Biggl\{\left(\tilde{r}_{i,j,\tau}^{(\ell)} \wedge r_{ij}\right) \times\Biggl[\Psi\!\left(1-\frac{\ell}{c_i}-\bigl(\alpha_{i,j-1\to\tau}-\alpha_{i,j\to\tau}\bigr)\right)-\Psi\!\left(1-\frac{\ell}{c_i}\right)\Biggr]\Biggr\}.
\end{align*}
{For $(i,j)\notin E$, define $\tilde{\Delta}_{ij}[T_1,T_2)\coloneqq 0$ for any $T_1<T_2$.}
\end{definition}

The remaining part is the same as that in Section~\ref{se:pf_value_LP}. We then define the incremental components $\tilde\vartheta_{ij}$ based on the {refined capacity-induced loss increment}:
\[
    \tilde\vartheta_{ij}\coloneqq \tilde\Delta_{ij}[t_j,t_j+d_{ij})
        \cdot \mathbb{I}[\text{Algorithm~\ref{alg:main_d_integer_active} assigns job $j$ to server $i$}]=\tilde\Delta_{ij}[t_j,t_j+d_{ij}).
\]
Since $\alpha_{i,j\to t}\le \alpha_{i,j-1\to t}$ and $\Psi$ is decreasing, every $\tilde\vartheta_{ij}$ is nonnegative. Hence,
$\tilde{\lambda}_j\ge 0$ and $\tilde{\theta}_i\ge 0$ for all $j$ and $i$.
The resulting variables $\{\tilde{\lambda}_j,\tilde{\theta}_i\}$ form the candidate dual solution induced by the trajectory of Algorithm~\ref{alg:main_d_integer_active}.

\noindent\underline{\bf Objective bound and feasibility of $\{\tilde\lambda_j,\tilde\theta_i\}$.}
Let $\widetilde{\mathrm{ALG}}$ denote the total reward collected by Algorithm~\ref{alg:main_d_integer_active}, and let $\mathrm{Dual}(\boldsymbol{\tilde\lambda},\boldsymbol{\tilde\theta})$ denote the dual objective evaluated at the constructed variables, where $\boldsymbol{\tilde \lambda}=\{\tilde \lambda_j\}_j$ and $\boldsymbol{\tilde \theta}=\{\tilde \theta_i\}_i$. The following Proposition~\ref{prop:pf_value_LP_relate_objectivedual_to_alg_refine} bounds $\mathrm{Dual}(\boldsymbol{\tilde\lambda},\boldsymbol{\tilde\theta})$ in terms of $\widetilde{\mathrm{ALG}}$, and Proposition~\ref{prop:pf_value_LP_dual_fea_refine} establishes the feasibility for $\{\tilde\lambda_j,\tilde\theta_i\}$. Their proofs are provided in Sections~\ref{subse:pf_pf_value_LP_1_refine} and Section~\ref{subse:pf_pf_value_LP_2_refine}, respectively.

\begin{proposition}[Objective value]
\label{prop:pf_value_LP_relate_objectivedual_to_alg_refine}
$\mathrm{Dual}(\boldsymbol{\tilde{\lambda}},\boldsymbol{\tilde{\theta}})
\le
\left[1+(1+\eta)\beta^{1/c_{\min}}\ln\beta\right]\cdot\widetilde{\mathrm{ALG}}$.
\end{proposition}

\begin{proposition}[Dual feasibility]\label{prop:pf_value_LP_dual_fea_refine}
The constructed dual variables $\{\tilde\lambda_j,\tilde\theta_i\}$ are feasible.
\end{proposition}

Combining Claim~\ref{claim:benchmark}, Propositions~\ref{prop:pf_value_LP_relate_objectivedual_to_alg_refine} and ~\ref{prop:pf_value_LP_dual_fea_refine}, we establish the desired competitive-ratio bound: \begin{align*}
    \left[1+(1+\eta)\beta^{1/c_{\min}}\ln\beta\right]\cdot \widetilde{\operatorname{ALG}}\ge\mathrm{Dual}(\boldsymbol{\tilde{\lambda}},\boldsymbol{\tilde{\theta}})\ge P_{\mathrm{OPT}}\ge \operatorname{OPT}.
\end{align*}
This finishes the proof of Proposition~\ref{prop:pf_value_LP_refine}.

\subsubsection{Proof of Proposition~\ref{prop:pf_value_LP_relate_objectivedual_to_alg_refine}} \label{subse:pf_pf_value_LP_1_refine}
The proof of Proposition~\ref{prop:pf_value_LP_relate_objectivedual_to_alg_refine} follows the similar lines in Section~\ref{subse:pf_pf_value_LP_1}. Specifically, let $\widetilde{\operatorname{ALG}}_j=\sum_{i\in \mathcal N(j)}r_{ij}d_{ij}\cdot\mathbb{I}[\text{Algorithm~\ref{alg:main_d_integer_active} assigns $j$ to $i$}]$, $\widetilde{\operatorname{Dual}}_j=\tilde\lambda_j+\sum_ic_i\tilde\vartheta_{ij}$, and $\widetilde\Gamma=1+(1+\eta)\beta^{1/c_{\min}}\ln\beta$. It suffices to show that, for every arriving job $j$, we have
\begin{equation}\label{eq:pf_competitive_ratio_dual_alg_refine}
    \widetilde{\operatorname{Dual}}_j\le \widetilde\Gamma\cdot \widetilde{\operatorname{ALG}}_j.
\end{equation}

Similar with the discussions in Section~\ref{subse:pf_pf_value_LP_1}, to prove Eq.~\eqref{eq:pf_competitive_ratio_dual_alg_refine}, we only need to consider the case when Algorithm~\ref{alg:main_d_integer_active} assigns job $j$ to some server $i$, and the only nontrivial task is to control the server-side increment $\tilde \vartheta_{ij}$. We continue to use a two-step argument, with the calculation details adapted to {the greedy relaxation loss}.

\noindent\underline{\bf Step 1: Bounding the server-side increment.}
Since job $j$ is assigned to server $i$, for every $\tau\in[t_j,t_j+d_{ij})$, we have that $\alpha_{i,j\to\tau}=\alpha_{i,j-1\to\tau}-\frac{1}{c_i}$. Applying Eq.~\eqref{eq:pf_psi_increment_bound} to the definition of $\tilde{\vartheta}_{ij}$ gives \begin{align}\label{eq:pf_value_LP_dual_increment_bound_refine}
    \nonumber c_i\tilde{\vartheta}_{ij}& = c_i\sum_{\tau\in \mathcal{T}[t_j,t_j+d_{ij})}
        \max_{\ell \in \{0\} \cup [|\tilde{\boldsymbol{r}}_{i,j,\tau}|]}\left\{\left(\tilde{r}_{i,j,\tau}^{(\ell)} \wedge r_{ij}\right)\left[\Psi\!\left(1-\frac{\ell}{c_i}-\frac{1}{c_i}\right)-\Psi\!\left(1-\frac{\ell}{c_i}\right)\right]\right\}\\
        \nonumber&\le c_i\sum_{\tau\in \mathcal{T}[t_j,t_j+d_{ij})}
        \max_{\ell \in \{0\} \cup [|\tilde{\boldsymbol{r}}_{i,j,\tau}|]}\left\{\left(\tilde{r}_{i,j,\tau}^{(\ell)} \wedge r_{ij}\right)\left[\beta^{1/c_i}\left(\Psi\!\left(1-\frac{\ell}{c_i}\right)+\eta\right)\frac{\ln\beta}{c_i}\right]\right\}\\
        \nonumber& \le \beta^{1/c_i}\ln\beta\cdot \mathcal{L}_{ij}^{\operatorname{GR}}+\eta\beta^{1/c_i}\ln\beta \sum_{\tau\in \mathcal{T}[t_j,t_j+d_{ij})}\max_{\ell \in \{0\} \cup [|\tilde{\boldsymbol{r}}_{i,j,\tau}|]}\left\{\tilde{r}_{i,j,\tau}^{(\ell)} \wedge r_{ij}\right\}\\
        & \le \beta^{1/c_i}\ln\beta\cdot \mathcal{L}_{ij}^{\operatorname{GR}}+\eta\beta^{1/c_i} r_{ij}d_{ij} \ln\beta,
\end{align}
where the last inequality due to that $\tilde{r}_{i,j,\tau}^{(\ell)} \wedge r_{ij}\le r_{ij}$ and $|\mathcal{T}[t_j,t_j+d_{ij})|=d_{ij}$ when $d_{ij}$ is integral.

\noindent\underline{\bf Step 2: Completing the per-arrival bound.}
Next, we use Eq.~\eqref{eq:pf_value_LP_dual_increment_bound_refine} to prove Eq.~\eqref{eq:pf_competitive_ratio_dual_alg_refine}. Since $r_{ij}d_{ij}\ge \mathcal L_{ij}^{\mathrm{GR}}$, we obtain that
\begin{align*}
    \widetilde{\operatorname{Dual}}_j=\tilde{\lambda}_j+c_i\tilde{\vartheta}_{ij} &\le
    r_{ij}d_{ij}+\beta^{1/c_i}\ln\beta\left(\mathcal{L}_{ij}^{\operatorname{GR}}+\eta r_{ij}d_{ij}\right)\\
    &\le \left[1+(1+\eta)\beta^{1/c_i}\ln\beta\right]r_{ij}d_{ij}=\widetilde\Gamma r_{ij}d_{ij}=\widetilde\Gamma\cdot\widetilde{\operatorname{ALG}}_j,
\end{align*}
where the last inequality follows from $c_i\ge c_{\min}$ and $\beta>1$. This proves Eq.~\eqref{eq:pf_competitive_ratio_dual_alg_refine} holds for every job $j$, and therefore concludes the proof of Proposition~\ref{prop:pf_value_LP_relate_objectivedual_to_alg_refine}. \Halmos

\subsubsection{Proof of Proposition~\ref{prop:pf_value_LP_dual_fea_refine}} \label{subse:pf_pf_value_LP_2_refine}
Following the similar discussions in Section~\ref{subse:pf_pf_value_LP_2}, to establish the dual feasibility of $\{\tilde\lambda_j,\tilde\theta_i\}$, it suffices to prove that, for every server $i$ and every feasible configuration $S\in\mathcal S_i$, it holds that  
\[
\tilde\theta_i\geq \sum_{j\in S}\mathcal{L}_{ij}^{\operatorname{GR}}.
\]
Fix a server $i$ and a feasible configuration $S\in\mathcal S_i$, we prove this inequality. The proof follows a similar two-step argument as in  Section~\ref{subse:pf_pf_value_LP_2}, where the key difference is the new sub-additivity lemma (Lemma~\ref{le:pf_dp_additive_refine}) in Step 2, thanks to {the greedy relaxation losses} and the corresponding refined loss increment.

\noindent\underline{\bf Step 1: Local charging.}
Following Section~\ref{subse:pf_pf_value_LP_2}, we first define a charging set \[
    \tilde{\mathcal C}_i(j)\coloneqq \{h: \text{Algorithm~\ref{alg:main_d_integer_active} assigns } h \text{ to } i,h< j, t_j<t_h+d_{ih} \}.
\]
The only ingredient in this step that differs from that in Step~1 in Section~\ref{subse:pf_pf_value_LP_2} is the following analogue of Lemma~\ref{le:pf_dp_charging}, whose proof is provided in Section~\ref{se:app_pf_dp_charging_refine}. 
\begin{lemma}[Local charging]\label{le:app_pf_dp_charging_refine}
    Fix a server $i$ and a job $j\in S\subseteq \mathcal{S}_i$, we have that \begin{align*}
        \sum_{h\in \tilde{\mathcal C}_i(j)}\tilde \Delta_{ih}[t_j,\left(t_j+d_{ij}\right)\wedge\left(t_{h}+d_{ih}\right)) \ge \mathcal{L}_{ij}^{\operatorname{GR}}.
    \end{align*}
\end{lemma}
Lemma~\ref{le:app_pf_dp_charging_refine} Applying Lemma~\ref{le:app_pf_dp_charging_refine} to every $j\in S$ and summing over $j\in S$ yields 
\begin{align}
    \sum_{j\in S}\mathcal{L}_{ij}^{\operatorname{GR}}&\le\sum_{j\in S}\sum_{h\in\tilde{\mathcal C}_i(j)}\tilde\Delta_{ih}\left[t_j,\left(t_j+d_{ij}\right)\wedge\left(t_{h}+d_{ih}\right)\right)\le\sum_{h=1}^m\sum_{j\in S:\tilde{\mathcal C}_i(j)\ni h} \tilde\Delta_{ih}\left[t_j,\left(t_j+d_{ij}\right)\wedge\left(t_{h}+d_{ih}\right)\right). \label{eq:local-charging-sum_refine}
\end{align}

\noindent\underline{\bf Step 2: Aggregation via sub-additivity.}
As discussed in Step~2 in Section~\ref{subse:pf_pf_value_LP_2}, given Eq.~\eqref{eq:local-charging-sum_refine}, the dual-feasibility constraints follow once we establish that
\begin{equation}\label{eq:pf_dp_additive_111_refine}
    \sum_{j\in S:\tilde{\mathcal C}_i(j)\ni h} \tilde\Delta_{ih}\left[t_j,\left(t_j+d_{ij}\right)\wedge\left(t_{h}+d_{ih}\right)\right)\le \tilde\vartheta_{ih},\qquad \forall h\in[m]. 
\end{equation}

It therefore remains to establish Eq.~\eqref{eq:pf_dp_additive_111_refine}, for which we introduce the following sub-additivity property. Its proof is provided in Section~\ref{se:app_pf_dp_additive_refine}.
\begin{lemma}[{Sub-additivity of refined loss increments}]\label{le:pf_dp_additive_refine}
    Fix job $h$ and any sequence $
        t_h\le \xi_1<\xi_2\le \xi_3<\xi_4\le\dots\le\xi_{2k-1}<\xi_{2k}\le t_h+d_{ih}$.
    Suppose that $\xi_{2u}-\xi_{2u-1}$ are all positive integers for any $u\in[k-1]$, we have that $
        \tilde\vartheta_{ih}=\tilde\Delta_{ih}[t_h,t_h+d_{ih}) \ge \sum_{u=1}^{k}\tilde\Delta_{ih}[\xi_{2u-1},\xi_{2u})$.
\end{lemma}

Next, we apply Lemma~\ref{le:pf_dp_additive_refine} to prove Eq.~\eqref{eq:pf_dp_additive_111_refine}. For notational convenience, we fix $h\in[m]$, and write $
        \{j_1,\dots,j_w\}=\{j\in S:h\in\tilde{\mathcal C}_i(j)\}$, where $j_1<j_2<\dots<j_w$.
Noting that when $w=0$, Eq.~\eqref{eq:pf_dp_additive_111_refine} holds directly. Suppose that $w\ge 1$, we define $b_u\coloneqq \left(t_{j_u}+d_{ij_u}\right)\wedge\left(t_{h}+d_{ih}\right)$ for $u\in[w]$, and show that $\{t_{j_1},b_{1},\dots,t_{j_w},b_{w}\}$ satisfies the conditions in Lemma~\ref{le:pf_dp_additive_refine}.
As in Step~2 in Section~\ref{subse:pf_pf_value_LP_2}, it holds that $b_{u}=t_{j_u}+d_{i{j_u}}$ for any $u\in[w-1]$, and we have that $
    t_h\le t_{j_1}<b_{1}\le\dots\le t_{j_w}<b_{w}\le t_h+d_{ih}$.
Since $d_{ij}$ are integers for any $i,j$, for each $u\in[w-1]$, it holds that $    b_{u}-t_{j_u}=t_{j_u}+d_{ij_{u}}-t_{j_u}=d_{ij_{u}}\in\mathbb{Z}_{>0}$.
Therefore, we can apply Lemma~\ref{le:pf_dp_additive_refine} to the time points $\{t_{j_1},b_{1},\dots,t_{j_w},b_{w}\}$,
which yields that $
    \tilde\vartheta_{ih}=\tilde\Delta_{ih}[t_h,t_h+d_{ih})\ge \sum_{j\in S:\tilde{\mathcal C}_i(j)\ni h}\tilde\Delta_{ih}[t_j,\left(t_{j}+d_{ij}\right)\wedge\left(t_{h}+d_{ih}\right))$.
This proves Eq.~\eqref{eq:pf_dp_additive_111_refine}, and we finish the proof of Proposition~\ref{prop:pf_value_LP_dual_fea_refine}. \Halmos

\subsection{Putting Things Together}\label{se:pf_choice_beta_eta_refine}
We now complete the proof of Theorem~\ref{tm:main_theorem_refine}.
When $\ln(\max\{\delta,D\})\geq e-1$, we take
$
    \eta=\eta^*=\frac{1}{\ln(\delta\vee D)}$, and
$    \beta=\beta^*=1+\delta D\ln(\delta\vee D)$,
so that $(\beta,\eta)\in\Theta$. Hence, Proposition~\ref{prop:capa_fea_refine} guarantees capacity feasibility. By Proposition~\ref{prop:pf_value_LP_refine} and following the similar calculation in Section~\ref{se:pf_choice_beta_eta}, we upper bound the competitive ratio  of Algorithm~\ref{alg:main_d_integer_active} by
$1+(1+\eta)\beta^{1/c_{\min}}\ln\beta  \leq 1+ \left(\mathcal O(1)+\ln \left(\delta D\right)+\ln\ln\left(\delta \vee D\right)\right)\beta^{1/c_{\min}}$. When $c_{\min}\gg \ln^2(\delta D)$,  we have $
\beta^{1/c_{\min}}=1+o\left(\frac{1}{\ln(\delta D)}\right)$ and  Algorithm~\ref{alg:main_d_integer_active} achieves a competitive ratio of $
    (\ln(\delta D) +  \ln\ln\left(\delta \vee D\right) + \mathcal{O}(1))$.
For the remaining case where $L\coloneqq\ln(\delta\vee D)<e-1$, we set
$\beta=\beta^*=2(\delta D+1)$ and $
\eta=\eta^*=1$, and 
conclude that Algorithm~\ref{alg:main_d_integer_active} is asymptotically $\mathcal O(1)$-competitive following the same argument as in Section~\ref{se:pf_choice_beta_eta}. \Halmos

\section{Numerical Experiments}\label{sec:experiments}
We conduct numerical experiments to compare TS-BAL and GR-BAL with Forward-Looking Balance (FLB) and Greedy. The experiments cover three settings. First, we construct two families of challenging instances: one designed to expose the limitations of FLB and Greedy under globally drifting rewards, and another motivated by our lower-bound construction to challenge TS-BAL and GR-BAL. Second, we evaluate the algorithms in randomly generated nonstationary environments with stochastic arrivals, time-varying reward rates, assignment-specific processing durations, and a compatibility structure that induces local competition for reusable capacity. Third, we examine the small-capacity performance of the algorithms by varying the server capacity from $2$ to $40$.

Detailed experimental settings and results are deferred to Appendix~\ref{app:experiments-details} due to space constraints. Overall, the results demonstrate that TS-BAL and GR-BAL perform robustly across the tested settings. On the artificially designed hard instances, they substantially outperform FLB and Greedy when rewards exhibit substantial global drift, while remaining competitive with FLB on the instance family motivated by our lower bound. In the randomly generated environments, TS-BAL and GR-BAL consistently attain the lowest average competitive ratios, with performance that remains stable as the horizon increases. After a brief initial non-monotonic phase, their competitive ratios generally decline as capacity increases. Although Greedy can be competitive at very small capacities in some settings, TS-BAL and GR-BAL eventually outperform it as capacity grows and substantially outperform FLB throughout the tested capacity range. These findings support the theoretical robustness of our algorithms and indicate that their advantages remain meaningful at finite capacity. We also observe that, although the current analysis yields a competitive-ratio bound for TS-BAL whose leading term is twice that of GR-BAL, TS-BAL performs slightly better than GR-BAL in the randomly generated environments, which further underscores the question raised at the end of Section~\ref{sec:algorithm_main_theorem}: whether the factor of $2$ in the leading term of TS-BAL’s competitive-ratio bound can be eliminated.

\section{Conclusion}

We study online bipartite matching with reusable server capacity under non-stationary rewards. Our locally bounded reward condition restricts reward variation only among jobs that can compete for the same capacity within a relevant time window, thereby accommodating substantial reward drift over long horizons. We develop two BALANCE-type algorithms that price an assignment by the future service trajectories it may displace. TS-BAL aggregates time-aware blocking losses over feasible reuse schedules, whereas GR-BAL applies a greedy relaxation with stronger sub-additivity and attains the asymptotically tight leading term $\ln(\delta D)$. Both algorithms perform robustly under globally drifting rewards and favorably at finite capacities.

These results show that opportunity costs for reusable capacity should reflect the most consequential feasible continuation that a current assignment may block. An interesting direction is to determine whether TS-BAL can attain the optimal leading coefficient without greedy relaxation.
%

\bibliographystyle{informs2014}
 \let\oldbibliography\thebibliography
 \renewcommand{\thebibliography}[1]{%
    \oldbibliography{#1}%
    \baselineskip14pt 
    \setlength{\itemsep}{10pt}
 }
\bibliography{refs}

@article{feng2025online,
  title={Online job assignment},
  author={Ekbatani, Farbod and Feng, Yiding and Kash, Ian and Niazadeh, Rad},
  journal={arXiv preprint arXiv:2506.06893},
  year={2025}
}

@article{goyal2025asymptotically,
  title={Asymptotically optimal competitive ratio for online allocation of reusable resources},
  author={Goyal, Vineet and Iyengar, Garud and Udwani, Rajan},
  journal={Operations Research},
  volume={73},
  number={4},
  pages={1897--1915},
  year={2025},
  publisher={INFORMS}
}

@article{huang2024online,
  title={Online matching: A brief survey},
  author={Huang, Zhiyi and Tang, Zhihao Gavin and Wajc, David},
  journal={ACM SIGecom Exchanges},
  volume={22},
  number={1},
  pages={135--158},
  year={2024},
  publisher={ACM New York, NY, USA}
}

@article{huo2022onlineReusableMultiClass,
  title={Online reusable resource allocations with multi-class arrivals},
  author={Huo, Tianming and Cheung, Wang Chi},
  journal={Available at SSRN 4320423},
  year={2022}
}

@inproceedings{karp1990optimal,
  title={An optimal algorithm for on-line bipartite matching},
  author={Karp, Richard M and Vazirani, Umesh V and Vazirani, Vijay V},
  booktitle={Proceedings of the twenty-second annual ACM symposium on Theory of computing},
  pages={352--358},
  year={1990}
}

@article{kalyanasundaram2000optimal,
  title={An optimal deterministic algorithm for online b-matching},
  author={Kalyanasundaram, Bala and Pruhs, Kirk R},
  journal={Theoretical Computer Science},
  volume={233},
  number={1-2},
  pages={319--325},
  year={2000},
  publisher={Elsevier}
}

@article{mehta2007adwords,
  title={Adwords and generalized online matching},
  author={Mehta, Aranyak and Saberi, Amin and Vazirani, Umesh and Vazirani, Vijay},
  journal={Journal of the ACM (JACM)},
  volume={54},
  number={5},
  pages={22--es},
  year={2007},
  publisher={ACM New York, NY, USA}
}

@article{gong2022online,
  title={Online assortment optimization with reusable resources},
  author={Gong, Xiao-Yue and Goyal, Vineet and Iyengar, Garud N and Simchi-Levi, David and Udwani, Rajan and Wang, Shuangyu},
  journal={Management Science},
  volume={68},
  number={7},
  pages={4772--4785},
  year={2022},
  publisher={INFORMS}
}

@inproceedings{nekouyan2026online,
  title={Online Rounding and Pricing Schemes for k-Rental Problems},
  author={Nekouyan, Hossein and Sun, Bo and Boutaba, Raouf and Tan, Xiaoqi},
  booktitle={Proceedings of the ACM Web Conference 2026},
  pages={201--212},
  year={2026}
}

@article{rusmevichientong2020DARP,
  title={Dynamic assortment optimization for reusable products with random usage durations},
  author={Rusmevichientong, Paat and Sumida, Mika and Topaloglu, Huseyin},
  journal={Management Science},
  volume={66},
  number={7},
  pages={2820--2844},
  year={2020},
  publisher={INFORMS}
}

@article{ma2020algorithms,
  title={Algorithms for online matching, assortment, and pricing with tight weight-dependent competitive ratios},
  author={Ma, Will and Simchi-Levi, David},
  journal={Operations Research},
  volume={68},
  number={6},
  pages={1787--1803},
  year={2020},
  publisher={INFORMS}
}

@article{aminian2026OJS,
  title  = {Online Job Selection: Reward Rate vs. Remaining Value},
  author = {Aminian, Mohammad Reza and Ma, Will and Xin, Linwei},
  year   = {2026},
  journal   = {Available at SSRN: https://ssrn.com/abstract=4644495 or http://dx.doi.org/10.2139/ssrn.4644495}
}

@article{simchi2025greedy,
  title={On greedy-like policies in online matching with reusable network resources and decaying rewards},
  author={Simchi-Levi, David and Zheng, Zeyu and Zhu, Feng},
  journal={Management Science},
  volume={71},
  number={10},
  pages={8908--8926},
  year={2025},
  publisher={INFORMS}
}

@article{golrezaei2014real,
  title={Real-time optimization of personalized assortments},
  author={Golrezaei, Negin and Nazerzadeh, Hamid and Rusmevichientong, Paat},
  journal={Management Science},
  volume={60},
  number={6},
  pages={1532--1551},
  year={2014},
  publisher={INFORMS}
}

@article{baek2026leveraging,
  title={Leveraging reusability: Improved competitive ratio of greedy for reusable resources},
  author={Baek, Jackie and Wang, Shixin},
  journal={Operations Research Letters},
  pages={107480},
  year={2026},
  publisher={Elsevier}
}

@inproceedings{delong2022online,
  title={Online bipartite matching with reusable resources},
  author={Delong, Steven and Farhadi, Alireza and Niazadeh, Rad and Sivan, Balasubramanian},
  booktitle={Proceedings of the 23rd ACM Conference on Economics and Computation},
  pages={962--963},
  year={2022}
}

@inproceedings{yao1977probabilistic,
  title={Probabilistic Computations: Toward a Unified Measure of Complexity},
  author={Yao, Andrew Chi-Chih},
  booktitle={Proceedings of the 18th Annual Symposium on Foundations of Computer Science},
  pages={222--227},
  year={1977},
  organization={IEEE}
}

@book{borodin1998online,
  title={Online Computation and Competitive Analysis},
  author={Borodin, Allan and El-Yaniv, Ran},
  year={1998},
  publisher={Cambridge University Press}
}

@article{liu2026unified,
  title={A Unified Framework for Online Combinatorial Allocation of Reusable Resources under Endogenous Deterioration},
  author={Liu, Qingsong and Hajiesmaili, Mohammad},
  journal={Proceedings of the ACM on Measurement and Analysis of Computing Systems},
  volume={10},
  number={2},
  pages={1--47},
  year={2026},
  publisher={ACM New York, NY, USA}
}

@article{aminian2026bayesian,
  title={Optimal Bayesian Online Allocation of Reusable Resources},
  author={Aminian, Mohammad Reza and Niazadeh, Rad and Nuti, Pranav},
  journal={Available at SSRN 6384158},
  year={2026}
}

@inproceedings{feng2022near,
  title={Near-optimal Bayesian online assortment of reusable resources},
  author={Feng, Yiding and Niazadeh, Rad and Saberi, Amin},
  booktitle={Proceedings of the 23rd ACM Conference on Economics and Computation},
  pages={964--965},
  year={2022}
}

@inproceedings{kang2025replenishment,
  title={A black-box approach for exogenous replenishment in online resource allocation},
  author={Kang, Suho and Liu, Ziyang and Udwani, Rajan},
  booktitle={International Conference on Integer Programming and Combinatorial Optimization},
  pages={326--340},
  year={2025},
  organization={Springer}
}

@article{udwani2024submodular,
  title={Optimality of Nonadaptive Algorithms in Online Submodular Welfare Maximization with Stochastic Outcomes},
  author={Udwani, Rajan},
  journal={Operations Research},
  year={2026},
  publisher={INFORMS}
}

@article{jia2024pricing,
  title={Online learning and pricing for service systems with reusable resources},
  author={Jia, Huiwen and Shi, Cong and Shen, Siqian},
  journal={Operations Research},
  volume={72},
  number={3},
  pages={1203--1241},
  year={2024},
  publisher={Informs}
}

@article{zhang2025nonstasto,
  title={Online Allocation of Reusable Resources in Nonstationary Environments},
  author={Zhang, Xilin and Cheung, Wang Chi},
  journal={Mathematics of Operations Research},
  year={2025},
  publisher={INFORMS}
}

@article{udwani2024adwords,
  title={Adwords with unknown budgets and beyond},
  author={Udwani, Rajan},
  journal={Management Science},
  volume={71},
  number={2},
  pages={1009--1026},
  year={2025},
  publisher={INFORMS}
}

@inproceedings{feldman2009displayads,
  title={Online ad assignment with free disposal},
  author={Feldman, Jon and Korula, Nitish and Mirrokni, Vahab and Muthukrishnan, Shanmugavelayutham and P{\'a}l, Martin},
  booktitle={International workshop on internet and network economics},
  pages={374--385},
  year={2009},
  organization={Springer}
}

@article{ball2009toward,
  title={Toward robust revenue management: Competitive analysis of online booking},
  author={Ball, Michael O and Queyranne, Maurice},
  journal={Operations Research},
  volume={57},
  number={4},
  pages={950--963},
  year={2009},
  publisher={INFORMS}
}

@article{buchbinder2009online,
  title={Online primal-dual algorithms for covering and packing},
  author={Buchbinder, Niv and Naor, Joseph},
  journal={Mathematics of Operations Research},
  volume={34},
  number={2},
  pages={270--286},
  year={2009},
  publisher={INFORMS}
}

@article{levi2010provably,
  title={Provably near-optimal LP-based policies for revenue management in systems with reusable resources},
  author={Levi, Retsef and Radovanovi{\'c}, Ana},
  journal={Operations Research},
  volume={58},
  number={2},
  pages={503--507},
  year={2010},
  publisher={INFORMS}
}

@inproceedings{aggarwal2011online,
  title={Online vertex-weighted bipartite matching and single-bid budgeted allocations},
  author={Aggarwal, Gagan and Goel, Gagan and Karande, Chinmay and Mehta, Aranyak},
  booktitle={Proceedings of the twenty-second annual ACM-SIAM symposium on Discrete Algorithms},
  pages={1253--1264},
  year={2011},
  organization={SIAM}
}

@inproceedings{azar2016online,
  title={Online algorithms for covering and packing problems with convex objectives},
  author={Azar, Yossi and Buchbinder, Niv and Chan, TH Hubert and Chen, Shahar and Cohen, Ilan Reuven and Gupta, Anupam and Huang, Zhiyi and Kang, Ning and Nagarajan, Viswanath and Naor, Joseph and others},
  booktitle={2016 IEEE 57th Annual Symposium on Foundations of Computer Science (FOCS)},
  pages={148--157},
  year={2016},
  organization={IEEE}
}

@article{owen2018price,
  title={Price and assortment optimization for reusable resources},
  author={Owen, Zachary and Simchi-Levi, David},
  journal={Available at SSRN 3070625},
  year={2018}
}

@article{feng2019linear,
  title={Linear programming based online policies for real-time assortment of reusable resources},
  author={Feng, Yiding and Niazadeh, Rad and Saberi, Amin},
  journal={Chicago Booth Research Paper},
  number={20-25},
  year={2019}
}

@inproceedings{besbes2019static,
  title={Static pricing: Universal guarantees for reusable resources},
  author={Besbes, Omar and Elmachtoub, Adam N and Sun, Yunjie},
  booktitle={Proceedings of the 2019 ACM Conference on Economics and Computation},
  pages={393--394},
  year={2019}
}

@article{dickerson2021allocation,
  title={Allocation problems in ride-sharing platforms: Online matching with offline reusable resources},
  author={Dickerson, John P and Sankararaman, Karthik A and Srinivasan, Aravind and Xu, Pan},
  journal={ACM Transactions on Economics and Computation (TEAC)},
  volume={9},
  number={3},
  pages={1--17},
  year={2021},
  publisher={ACM New York, NY, USA}
}

@article{feng2025robustness,
  title={Robustness of online inventory balancing to inventory shocks},
  author={Feng, Yiding and Niazadeh, Rad and Saberi, Amin},
  journal={arXiv preprint arXiv:2511.16044},
  year={2025}
}

@article{baek2022bifurcating,
  title={Bifurcating constraints to improve approximation ratios for network revenue management with reusable resources},
  author={Baek, Jackie and Ma, Will},
  journal={Operations Research},
  volume={70},
  number={4},
  pages={2226--2236},
  year={2022},
  publisher={INFORMS}
}

@article{feng2025batching,
  title={Batching and optimal multistage bipartite allocations},
  author={Feng, Yiding and Niazadeh, Rad},
  journal={Management Science},
  volume={71},
  number={5},
  pages={4108--4130},
  year={2025},
  publisher={INFORMS}
}

@article{willMa2026dynamic,
  title={Dynamic pricing for reusable resources: The power of two prices},
  author={Balseiro, Santiago R and Ma, Will and Zhang, Wenxin},
  journal={Operations Research},
  volume={74},
  number={3},
  pages={1476--1495},
  year={2026},
  publisher={INFORMS}
}

\newpage

\begin{center}
    {\Large \bf Supplementary Materials}
\end{center}
\begin{APPENDICES}

\medskip

\section{Omitted Proof(s) in Section~\ref{se:configuration_LP}}
\subsection{Proof of Claim~\ref{claim:benchmark}}\label{se:app_pf_claim_benchmark}
We now explain why the configuration LP is a relaxation of the offline benchmark. 
Consider any feasible offline assignment. For each server $i$, represent its 
capacity $c_i$ by $c_i$ labeled identical capacity units. Since the offline 
assignment is capacity feasible, whenever a job assigned to server $i$ starts, 
there is at least one capacity unit of server $i$ that is idle at that time. 
Assign the job to one such idle unit and keep it on that unit throughout its 
processing interval.

For each capacity unit $w\in[c_i]$ of server $i$, let $S_{iw}$ be the set of jobs 
processed by this unit in the above construction. By construction, the processing 
intervals of jobs in $S_{iw}$ are pairwise non-overlapping, and every job in 
$S_{iw}$ is compatible with server $i$. Hence $S_{iw}\in\mathcal S_i$. We then 
construct a primal solution by setting
\[
    x(i,S)=\bigl|\{w\in[c_i]: S_{iw}=S\}\bigr|,
    \qquad \forall i\in[n],\; S\in\mathcal S_i,
\]
and setting all other variables to zero. This solution satisfies the job 
constraints because each job selected by the offline assignment is placed on 
exactly one capacity unit, and it satisfies the server constraints because server 
$i$ uses exactly $c_i$ capacity units. Moreover, the objective value of this 
primal solution is exactly the total reward collected by the offline assignment.

Applying this construction to an optimal offline assignment yields
\[
    \operatorname{OPT}\le P_{\mathrm{OPT}}.
\]
Thus, the configuration LP provides a valid upper bound on the offline optimum. \Halmos

\section{Omitted Proof(s) in Section~\ref{se:pf_main_tm}}
\subsection{Proof of Lemma~\ref{le:capa_fea_lb_integral}}\label{sec:app_pf_le_capa_fea_integral}
Suppose that Algorithm~\ref{alg:main_d_integer_accurate} assigns server $i$ to job $j$. Noting that $\{t_j\}\in \mathfrak{S}[t_j,t_j+d_{ij})$, we have that \[
    \mathcal{L}_{ij}^{\operatorname{TS}}=\max_{\boldsymbol{\tau}\in\mathfrak S[t_j,t_j+d_{ij})}
 \left\{\sum_{\tau\in\boldsymbol{\tau}}\hat r_{i,j,\tau}\,\Psi\!\left(\alpha_{i,j-1\to \tau}\right)
 \right\}\ge \hat r_{i,j,t_j} \Psi\!(\alpha_{i,j-1\to t_j}).
\]
Recall that $\hat r_{i,j,t_j}=\min\left\{ r_{ik}: k\le j,\; i\in N(k),\;t_k\ge t_j-D\right\}$. For every job $k$ in this set, both $j$ and $k$ are compatible with server $i$, and $D \ge |t_j-t_k|$. By Assumption~\ref{assumption:local-bounded-heterogeneity}, it follows that $r_{ik}\ge r_{ij}/\delta$. Consequently, we have that $\hat r_{i,j,t_j}\ge r_{ij}/\delta$.

Since whenever Algorithm~\ref{alg:main_d_integer_accurate} assigns $i$ to job $j$, the corresponding adjusted reward is positive, we obtain that \[   
    r_{ij}d_{ij}>\hat r_{i,j,t_j} \Psi\!(\alpha_{i,j-1\to t_j})\ge \frac{1}{\delta}r_{ij}\Psi\!(\alpha_{i,j-1\to t_j}).\]
Since $r_{ij}>0$, we can conclude that
$
    \delta d_{ij}> \Psi\!(\alpha_{i,j-1\to t_j})$,
which proves the lemma. \Halmos

\subsection{Proof of Lemma~\ref{le:pf_dp_charging}}\label{se:app_pf_dp_charging}
To begin with, we use $h_1,\dots,h_s$ to denote the elements in $\mathcal C_i(j)$: \[
    \mathcal C_i(j)=\{h_1,\ldots,h_s\}, \qquad h_1<\cdots<h_s,
\]
and write $I_j\coloneqq [t_j,t_j+d_{ij})$. 

If $s=0$, no job assigned to server $i$ before job $j$ is active
at $t_j$. Since a completed job cannot become active again, we have
\[
\alpha_{i,j-1\to t}=1,
\qquad t\in I_j.
\]
Thus, $\mathcal{L}_{ij}^{\operatorname{TS}}=0$, and the result follows.
Suppose henceforth that $s\ge1$. In this case, we present the proof proceeds in the following three steps.

\noindent\underline{\bf Step 1: Characterizing the capacity recursion on $I_j$.}
For each $w\in[s]$, denote $b_w\coloneqq \left(t_{j}+d_{ij}\right)\wedge\left(t_{h_w}+d_{ih_w}\right)$. By Definition~\ref{df:available_capa}, for every $h\in[j-1]$ and $t\in I_j$, we have that \[
    \alpha_{i,h\to t}=\alpha_{i,h-1\to t}-\frac{1}{c_i}
    \mathbb I\{\text{Algorithm~\ref{alg:main_d_integer_accurate} assign }h \text{ to } i\}
    \mathbb I\{t<t_h+d_{ih}\}.
\]
Since $\mathcal C_i(j)$ consists precisely of the jobs assigned to server $i$ before $j$ that are active at $t_j$, for $h\in[j-1],\ t\in I_j$, the recursion of $\alpha_{i,h\to t}$ over $h$ can be written as
\begin{equation}
\alpha_{i,h\to t}
=
\begin{cases}
\displaystyle
\alpha_{i,h-1\to t}-\frac{1}{c_i},
&
h=h_w\text{ for some }w\in[s],
\quad t\in[t_j,b_w),\\[6pt]
\alpha_{i,h-1\to t},
&
\text{otherwise}.
\end{cases}
\label{eq:lemma2-capacity-recursion}
\end{equation}
Recall that $\alpha_{i,0\to t}=1$ for any $t\ge 0$. Therefore, for every $t\in I_j$, we obtain that,
\begin{equation}
\sum_{w=1}^s \left[ \Psi(\alpha_{i,h_w\to t})-\Psi(\alpha_{i,h_w-1\to t})\right]
=\Psi(\alpha_{i,j-1\to t})-\Psi(1)=\Psi(\alpha_{i,j-1\to t}).
\label{eq:lemma2-potential-telescope}
\end{equation}

\noindent\underline{\bf Step 2: Charging each capacity increment to a ``worst'' reuse schedule.}
We use $\boldsymbol{\tau}^{\star}$ to denote a ``worst'' reuse schedule for $\mathcal{L}_{ij}^{\operatorname{TS}}$. Formally, let \[
\boldsymbol{\tau}^{\star}
\in
\underset{\boldsymbol{\tau}\in
\mathfrak{S}[t_j,t_j+d_{ij})}{\arg\max}
\sum_{\tau\in\boldsymbol{\tau}}
\hat r_{i,j,\tau}
\Psi(\alpha_{i,j-1\to\tau}).
\]

For each $w\in[s]$, let
\[
\boldsymbol{\tau}^{\,w}
:=
\left\{
\tau\in\boldsymbol{\tau}^{\star}:
\tau<b_w
\right\}.
\]
Then it directly holds that $\boldsymbol{\tau}^{\,w}\in\mathfrak{S}[t_j,b_w)$.

For every $\tau\in\boldsymbol{\tau}^{\,w}$, since
\[
[t_j,b_w)
\subseteq
[t_{h_w},t_{h_w}+d_{ih_w})
\subseteq
[t_{h_w},t_{h_w}+D),
\]
$\hat r_{i,h_w,\tau}$ is well defined. Furthermore, since $ \mathcal K_{i,h_w}(\tau) \subseteq \mathcal K_{i,j}(\tau)$, we have that
\begin{equation}
\hat r_{i,h_w,\tau}
\ge
\hat r_{i,j,\tau}.
\label{eq:lemma2-reward-comparison}
\end{equation}

Using $\boldsymbol{\tau}^{\,w}$ as a feasible path in Eq.~\eqref{eq:pf_df_delta_relate_increment_1} and applying
Eq.~\eqref{eq:lemma2-reward-comparison}, we obtain
\begin{align}
\Delta_{ih_w}[t_j,b_w)\ge
\sum_{\tau\in\boldsymbol{\tau}^{\,w}}
\hat r_{i,j,\tau}
\left[
\Psi(\alpha_{i,h_w\to\tau})
-
\Psi(\alpha_{i,h_w-1\to\tau})
\right].
\label{eq:lemma2-restricted-charge}
\end{align}

Meanwhile, if there exists $\tau\in\boldsymbol{\tau}^{\star}\setminus\boldsymbol{\tau}^{\,w}$, 
then for such $\tau$, it must holds that $\tau\ge b_w=t_{h_w}+d_{ih_w}$. Hence, by Eq.~\eqref{eq:lemma2-capacity-recursion}, we have that $ \alpha_{i,h_w\to\tau}=\alpha_{i,h_w-1\to\tau}$. Consequently,
Eq.~\eqref{eq:lemma2-restricted-charge} remains valid when
$\boldsymbol{\tau}^{\,w}$ is replaced by $\boldsymbol{\tau}^{\star}$:
\begin{align}
\Delta_{ih_w}[t_j,b_w)\ge
\sum_{\tau\in\boldsymbol{\tau}^{\star}}
\hat r_{i,j,\tau}
\left[
\Psi(\alpha_{i,h_w\to\tau})
-
\Psi(\alpha_{i,h_w-1\to\tau})
\right].
\label{eq:lemma2-full-path-charge}
\end{align}

\noindent\underline{\bf Step 3: Aggregating the charges.}
Summing Eq.~\eqref{eq:lemma2-full-path-charge} over
$w=1,\ldots,s$, changing the order of summation, and applying
Eq.~\eqref{eq:lemma2-potential-telescope}, we obtain
\begin{align*}
&~~~\sum_{h\in \mathcal C_i(j)}\Delta_{ih}[t_j,
\left(t_{j}+d_{ij}\right)\wedge\left(t_{h}+d_{ih}\right))\\
&=\sum_{w=1}^s
\Delta_{ih_w}[t_j,b_{w})\ge
\sum_{\tau\in\boldsymbol{\tau}^{\star}}
\hat r_{i,j,\tau}
\sum_{w=1}^s
\left[
\Psi(\alpha_{i,h_w\to\tau})
-
\Psi(\alpha_{i,h_w-1\to\tau})
\right]\\
&=
\sum_{\tau\in\boldsymbol{\tau}^{\star}}
\hat r_{i,j,\tau}
\left[
\Psi(\alpha_{i,j-1\to\tau})
-
\Psi(1)
\right]=
\sum_{\tau\in\boldsymbol{\tau}^{\star}}
\hat r_{i,j,\tau}
\Psi(\alpha_{i,j-1\to\tau})=\mathcal{L}_{ij}^{\operatorname{TS}},
\end{align*}
where the third equality follows from $\Psi(1)=0$. This completes
the proof. \Halmos

\subsection{Proof of Lemma~\ref{le:pf_dp_additive}}\label{se:app_pf_dp_additive}
Fix $i$ and a time interval satisfying $t_h \le a < b$, let \[
    g_h(\tau)\coloneqq \hat r_{i,h,\tau} \left[ \Psi\!(\alpha_{i,h\to\tau})-\Psi\!(\alpha_{i,h-1\to\tau})\right].
\]
Since $\alpha_{i,h\to\tau}\le\alpha_{i,h-1\to\tau}$ and $\Psi$ is decreasing, we have $g_h(\tau)\ge 0$. By Eq.~\eqref{eq:pf_df_delta_relate_increment_1}, we have that \[
    \Delta_{ih}[a,b)=\max_{\boldsymbol{\tau}\in\mathfrak{S}[a,b)} \sum_{\tau\in \boldsymbol{\tau}} g_h(\tau).
\]

For each $u\in[k]$, let \[
    \boldsymbol{\tau}_u\coloneqq \underset{\boldsymbol{\tau}\in\mathfrak{S}[\xi_{2u-1},\xi_{2u})}{\arg\max}\sum_{\tau\in\boldsymbol{\tau}} g_h(\tau)
\]
be {a ``worst'' reuse schedule} on $[\xi_{2u-1},\xi_{2u})$. Partition these schedules according to the parity of their indices, and define \[
\boldsymbol{\tau}^{\mathrm{odd}}\coloneqq\bigsqcup_{\substack{u\in[k]\\u\text{ odd}}} \boldsymbol{\tau}_u,
\qquad
\boldsymbol{\tau}^{\mathrm{even}}\coloneqq\bigsqcup_{\substack{u\in[k]\\u\text{ even}}} \boldsymbol{\tau}_u,
\]
where the elements in each union are listed in increasing order.

We first verify that both unions are reuse schedules over
$[t_h,t_h+d_{ih})$, i.e., $\boldsymbol{\tau}^{\mathrm{odd}}$ and $\boldsymbol{\tau}^{\mathrm{even}}$ are feasible elements that satisfy the conditions in the definition of $\mathfrak{S}[t_h,t_h+d_{ih})$. Firstly, feasibility within each $\boldsymbol{\tau}_u$ follows from its construction. Then we prove the feasibility over different $\boldsymbol{\tau}_u$. Consider two paths $\boldsymbol{\tau}_u$ and $\boldsymbol{\tau}_{u'}$ belonging to the same parity class. Without loss of generality, suppose that $\boldsymbol{\tau}_u$ and $\boldsymbol{\tau}_{u'}$ belong to $\boldsymbol{\tau}^{\mathrm{odd}}$ and $u'>u$. Hence, for any $\tau\in \boldsymbol{\tau}_u$ and $\tau'\in \boldsymbol{\tau}_{u'}$, we have that $\tau<\xi_{2u}, \tau'\ge \xi_{2u+3}$.
Moreover, since \[
    \xi_{2u+3}-\xi_{2u} \ge \xi_{2u+2}-\xi_{2u+1}\ge 1,
\]
where the last inequality follows from the hypothesis applied to the
interval $[\xi_{2u+1},\xi_{2u+2})$, we obtain that $\tau'-\tau\ge 1$. Therefore, we can conclude that $\boldsymbol{\tau}_u\sqcup\boldsymbol{\tau}_{u'}\subseteq \mathfrak{S}[t_h,t_h+d_{ih})$ for any $u\neq  u'$, $u,u'$ odd, which yields \begin{equation} \label{eq:pf_dp_additive_odd_part}
    \boldsymbol{\tau}^{\mathrm{odd}} \in \mathfrak{S}[t_h,t_h+d_{ih}).
\end{equation}
Following the same argument, we can also conclude that 
\begin{equation} \label{eq:pf_dp_additive_even_part}
    \boldsymbol{\tau}^{\mathrm{even}} \in \mathfrak{S}[t_h,t_h+d_{ih}).
\end{equation}

By Eq.~\eqref{eq:pf_dp_additive_odd_part} and Eq.~\eqref{eq:pf_dp_additive_even_part}, the optimality of $\Delta_{ih}[t_h,t_h+d_{ih})$ implies that
\[
\begin{aligned}
    \Delta_{ih}[t_h,t_h+d_{ih}) \ge \sum_{\substack{u\in[k]\\u\text{ odd}}}
    \sum_{\tau\in \boldsymbol{\tau}_u} g_h(\tau) = \sum_{\substack{u\in[k]\\u\text{ odd}}}
    \Delta_{ih}[\xi_{2u-1},\xi_{2u}),
\end{aligned}
\]
and, similarly,\[
    \Delta_{ih}[t_h,t_h+d_{ih})\ge\sum_{\substack{u\in[k]\\u\text{ even}}}
    \Delta_{ih}[\xi_{2u-1},\xi_{2u}).
\]
Adding the two inequalities yields\[
    2\Delta_{ih}[t_h,t_h+d_{ih})\ge
    \sum_{u=1}^{k}\Delta_{ih}[\xi_{2u-1},\xi_{2u}),
\]
which proves the lemma. \Halmos

\section{Omitted Proof(s) in Section~\ref{se:pf_main_tm_refine}}
\subsection{Proof of Lemma~\ref{le:capa_fea_lb_integral_refine}}\label{se:app_pf_capa_fea_lb_integral_refine}
For notational convenience, let $L_{i,j,\tau}^{\mathrm{GR}}\coloneqq \mathcal L_{ij}^{\mathrm{GR}}[\tau,\tau+1)$ for each $\tau\in\mathcal T[t_j,t_j+d_{ij})$.

Similar to the proof in Section~\ref{sec:app_pf_le_capa_fea_integral}, suppose that Algorithm~\ref{alg:main_d_integer_active} assigns server $i$ to job $j$. By the definition of $\mathcal L_{ij}^{\mathrm{GR}}$, we have that \begin{align*}
    \mathcal L_{ij}^{\mathrm{GR}}=\sum_{\tau\in\mathcal T[t_j,t_j+d_{ij})}L_{i,j,\tau}^{\mathrm{GR}}\ge L_{i,j,t_j}^{\mathrm{GR}}=\max_{\ell \in \{0\} \cup [|\tilde{\boldsymbol{r}}_{i,j,t_j}|]}\left\{\left(\tilde{r}_{i,j,t_j}^{(\ell)} \wedge r_{ij}\right)\Psi\!(1-\ell/c_i)\right\}.
\end{align*}
Recall that $\tilde{r}_{i,j,t_j}^{(0)}=+\infty$ and the elements in $\boldsymbol{\tilde r}_{i,j,t_j}$ are the reward rates of active jobs at time $t_j$. For every job $k$ active at time $t_j$, both $k$ and $j$ are compatible with server $i$, and $D\ge |t_j-t_k|$. By Assumption~\ref{assumption:local-bounded-heterogeneity}, it follows that $r_{ik}\ge r_{ij}/\delta$. Noting that $\delta\ge 1$, consequently, for any $\ell \in \{0\} \cup [|\tilde{\boldsymbol{r}}_{i,j,t_j}|]$, we have $\tilde{r}_{i,j,t_j}^{(\ell)}\ge r_{ij}/\delta$. Hence, we obtain that \begin{align*}
    &\max_{\ell \in \{0\} \cup [|\tilde{\boldsymbol{r}}_{i,j,t_j}|]}\left\{\left(\tilde{r}_{i,j,t_j}^{(\ell)} \wedge r_{ij}\right)\Psi\!(1-\ell/c_i)\right\}\\
    &\qquad\qquad\qquad\ge \max_{\ell \in \{0\} \cup [|\tilde{\boldsymbol{r}}_{i,j,t_j}|]}\left\{\frac{r_{ij}}{\delta}\Psi\!(1-\ell/c_i)\right\}=\frac{r_{ij}}{\delta}\Psi\!(\alpha_{i,j-1\to t_j}).
\end{align*}
Since whenever Algorithm~\ref{alg:main_d_integer_active} assigns $i$ to $j$, the corresponding adjusted reward is positive, we obtain that \[
    r_{ij}d_{ij}>\max_{\ell \in \{0\} \cup [|\tilde{\boldsymbol{r}}_{i,j,t_j}|]}\left\{\left(\tilde{r}_{i,j,t_j}^{(\ell)} \wedge r_{ij}\right)\Psi\!(1-\ell/c_i)\right\}\ge \frac{r_{ij}}{\delta}\Psi\!(\alpha_{i,j-1\to t_j}).
\]
Since $r_{ij}>0$, we can conclude that $
    \delta d_{ij}>\Psi\!(\alpha_{i,j-1\to t_j})$,
which proves the lemma. \Halmos

\subsection{Proof of Lemma~\ref{le:app_pf_dp_charging_refine}}\label{se:app_pf_dp_charging_refine}

By the definition of $\mathcal L_{ij}^{\mathrm{GR}}$ and $\tilde\Delta_{ih}[a,b)$, it is sufficient to prove that, for each $\tau\in\mathcal T[t_j,t_j+d_{ij})$, we have that, 
\begin{align}\label{eq:pf_dp_charging_refine_sufficient_form_1}
    \nonumber&\sum_{h\in\tilde{\mathcal{C}}_i(j)}\max_{\ell \in \{0\} \cup [|\tilde{\boldsymbol{r}}_{i,h,\tau}|]}\left\{\left(\tilde{r}_{i,h,\tau}^{(\ell)} \wedge r_{ih}\right)\left[\Psi\!\left(1-\frac{\ell}{c_i}-(\alpha_{i,h-1\to \tau}-\alpha_{i,h\to \tau})\right)-\Psi\!\left(1-\frac{\ell}{c_i}\right)\right]\right\}\\
    \ge & \max_{\ell \in \{0\} \cup [|\tilde{\boldsymbol{r}}_{i,j,\tau}|]}\left\{\left(\tilde{r}_{i,j,\tau}^{(\ell)} \wedge r_{ij}\right)\Psi\!\left(1-\frac{\ell}{c_i}\right)\right\}.
\end{align}

Fix $\tau\in\mathcal T[t_j,t_j+d_{ij})$. Define $\tilde{\mathcal C}_i(j,\tau)$ to be the set of jobs that arrive before job $j$, are assigned to server $i$ and \textbf{are active at time $\tau$}. Formally, let \[
    \tilde{\mathcal C}_i(j,\tau)\coloneqq\{h:\text{Algorithm~\ref{alg:main_d_integer_active} assign } h \text{ to }i, h<j,\tau<t_h+d_{ih}\}.
\]
By this definition, for any $\tilde{\mathcal{C}}_i(j)$, we have that \[
\alpha_{i,h-1\to\tau}-\alpha_{i,h\to \tau} = \left\{
\begin{array}{ll}
1/c_i, & h\in \tilde{\mathcal{C}}_i(j,\tau) \\
0, & h\in \tilde{\mathcal{C}}_i(j)\setminus\tilde{\mathcal{C}}_i(j,\tau).
\end{array}
\right.
\]
Since $\tilde{\mathcal{C}}_i(j,\tau)\subseteq\tilde{\mathcal{C}}_i(j)$, we obtain that
\begin{align}\label{eq:pf_dp_charging_refine_sufficient_form_2}
    \nonumber&\sum_{h\in\tilde{\mathcal{C}}_i(j)}\max_{\ell \in \{0\} \cup [|\tilde{\boldsymbol{r}}_{i,h,\tau}|]}\left\{\left(\tilde{r}_{i,h,\tau}^{(\ell)} \wedge r_{ih}\right)\left[\Psi\!\left(1-\frac{\ell}{c_i}-(\alpha_{i,h-1\to \tau}-\alpha_{i,h\to \tau})\right)-\Psi\!\left(1-\frac{\ell}{c_i}\right)\right]\right\}\\
    \ge&\sum_{h\in\tilde{\mathcal{C}}_i(j,\tau)}\max_{\ell \in \{0\} \cup [|\tilde{\boldsymbol{r}}_{i,h,\tau}|]}\left\{\left(\tilde{r}_{i,h,\tau}^{(\ell)} \wedge r_{ih}\right)\left[\Psi\!\left(1-\frac{\ell}{c_i}-\frac{1}{c_i}\right)-\Psi\!\left(1-\frac{\ell}{c_i}\right)\right]\right\}.
\end{align}

Combining Eq.~\eqref{eq:pf_dp_charging_refine_sufficient_form_1} with Eq.~\eqref{eq:pf_dp_charging_refine_sufficient_form_2}, it is sufficient to prove that, for any $\ell\in\{0\} \cup [|\tilde{\boldsymbol{r}}_{i,j,\tau}|]=\{0,1,\dots,c_i(1-\alpha_{i,j-1\to \tau})\}$, it holds that \begin{align}\label{eq:pf_dp_charging_refine_sufficient_form_3}
    \sum_{h\in\tilde{\mathcal{C}}_i(j,\tau)}\max_{q \in \{0\} \cup [|\tilde{\boldsymbol{r}}_{i,h,\tau}|]}\left\{\left(\tilde{r}_{i,h,\tau}^{(q)} \wedge r_{ih}\right)\left[\Psi\!\left(1-\frac{q}{c_i}-\frac{1}{c_i}\right)-\Psi\!\left(1-\frac{q}{c_i}\right)\right]\right\}
    \ge\left(\tilde{r}_{i,j,\tau}^{(\ell)} \wedge r_{ij}\right)\Psi\!\left(1-\frac{\ell}{c_i}\right).
\end{align}

Let $h_1<\dots<h_{\ell}$ denote arbitrary $\ell$ active jobs at time $\tau$ with reward rate no smaller than $\tilde{r}_{i,h,\tau}^{(\ell)}$. By the definition of $\tilde{r}_{i,j,\tau}^{(\ell)}$ and $\tilde{\mathcal C}_i(j,\tau)$, and the fact that $\ell\le |\tilde{\boldsymbol{r}}_{i,j,\tau}|$, we can claim that, among all jobs in $\tilde{\mathcal C}_i(j,\tau)$, there exists at least $\ell$ jobs with reward rates no smaller than $\tilde{r}_{i,j,\tau}^{(\ell)}$.

Since $h_1<\dots<h_{\ell}$, we can claim that, for each $q\in[\ell]$, jobs $h_1,\dots,h_{q-1}$ arrive before job $h_q$ and are active at time $\tau$. This yields that $\alpha_{i,h_q-1\to \tau}\le 1-(q-1)/c_i$ for any $q\in[\ell]$, which means $q-1\in\{0,1,\dots,c_i(1-\alpha_{i,h_q-1\to \tau})\}=\{0\} \cup [|\tilde{\boldsymbol{r}}_{i,h_q,\tau}|]$. Hence, we obtain that \begin{align}\label{eq:pf_dp_charging_refine_midstep1}
    \nonumber&\sum_{h\in\tilde{\mathcal{C}}_i(j,\tau)}\max_{q \in \{0\} \cup [|\tilde{\boldsymbol{r}}_{i,h,\tau}|]}\left\{\left(\tilde{r}_{i,h,\tau}^{(q)} \wedge r_{ih}\right)\left[\Psi\!\left(1-\frac{q}{c_i}-\frac{1}{c_i}\right)-\Psi\!\left(1-\frac{q}{c_i}\right)\right]\right\} \\
   \ge& \sum_{q=1}^{\ell} 
    \left(\tilde{r}_{i,h_q,\tau}^{(q-1)} \wedge r_{ih_q}\right)\left[\Psi\!\left(\frac{c_i-(q-1)}{c_i}-\frac{1}{c_i}\right)-\Psi\!\left(\frac{c_i-(q-1)}{c_i}\right)\right].
\end{align}
For each $q\in[\ell]$, noting that $r_{ih_1},\dots,r_{ih_{q-1}}\ge \tilde{r}_{i,j,\tau}^{(\ell)}$, we can claim that, among all the jobs arrive before $h_q$, there are at least $q-1$ jobs that are active at time $\tau$ and have reward rate at least $\tilde{r}_{i,j,\tau}^{(\ell)}$. Moreover, noting that $r_{ih_q}\ge \tilde{r}_{i,j,\tau}^{(\ell)}$ by its definition. Hence, we obtain that 
\begin{equation}\label{eq:pf_dp_charging_refine_midstep2}
    \left(\tilde{r}_{i,h_q,\tau}^{(q-1)} \wedge r_{ih_q}\right)\ge \tilde{r}_{i,j,\tau}^{(\ell)}, \qquad \forall q\in[\ell].
\end{equation}

Combining Eq.~\eqref{eq:pf_dp_charging_refine_midstep1} and Eq.~\eqref{eq:pf_dp_charging_refine_midstep2}, we can conclude that \begin{align*}
    &\sum_{h\in\tilde{\mathcal{C}}_i(j,\tau)}\max_{q \in \{0\} \cup [|\tilde{\boldsymbol{r}}_{i,h,\tau}|]}\left\{\left(\tilde{r}_{i,h,\tau}^{(q)} \wedge r_{ih}\right)\left[\Psi\!\left(1-\frac{q}{c_i}-\frac{1}{c_i}\right)-\Psi\!\left(1-\frac{q}{c_i}\right)\right]\right\} \\
    \ge&\sum_{q=1}^{\ell} 
    \left(\tilde{r}_{i,h_q,\tau}^{(q-1)} \wedge r_{ih_q}\right)\left[\Psi\!\left(\frac{c_i-(q-1)}{c_i}-\frac{1}{c_i}\right)-\Psi\!\left(\frac{c_i-(q-1)}{c_i}\right)\right]\\
    \ge&\sum_{q=1}^{\ell} \tilde{r}_{i,j,\tau}^{(\ell)}
    \left[\Psi\!\left(\frac{c_i-(q-1)}{c_i}-\frac{1}{c_i}\right)-\Psi\!\left(\frac{c_i-(q-1)}{c_i}\right)\right]\\
    =&\tilde{r}_{i,j,\tau}^{(\ell)}\Psi\!\left(1-\frac{\ell}{c_i}\right)\ge \left(\tilde{r}_{i,j,\tau}^{(\ell)} \wedge r_{ij}\right)\Psi\!\left(1-\frac{\ell}{c_i}\right),
\end{align*}
where the second inequality is because $\Psi$ is decreasing. Therefore, we prove Eq.~\eqref{eq:pf_dp_charging_refine_sufficient_form_3} and finish the proof of Lemma~\ref{le:app_pf_dp_charging_refine}. \Halmos

\subsection{Proof of Lemma~\ref{le:pf_dp_additive_refine}}\label{se:app_pf_dp_additive_refine}
We first present a decreasing property of $\tilde\Delta_{ih}[a,b)$. Its proof is deferred to Section~\ref{se:app_pf_dp_additive_refine_delta_de}.
\begin{lemma}\label{le:pf_dp_additive_refine_delta_de}
    Fix any server $i$, job $h$ and a time interval length $d$. For any time interval satisfying $t_2\ge t_1\ge t_h$, we have that \[
        \tilde\Delta_{ih}[t_1,t_1+d)\ge \tilde\Delta_{ih}[t_2,t_2+d).
    \]
\end{lemma}

For simplicity, denote
\begin{align*}
    &\Delta_{i,h,\tau}^{\operatorname{Re}}\coloneqq \tilde\Delta_{ih}[\tau,\tau+1)=\\
    &\max_{l \in \{0\} \cup [|\tilde{\boldsymbol{r}}_{i,h,\tau}|]}\left\{\left(\tilde{r}_{i,h,\tau}^{(\ell)} \wedge r_{ih}\right)\left[\Psi\!\left(1-\frac{\ell}{c_i}-(\alpha_{i,h-1\to \tau}-\alpha_{i,h\to \tau})\right)-\Psi\!\left(1-\frac{\ell}{c_i}\right)\right]\right\},
\end{align*}
then for any $T_1,T_2$ such that $t_h\le T_1<T_2\le t_h+d_{ih}$, we have \begin{align*}
    \tilde\Delta_{ih}[T_1,T_2)=\sum_{\tau\in\mathcal{T}[T_1,T_2)}\Delta_{i,h,\tau}^{\operatorname{Re}}.
\end{align*}

Let $\rho_v=t_h+\lfloor \xi_v-t_h\rfloor$ for any $v\in[2k-1]$, let $\rho_{2k}=\rho_{2k-1}+\xi_{2k}-\xi_{2k-1}$. Since for $u\in[k-1]$, $\xi_{2u}-\xi_{2u-1}$ are positive integers, we have that $\rho_{2u}-\rho_{2u-1}=\xi_{2u}-\xi_{2u-1}$. Noting that $\rho_{2u}\le \xi_{2u} \le t_h+d_{ih}$ and $t_h\le \rho_{2u-1}\le \xi_{2u-1}$, by Lemma~\ref{le:pf_dp_additive_refine_delta_de}, it follows that 
\begin{align}\label{eq:pf_dp_additive_refine_4}
    \sum_{u=1}^{k}\tilde\Delta_{ih}[\xi_{2u-1},\xi_{2u})
    \le \sum_{u=1}^{k}\tilde\Delta_{ih}[\rho_{2u-1},\rho_{2u})= \sum_{u=1}^k \sum_{\tau\in\mathcal T[\rho_{2u-1},\rho_{2u})}\Delta_{i,h,\tau}^{\operatorname{Re}}.
\end{align}
Since $\rho_{2u}-\rho_{2u-1}\in\mathbb N$ for any $u\in[k-1]$, we have that
\begin{equation}\label{eq:pf_dp_additive_refine_1}
    \mathcal T[\rho_{2u-1},\rho_{2u})=\{\rho_{2u-1},\rho_{2u-1}+1,\dots,\rho_{2u}-1\},\qquad \forall u\in[k-1].
\end{equation}
Since $\rho_{2u}\le\rho_{2u+1}$, we have that
\begin{equation}\label{eq:pf_dp_additive_refine_2}
   \mathcal T[\rho_{2u-1},\rho_{2u}) \cap \mathcal T[\rho_{2u^{\prime}-1},\rho_{2u^{\prime}})=\emptyset,\qquad \forall u\neq u^{\prime},\qquad u,u^{\prime}\in[k].
\end{equation}

By the definition of $\rho_{u}$, we have $\rho_{2u-1}\in\mathcal T[t_h,t_h+d_{ih}),\rho_{2u}\le t_h+d_{ih}$ for any $u\in [k]$, which yields that 
\begin{equation}\label{eq:pf_dp_additive_refine_3}
    \mathcal T[\rho_{2u-1},\rho_{2u})\subseteq\mathcal T[t_h,t_h+d_{ih}),\qquad \forall u\in[k].
\end{equation}
Recall that $\rho_{2u+1}\ge\rho_{2u}$, combining Eq.~\eqref{eq:pf_dp_additive_refine_1}, Eq.~\eqref{eq:pf_dp_additive_refine_2}, and Eq.~\eqref{eq:pf_dp_additive_refine_3}, we obtain that \begin{align*}
    \underset{u\in[k]}{\bigsqcup} \mathcal T[\rho_{2u-1},\rho_{2u})\subseteq \mathcal T[t_h,t_h+d_{ih}).
\end{align*}
Since $\Delta_{i,h,\tau}^{\operatorname{Re}}\ge 0$ for any $\tau\in\mathcal T[t_h,t_h+d_{ih})$, combining this result with Eq.~\eqref{eq:pf_dp_additive_refine_4}, we can conclude that \begin{align*}
    \sum_{u=1}^{k}\tilde\Delta_{ih}[\xi_{2u-1},\xi_{2u})
    \le& \sum_{u=1}^k \sum_{\tau\in\mathcal T[\rho_{2u-1},\rho_{2u})}\Delta_{i,h,\tau}^{\operatorname{Re}}=\underset{\tau\in \underset{u\in[k]}{\bigsqcup} \mathcal T[\rho_{2u-1},\rho_{2u})}{\sum}\Delta_{i,h,\tau}^{\operatorname{Re}}\\
    \le& \underset{\tau\in\mathcal T[t_h,t_h+d_{ih})}{\sum}\Delta_{i,h,\tau}^{\operatorname{Re}}=\tilde\Delta_{ih}[t_h,t_h+d_{ih}),
\end{align*}
which finishes the proof of Lemma~\ref{le:pf_dp_additive_refine}. \Halmos

\subsection{Proof of Lemma~\ref{le:pf_dp_additive_refine_delta_de}}\label{se:app_pf_dp_additive_refine_delta_de}
Fix $t_1,t_2,d$ such that $t_2\ge t_1\ge t_h$. By the definition of $\tilde\Delta_{ih}[a,b)$, it is sufficient to prove that, for each $\tau\in\mathcal{T}[t_1,t_1+d)$, it holds that \begin{align}\label{eq:pf_dp_additive_refine_delta_de_4}
    \nonumber&\max_{\ell \in \{0\} \cup [|\tilde{\boldsymbol{r}}_{i,h,\tau}|]}\left\{\left(\tilde{r}_{i,h,\tau}^{(\ell)} \wedge r_{ih}\right)\left[\Psi\!\left(1-\frac{\ell}{c_i}-(\alpha_{i,h-1\to \tau}-\alpha_{i,h\to \tau})\right)-\Psi\!\left(1-\frac{\ell}{c_i}\right)\right]\right\} \\
    \ge & \max_{\ell \in \{0\} \cup [|\tilde{\boldsymbol{r}}_{i,h,\tilde\tau}|]}\left\{\left(\tilde{r}_{i,h,\tilde\tau}^{(\ell)} \wedge r_{ih}\right)\left[\Psi\!\left(1-\frac{\ell}{c_i}-(\alpha_{i,h-1\to \tilde\tau}-\alpha_{i,h\to \tilde\tau})\right)-\Psi\!\left(1-\frac{\ell}{c_i}\right)\right]\right\},
\end{align}
where $\tilde \tau=\tau+t_2-t_1$. Noting that when $\tilde\tau\ge t_h+d_{ih}$, we have $\alpha_{i,h-1\to\tilde\tau}=\alpha_{i,h\to\tilde\tau}$, so Eq.~\eqref{eq:pf_dp_additive_refine_delta_de_4} holds directly. Hence, in the following part, we suppose $\tilde\tau< t_h+d_{ih}$ and prove Eq.~\eqref{eq:pf_dp_additive_refine_delta_de_4} in this case.

For simplicity, we expand the definition of $\tilde{\mathcal C}_i(h,\tau)$ in Section~\ref{le:app_pf_dp_charging_refine}. Formally, fix server $i$ and job $h$, for each $t\ge t_h$, we define \[
    \tilde{\mathcal C}_i(h,t)\coloneqq\{\tilde h:\text{Algorithm~\ref{alg:main_d_integer_active} assign } \tilde h \text{ to }i, \tilde h\le h-1,t<t_{\tilde h}+d_{i\tilde h}\}.
\]
Since $\tilde \tau\ge \tau \ge t_h$, we have that \begin{equation}\label{eq:pf_dp_additive_refine_delta_de_1}
    \tilde {\mathcal C}_i(h,\tilde\tau)\subseteq \tilde {\mathcal C}_i(h,\tau) .
\end{equation}
Recall that \[
    \tilde r_{i,h,t}^{(1)}\ge \tilde r_{i,h,t}^{(2)}\ge\dots\ge \tilde r_{i,h,t}^{(|\tilde{\boldsymbol{r}}_{i,h,t}|)}
\] 
are the reward rates of jobs in $\tilde{\mathcal C}_i(h,t)$ ranked in order. Since $\alpha_{i,h-1\to\tau}\le \alpha_{i,h-1\to\tilde\tau}$, we have \(|\tilde{\boldsymbol{r}}_{i,h,\tilde\tau}|=c_i(1-\alpha_{i,h-1\to\tilde\tau})\le c_i(1-\alpha_{i,h-1\to\tau})=|\tilde{\boldsymbol{r}}_{i,h,\tau}|\), thus Eq.~\eqref{eq:pf_dp_additive_refine_delta_de_1} implies that \begin{equation}\label{eq:pf_dp_additive_refine_delta_de_2}
    \tilde r_{i,h,\tau}^{(\ell)}\ge \tilde r_{i,h,\tilde\tau}^{(\ell)},\qquad \forall \ell\in[|\tilde{\boldsymbol{r}}_{i,h,\tilde\tau}|].
\end{equation}

Noting that for any $t\ge t_h$, we have that \[
\alpha_{i,h-1\to t}-\alpha_{i,h\to t} = \left\{
\begin{array}{ll}
1/c_i, & \text{Algorithm~\ref{alg:main_d_integer_active} assigns } h \text{ to }i,\text{ and } t<t_h+d_{ih} \\
0, & \text{otherwise}.
\end{array}
\right.
\]
Since $t_h\le \tau\le \tilde\tau$, this equality yields that $\alpha_{i,h-1\to\tau}-\alpha_{i,h\to\tau} \ge \alpha_{i,h-1\to\tilde\tau}-\alpha_{i,h\to\tilde\tau}$. Since $\Psi$ is decreasing, we obtain that \begin{align}\label{eq:pf_dp_additive_refine_delta_de_3}
    \Psi\!\left(1-\frac{\ell}{c_i}-(\alpha_{i,h-1\to \tau}-\alpha_{i,h\to \tau})\right)\ge \Psi\!\left(1-\frac{\ell}{c_i}-(\alpha_{i,h-1\to \tilde\tau}-\alpha_{i,h\to \tilde\tau})\right),\ \forall \ell\in\{0\} \cup[|\tilde{\boldsymbol{r}}_{i,h,\tilde\tau}|].
\end{align}
Noting that $\tilde{r}_{i,h,\tau}^{(0)} \wedge r_{ih}=r_{ih}=\tilde{r}_{i,h,\tilde\tau}^{(0)} \wedge r_{ih}$, combining Eq.~\eqref{eq:pf_dp_additive_refine_delta_de_2} and Eq.~\eqref{eq:pf_dp_additive_refine_delta_de_3}, we can conclude that \begin{align*}
    &\max_{\ell \in \{0\} \cup [|\tilde{\boldsymbol{r}}_{i,h,\tau}|]}\left\{\left(\tilde{r}_{i,h,\tau}^{(\ell)} \wedge r_{ih}\right)\left[\Psi\!\left(1-\frac{\ell}{c_i}-(\alpha_{i,h-1\to \tau}-\alpha_{i,h\to \tau})\right)-\Psi\!\left(1-\frac{\ell}{c_i}\right)\right]\right\} \\
    \ge & \max_{\ell \in \{0\} \cup [|\tilde{\boldsymbol{r}}_{i,h,\tilde\tau}|]}\left\{\left(\tilde{r}_{i,h,\tilde\tau}^{(\ell)} \wedge r_{ih}\right)\left[\Psi\!\left(1-\frac{\ell}{c_i}-(\alpha_{i,h-1\to \tilde\tau}-\alpha_{i,h\to \tilde\tau})\right)-\Psi\!\left(1-\frac{\ell}{c_i}\right)\right]\right\},
\end{align*}
which proves Lemma~\ref{le:pf_dp_additive_refine_delta_de}.

\section{Lower Bounds}\label{sec:lower-bound}
This main proof idea in this part follows from \cite{feng2025online}. We first isolate the distributional argument shared by the lower-bound constructions. This argument will be essential in the analysis of the lower bound.
\begin{lemma}[Prefix-instance distribution]
\label{lem:prefix-instance-distribution}
Consider a single server with capacity \(c\) and batches
\(B_0,B_1,\ldots,B_L\), each containing \(c\) identical jobs.
Suppose that jobs from different batches have overlapping
processing intervals and that each job in batch \(B_\ell\) has
total value \(v_\ell\), where
\(0<v_0\le v_1\le\dots\le v_L\).
For each \(k=0,\ldots,L\), let \(I_k\) be the instance in which
batches \(B_0,\ldots,B_k\) are compatible with the server and all
subsequent batches are incompatible. There exists a probability
distribution over \(I_0,\ldots,I_L\) such that every deterministic
fractional online algorithm satisfies
\[
    \frac{\mathbb{E}[\mathrm{OPT}]}
         {\mathbb{E}[\mathrm{ALG}]}
    \ge
    1+\sum_{\ell=0}^{L-1}
    \left(1-\frac{v_\ell}{v_{\ell+1}}\right).
\]
Consequently, the same expression lower-bounds the competitive
ratio of every randomized online algorithm on this family of
instances.
\end{lemma}

\proof{Proof.}
Define probabilities
\[
    p_k=\frac{v_0}{v_k}-\frac{v_0}{v_{k+1}},
    \quad k=0,\ldots,L-1,
    \qquad
    p_L=\frac{v_0}{v_L}.
\]
Because \(v_0\le\dots\le v_L\), these probabilities are
nonnegative and sum to one.

Fix a deterministic fractional online algorithm. Let \(x_\ell\)
denote the total amount of capacity assigned to batch \(B_\ell\)
whenever that batch is compatible. This quantity is well defined
because all instances \(I_k\) with \(k\ge\ell\) have the same
revealed history through batch \(B_\ell\). Since jobs from
different compatible batches overlap, feasibility implies
\(\sum_{\ell=0}^Lx_\ell\le c\).

The probability that batch \(B_\ell\) is compatible is
\[
    \sum_{k=\ell}^Lp_k=\frac{v_0}{v_\ell}.
\]
It follows that
\[
    \mathbb{E}[\mathrm{ALG}]
    =
    \sum_{\ell=0}^L
    x_\ell v_\ell
    \sum_{k=\ell}^Lp_k
    =
    v_0\sum_{\ell=0}^Lx_\ell
    \le cv_0.
\]

On instance \(I_k\), the offline optimum assigns all \(c\) units
of capacity to the last compatible batch, so
\(\mathrm{OPT}(I_k)=cv_k\). Therefore,
\[
    \mathbb{E}[\mathrm{OPT}]
    =
    cv_0
    \left[
        1+\sum_{\ell=0}^{L-1}
        \left(1-\frac{v_\ell}{v_{\ell+1}}\right)
    \right].
\]
The stated bound follows. The randomized lower bound then follows
from Yao's minimax principle \citep{yao1977probabilistic,borodin1998online}.
\Halmos\endproof

\proof{Proof of Theorem~\ref{tm:lb_main}.}
Fix \(M\in\mathbb{N}\), and let \(L=M+\lfloor D\rfloor-1\). Consider one server
with capacity \(c\). At time \(t_\ell=\ell/(L+1)\), batch
\(B_\ell\) of \(c\) identical jobs arrives.

For \(\ell=0,\ldots,M\), each job in \(B_\ell\) has reward rate
\(r_\ell=\delta^{\ell/M}\) and duration \(d_\ell=1\).
For \(k=2,\ldots,\lfloor D\rfloor\), each job in batch \(B_{M+k-1}\) has reward
rate \(\delta\) and duration \(k\). Thus, the corresponding job
values are
\[
    v_\ell=\delta^{\ell/M},
    \quad \ell=0,\ldots,M,
    \qquad
    v_{M+k-1}=\delta k,
    \quad k=2,\ldots,\lfloor D\rfloor.
\]

All durations are integers in \(\{1,\ldots,\lfloor D\rfloor\}\). Moreover, all batches arrive before time one and every job has duration at least
one. Hence jobs from distinct batches overlap. The reward rates
lie in \([1,\delta]\), and all arrivals occur within an interval
of length less than \(1\le D\). Therefore, Assumption~\ref{assumption:local-bounded-heterogeneity} is
satisfied.

Form the prefix instances described in
Lemma~\ref{lem:prefix-instance-distribution}. Applying that lemma
gives
\[
\begin{aligned}
    \frac{\mathbb{E}[\mathrm{OPT}]}
         {\mathbb{E}[\mathrm{ALG}]}
    \ge
    1+M\left(1-\delta^{-1/M}\right)
    +\sum_{k=1}^{\lfloor D\rfloor-1}
    \left(1-\frac{k}{k+1}\right) =
    H_D+M\left(1-\delta^{-1/M}\right).
\end{aligned}
\]
Letting \(M\to\infty\) yields \(H_D+\ln\delta\). Finally,
\(H_D\ge\ln D\), so the bound is at least \(\ln(\delta D)\).
\Halmos\endproof

Theorem~\ref{tm:lb_main} allows the reward
rate to change by a factor \(\delta\) over a short time interval.
We next show that logarithmic lower bounds persist when the
logarithm of the reward rate is Lipschitz-continuous in time.

\begin{theorem}\label{tm:lb_lipschitz} 
Let \(D\in\mathbb{N}\). Suppose Assumption~1 is replaced by the
following restricted condition:
\[
    1\le d_{ij}\le D,
    \qquad
    r_{ij}>0,
\]
and, for every server \(i\) and every pair of jobs \(j_1,j_2\)
compatible with \(i\),
\[
    \max\left\{
        \frac{r_{ij_1}}{r_{ij_2}},
        \frac{r_{ij_2}}{r_{ij_1}}
    \right\}
    \le
    \exp\left(
        \frac{\ln\delta}{D}
        \left|t_{j_1}-t_{j_2}\right|
    \right).
\]
Then every randomized online algorithm has competitive ratio at least
\[
    \max\{1+\ln\delta,H_D\},
\]
even if fractional assignments are allowed. This lower bound
holds using only integer-valued processing durations.
\end{theorem}

\proof{Proof.}
We give two single-server constructions and apply Lemma~\ref{lem:prefix-instance-distribution} to each. The first
construction varies reward rates over time while keeping all
durations fixed at \(D\), and yields the lower bound
\(1+\ln\delta\). The second construction keeps reward rates fixed
and varies integer durations from \(1\) to \(D\), yielding the lower
bound \(H_D\).

\noindent\underline{\bf Reward variation.}
Let \(a=\ln\delta/D\). Fix \(K\ge2\) and
\(\varepsilon\in(0,D)\), and define
\(\Delta=(D-\varepsilon)/(K-1)\). For
\(\ell=0,\ldots,K-1\), batch \(B_\ell\) of \(c\) identical jobs
arrives at time \(t_\ell=\ell\Delta\). Every job has integer
duration \(D\), reward rate \(r_\ell=e^{at_\ell}\), and value
\(v_\ell=Dr_\ell\).

All batches arrive before time \(D\), so jobs from distinct
batches overlap. Moreover,
\[
    \max\left\{
        \frac{r_\ell}{r_m},
        \frac{r_m}{r_\ell}
    \right\}
    =
    e^{a|t_\ell-t_m|},
\]
and hence the Lipschitz condition holds. By
Lemma~\ref{lem:prefix-instance-distribution},
\[
    \frac{\mathbb{E}[\mathrm{OPT}]}
         {\mathbb{E}[\mathrm{ALG}]}
    \ge
    1+(K-1)\left(1-e^{-a\Delta}\right).
\]
Letting \(K\to\infty\) and then \(\varepsilon\downarrow0\)
gives the lower bound \(1+\ln\delta\).

\noindent\underline{\bf Duration variation.}
The claim is immediate when \(D<2\), so suppose \(D\ge2\).
Fix \(\varepsilon\in(0,1)\). For \(k=1,\ldots,\lfloor D\rfloor\), batch
\(B_{k-1}\) of \(c\) identical jobs arrives at time
\[
    t_{k-1}
    =
    \frac{(k-1)\varepsilon}{\lfloor D\rfloor-1}.
\]
Every job in this batch has reward rate one, integer duration
\(k\), and value \(v_{k-1}=k\).

All batches arrive before time one, while every duration is at
least one, so jobs from distinct batches overlap. Since all reward
rates are equal, the Lipschitz condition holds trivially. Applying
Lemma~\ref{lem:prefix-instance-distribution} gives
\[
    \frac{\mathbb{E}[\mathrm{OPT}]}
         {\mathbb{E}[\mathrm{ALG}]}
    \ge
    1+\sum_{k=1}^{\lfloor D\rfloor-1}
    \left(1-\frac{k}{k+1}\right)
    =
    H_D.
\]

Combining the two constructions proves the result.
\Halmos\endproof

\section{A Counterexample for a Natural Pricing Rule}\label{app:failed_candidates}
In this part, we will show that, the candidate algorithm in Remark~\ref{rem:comparison_feng}, which adopts the assignment-specific loss $\mathcal L_{ij}=\sum_{\tau\in\{t_j+k:k\in\mathbb{N},k<d_{ij}\}}\hat r_{i,j,\tau}\Psi\!(\alpha_{i,j-1\to\tau})$, can only achieve a competitive ratio of $\Omega(\delta)$. For simplicity, we refer to this algorithm as ``M-FLB'' (a modification of FLB in \cite{feng2025online}).

\noindent\underline{\bf Construction of the counterexample.} Consider one server with capacity $c$ and suppose $\delta\gg D$. Take $0<\epsilon<1$. At time $0$, an anchor job with reward rate $1/\delta$ and duration $D$ arrives. At time $\epsilon$, $c$ low-reward jobs with reward rate $1$ and duration $D$ arrive. At time $D+\epsilon/2$, $c$ high-reward jobs with reward rate $\delta$ and duration $1$ arrive.

\noindent\underline{\bf Analysis of the competitive ratio of the constructed hard instance.} 
We next show that the M-FLB only has $\Omega(\delta)$ competitive ratio under our constructed counterexample. Suppose that $1\ll c$. Firstly, the M-FLB will accept the anchor job. Then, the M-FLB will accept $\alpha c$ low-reward jobs, where $\alpha$ is determined by solving the following equation \begin{equation}\label{eq:conterexample_low_reward}
    D=\frac{1}{\delta}D\Psi\!\left(1-\alpha-\frac{1}{c}\right).
\end{equation}
Finally, when facing the high-reward job, suppose that M-FLB accept $a$ such jobs. Since there are $\alpha c$ jobs are still active at time $D+\epsilon/2$, by the design of M-FLB, we have \begin{equation}\label{eq:conterexample_high_reward}
    \delta = \Psi\!\left(1-\alpha-a/c\right).
\end{equation}
Combining Eq.~\eqref{eq:conterexample_low_reward} and Eq.~\eqref{eq:conterexample_high_reward}, we obtain that $a=1$. Therefore, the total reward obtained by applying M-FLB is \[
    \frac{D}{\delta}+\alpha cD+a\delta=\frac{D}{\delta}+\delta+\alpha cD\le \frac{D}{\delta}+\delta+cD,
\]
where the last inequality is due to $0\le \alpha\le 1$. Recall that $D\ll \delta$, noting that the optimal policy chooses all of the high-reward jobs, the optimal revenue is at least $\delta c$. Therefore, the competitive ratio of M-FLB is at least $\Omega(\delta/(2D))$ when $c>2\max\{D,\delta\}$, and $\Omega(\delta/(2D))=\Omega(\delta/2)$ by taking $D=1$.

\section{Real-Valued Durations}
\label{app:real-duration}

This appendix extends Algorithm~\ref{alg:main_d_integer_active} and Theorem~\ref{tm:main_theorem_refine} to the setting with real-valued processing durations by following a similar technique as that in \cite{feng2025online}. The algorithm is unchanged except that the inspection-time defined in Section~\ref{subsec:GR-BAL} and the loss in Eq.~\eqref{eq:def-greedy-relaxation-loss} are refined. We present these changes below.

\begin{definition}[A $\gamma$-grid inspection time]\label{df:inspection_time_refine_change}
    Fix two time points $T_1<T_2$, define \[
        \mathcal{T}_{\gamma}[T_1,T_2)\coloneqq \{T_1+l/\gamma:l\in\mathbb{Z}_{\ge 0},T_1+l/\gamma<T_2\}=\{T_1,T_1+1/\gamma,\dots,T_1+\lceil (T_2-T_1)\gamma-1\rceil/\gamma\}.
    \]
\end{definition}

We then define a $\gamma$-grid greedy-relaxation loss on $[T_1,T_2)$ to be \begin{align*}  
\mathcal L_{ij}^{\mathrm{GR},\gamma}[T_1,T_2)
\coloneqq\sum_{\tau\in\mathcal T_{\gamma}(T_1,T_2)}\max_{\ell \in \{0\} \cup [|\tilde{\boldsymbol{r}}_{i,j,\tau}|]}\left\{\left(\tilde r_{i,j,\tau}^{(\ell)}\wedge r_{ij}\right)\Psi\!(1-\ell/c_i)\right\},
\end{align*}
and for a candidate assignment of job $j$ to server $i$, we define the corresponding loss to be \[
    \mathcal L_{ij}^{\mathrm{GR},\gamma}\coloneqq \mathcal L_{ij}^{\mathrm{GR},\gamma}[t_j,t_j+d_{ij}).
\]
We formally define our $\gamma$-grid refined greedy-relaxation BALANCE algorithm ($\gamma$GR-BAL for short), as follows.

\LinesNotNumbered
\begin{algorithm}[H]
\caption{The $\gamma$-grid Greedy-Relaxation BALance Algorithm  ($\gamma$GR-BAL)}\label{alg:main_d_realvalue}
Instantiate the BALANCE framework (Algorithm~\ref{alg:BALANCE}) with
\[
\mathcal{L}_{ij} = \mathcal L_{ij}^{\mathrm{GR},\gamma}\coloneqq \mathcal L_{ij}^{\mathrm{GR},\gamma}[t_j,t_j+d_{ij}) .
\]
\end{algorithm}

\begin{theorem}[Competitive ratio for real-valued durations]\label{thm:real-duration}
Suppose Assumption~\ref{assumption:local-bounded-heterogeneity} holds and processing durations are real-valued. There exists a inspection-frequency scalar $\gamma^*\ge 2$ and penalty parameters
$(\eta^*,\beta^*)$ such that, as $c_{\min}:=\min_i c_i\to\infty$, Algorithm~\ref{alg:main_d_realvalue} is
\[
\ln(\delta D)+2\ln\ln(\delta\vee D)+\mathcal O(1)
\]
competitive. When $\delta\vee D$ is bounded by some constant, the bound is interpreted as $\mathcal O(1)$.
\end{theorem}

The proof of Theorem~\ref{thm:real-duration} follows a very similar streaming line as that in Section~\ref{se:pf_main_tm_refine}, except for Step~2 in Section~\ref{subse:pf_pf_value_LP_2_refine}.

To prove Theorem~\ref{thm:real-duration}, similar to Section~\ref{se:pf_main_tm_refine}, we first give the following capacity-feasibility result, whose proof is the same as that in Section~\ref{se:pf_capa_fea_refine}.

\begin{proposition}[Feasibility of Algorithm~\ref{alg:main_d_realvalue}]\label{prop:capa_fea_realvalue}
Recall that $\Theta$ is the admissible parameter set defined in Definition~\ref{df:capa_fea}. Fix $\eta$ and $\beta$, if $(\beta,\eta)\in\Theta$, then Algorithm~\ref{alg:main_d_realvalue} never violates the server-capacity constraints.
\end{proposition}

In the remainder of Section~\ref{app:real-duration}, we will proceed the proof of Theorem~\ref{thm:real-duration} in two steps:  Section~\ref{se:pf_value_LP_realvalue} derives a competitive-ratio bound for fixed admissible parameters, and Section~\ref{se:pf_choice_beta_eta_realvalue} chooses $(\gamma,\beta,\eta)$ and completes the proof.

\subsection{A Primal-Dual Bound for Algorithm~\ref{alg:main_d_realvalue} under Fixed Admissible Parameters}\label{se:pf_value_LP_realvalue}
In this subsection, we aim to prove the following performance guarantee for Algorithm~\ref{alg:main_d_realvalue}.
\begin{proposition}\label{prop:pf_value_LP_realvallue}
    Suppose that $(\beta,\eta)\in\Theta$ and $\gamma\ge 2$. Then the competitive ratio of Algorithm~\ref{alg:main_d_realvalue} is upper bounded by\[
        1+\frac{\gamma}{\gamma-1}(1+\eta(\gamma+1))\beta^{1/c_{\min}}\ln \beta.
    \]
\end{proposition}

Following the similar lines in Section~\ref{se:pf_value_LP_refine}, we carry out a primal-dual analysis in three steps. First, we construct dual variables $\{\bar\lambda_j,\bar\theta_i\}$ from the trajectory of Algorithm~\ref{alg:main_d_realvalue}. Second, in Proposition~\ref{prop:pf_value_LP_relate_objectivedual_to_alg_refine_realval}, we compare the dual objective value induced by $\{\bar\lambda_j,\bar\theta_i\}$ to the total reward collected by the algorithm. Finally, in Proposition~\ref{prop:pf_value_LP_dual_fea_refine_realval}, we prove that the constructed dual variables $\{\bar\lambda_j,\bar\theta_i\}$ are feasible; weak duality then yields the desired competitive-ratio bound.

\noindent\underline{\bf Construction of dual variables.} 
The dual variables $\{\bar\lambda_j,\bar\theta_i\}$ are constructed in a manner similar to those in Section~\ref{se:pf_value_LP_refine}, with several components modified to reflect the refined loss. For each job $j$ and each server $i$, we set
\[\bar\lambda_j:=\max\left(0,\max_{i\in\mathcal N(j)}
    \left\{
        r_{ij}d_{ij}
        -
       \mathcal L_{ij}^{\mathrm{GR},\gamma}\right\}\right), \qquad \text{and} \qquad 
\bar{\theta}_i:=\sum_{j=1}^{m}\bar\vartheta_{ij},
\]
where $\bar\vartheta_{ij}$ is the incremental component corresponding to the duration when the algorithm processes job $j$. To specify these incremental components, we now introduce the following definition.
\begin{definition}[A Changed Capacity-Induced Refined-Loss Increment]
\label{def:pathoc-increment_refine_realval}
Fix a compatible pair $(i,j)\in E$ and a time interval satisfying $ t_j\le T_1<T_2$, define $\bar\Delta_{ij}[T_1,T_2)\coloneqq$
\begin{align*}
\sum_{\tau\in\mathcal T_{\gamma}[T_1,T_2)}\max_{\{0\} \cup[|\boldsymbol{\tilde r}_{i,j,\tau}|]}\Biggl\{\left(\tilde{r}_{i,j,\tau}^{(\ell)}\wedge r_{ij}\right) \times\Biggl[\Psi\!\left(1-\frac{\ell}{c_i}-\bigl(\alpha_{i,j-1\to\tau}-\alpha_{i,j\to\tau}\bigr)\right)-\Psi\!\left(1-\frac{\ell}{c_i}\right)\Biggr]\Biggr\}.
\end{align*}
For $(i,j)\notin E$, define $\bar\Delta_{ij}[T_1,T_2)=0$ for any $T_1<T_2$.
\end{definition}

The remaining part is the same as that in Section~\ref{se:pf_value_LP_refine}. We then define the incremental components $\bar\vartheta_{ij}$ based on the capacity-induced refined-loss increment:
\[
    \bar\vartheta_{ij}\coloneqq \frac{\gamma}{\gamma-1}\bar\Delta_{ij}[t_j,t_j+d_{ij})
        \cdot \mathbb{I}[\text{Algorithm~\ref{alg:main_d_realvalue} assigns job $j$ to server $i$}]=\frac{\gamma}{\gamma-1}\bar\Delta_{ij}[t_j,t_j+d_{ij}).
\]
Since $\alpha_{i,j\to t}\le \alpha_{i,j-1\to t}$ and $\Psi$ is decreasing, every $\bar\vartheta_{ij}$ is nonnegative. Hence,
$\bar{\lambda}_j\ge 0$ and $\bar{\theta}_i\ge 0$ for all $j$ and $i$.
The resulting variables $\{\bar{\lambda}_j,\bar{\theta}_i\}$ form the candidate dual solution induced by the trajectory of Algorithm~\ref{alg:main_d_realvalue}.

\noindent\underline{\bf Objective bound and feasibility of $\{\bar\lambda_j,\bar\theta_i\}$.}
Let $\widebar{\mathrm{ALG}}$ denote the total reward collected by Algorithm~\ref{alg:main_d_realvalue}, and let $\mathrm{Dual}(\boldsymbol{\bar\lambda},\boldsymbol{\bar\theta})$ denote the dual objective evaluated at the constructed variables, where $\boldsymbol{\bar \lambda}=\{\bar \lambda_j\}_j$ and $\boldsymbol{\bar \theta}=\{\bar \theta_i\}_i$. The following Proposition~\ref{prop:pf_value_LP_relate_objectivedual_to_alg_refine_realval} bounds $\mathrm{Dual}(\boldsymbol{\bar\lambda},\boldsymbol{\bar\theta})$ in terms of $\widebar{\mathrm{ALG}}$, and Proposition~\ref{prop:pf_value_LP_dual_fea_refine_realval} establishes the feasibility for $\{\bar\lambda_j,\bar\theta_i\}$. Their proofs are provided in Sections~\ref{subse:pf_pf_value_LP_1_refine_realval} and Section~\ref{subse:pf_pf_value_LP_2_refine_realval}, respectively.

\begin{proposition}[Objective value]
\label{prop:pf_value_LP_relate_objectivedual_to_alg_refine_realval}
$\mathrm{Dual}(\boldsymbol{\bar{\lambda}},\boldsymbol{\bar{\theta}})
\le
\left[ 1+\frac{\gamma}{\gamma-1}(1+\eta(\gamma+1))\beta^{1/c_{\min}}\ln \beta\right]\cdot\widebar{\mathrm{ALG}}$.
\end{proposition}

\begin{proposition}[Dual feasibility]\label{prop:pf_value_LP_dual_fea_refine_realval}
When $\gamma\ge 2$, the constructed dual variables $\{\bar\lambda_j,\bar\theta_i\}$ are feasible.
\end{proposition}

Combining Claim~\ref{claim:benchmark}, Propositions~\ref{prop:pf_value_LP_relate_objectivedual_to_alg_refine_realval} and ~\ref{prop:pf_value_LP_dual_fea_refine_realval}, we establish the desired competitive-ratio bound: \begin{align*}
    \left[ 1+\frac{\gamma}{\gamma-1}(1+\eta(\gamma+1))\beta^{1/c_{\min}}\ln \beta\right]\cdot \widebar{\operatorname{ALG}}\ge\mathrm{Dual}(\boldsymbol{\bar{\lambda}},\boldsymbol{\bar{\theta}})\ge P_{\mathrm{OPT}}\ge \operatorname{OPT}.
\end{align*}
This finishes the proof of Proposition~\ref{prop:pf_value_LP_realvallue}.

\subsubsection{Proof of Proposition~\ref{prop:pf_value_LP_relate_objectivedual_to_alg_refine_realval}} \label{subse:pf_pf_value_LP_1_refine_realval}

The proof of Proposition~\ref{prop:pf_value_LP_relate_objectivedual_to_alg_refine_realval} follows the similar lines in Section~\ref{subse:pf_pf_value_LP_1_refine}. Specifically, let $\widebar{\operatorname{ALG}}_j=\sum_{i\in \mathcal N(j)}r_{ij}d_{ij}\cdot\mathbb{I}[\text{Algorithm~\ref{alg:main_d_realvalue} assigns $j$ to $i$}]$, $\widebar{\operatorname{Dual}}_j=\bar\lambda_j+\sum_ic_i\bar\vartheta_{ij}$, and $\widebar\Gamma=1+\frac{\gamma}{\gamma-1}(1+\eta(\gamma+1))\beta^{1/c_{\min}}\ln \beta$. It suffices to show that, for every arriving job $j$, we have
\begin{equation}\label{eq:pf_competitive_ratio_dual_alg_refine_realval}
    \widebar{\operatorname{Dual}}_j\le \widebar\Gamma\cdot \widebar{\operatorname{ALG}}_j.
\end{equation}

Following the same streaming line as that in Section~\ref{subse:pf_pf_value_LP_1_refine}, we still only need to consider the case when Algorithm~\ref{alg:main_d_realvalue} assigns job $j$ to some server $i$. In this case, we obtain that \begin{align}
    \nonumber c_i\bar{\vartheta}_{ij}&  \le \frac{\gamma}{\gamma-1}\left(\beta^{1/c_i}\ln\beta\cdot\mathcal L_{ij}^{\mathrm{GR},\gamma}+\eta\beta^{1/c_i}\ln\beta \sum_{\tau\in \mathcal{T}_{\gamma}[t_j,t_j+d_{ij})}\max_{\{0\} \cup[|\boldsymbol{\tilde r}_{i,j,\tau}|]}\left\{\left(\tilde{r}_{i,j,\tau}^{(\ell)}\wedge r_{ij}\right)\right\}\right)\\
        & \le \frac{\gamma}{\gamma-1}\left(\beta^{1/c_i}\ln\beta\cdot\mathcal L_{ij}^{\mathrm{GR},\gamma}+\eta(\gamma+1)\beta^{1/c_i}\ln\beta r_{ij}d_{ij}\right),
\end{align}
where the last inequality due to the fact that $|\mathcal{T}_{\gamma}[t_j,t_j+d_{ij})|\le (\gamma+1)d_{ij}$.

Moreover, since $\mathcal L_{ij}^{\mathrm{GR},\gamma}\le r_{ij}d_{ij}$ when Algorithm~\ref{alg:main_d_realvalue} assigns job $j$ to some server $i$, we can conclude that \begin{align*}
    \widebar{\operatorname{Dual}}_j=\bar{\lambda}_j+c_i\bar{\vartheta}_{ij} &\le
    r_{ij}d_{ij}+\frac{\gamma}{\gamma-1}\beta^{1/c_i}\ln\beta\left(\mathcal L_{ij}^{\mathrm{GR},\gamma}+\eta(\gamma+1) r_{ij}d_{ij}\right)\\
    &\le \left[1+\frac{\gamma}{\gamma-1}(1+\eta(\gamma+1))\beta^{1/c_{\min}}\ln \beta\right]r_{ij}d_{ij}=\widebar\Gamma r_{ij}d_{ij}=\widebar\Gamma\cdot\widebar{\operatorname{ALG}}_j,
\end{align*}
which completes the proof of Proposition~\ref{prop:pf_value_LP_relate_objectivedual_to_alg_refine_realval}. \Halmos

\subsubsection{Proof of Proposition~\ref{prop:pf_value_LP_dual_fea_refine_realval}} \label{subse:pf_pf_value_LP_2_refine_realval}
Following the similar discussions in Section~\ref{subse:pf_pf_value_LP_2_refine}, to establish the dual feasibility of $\{\bar\lambda_j,\bar\theta_i\}$, it suffices to prove that, for every server $i$ and every feasible configuration $S\in\mathcal S_i$, it holds that \begin{equation}\label{eq:dual_fea_thetai_lb_refine_realval}
    \bar\theta_i\geq \sum_{j\in S}\mathcal L_{ij}^{\mathrm{GR},\gamma}.
\end{equation}

Fix a server $i$ and a feasible configuration $S\in\mathcal S_i$, we prove Eq.~\eqref{eq:dual_fea_thetai_lb_refine_realval}. The proof follows a similar two-step argument as in  Section~\ref{subse:pf_pf_value_LP_2_refine}.

\noindent\underline{\bf Step 1: Local charging.}
Following Section~\ref{subse:pf_pf_value_LP_2_refine}, we first define a charging set \[
    \bar{\mathcal C}_i(j)\coloneqq \{h: \text{Algorithm~\ref{alg:main_d_realvalue} assigns } h \text{ to } i,h< j, t_j<t_h+d_{ih} \}.
\]
Then, we obtain the following lemma.

\begin{lemma}[Local charging]\label{le:app_pf_dp_charging_refine_realval}
    Fix a server $i$ and a job $j\in S\subseteq \mathcal{S}_i$, we have that \begin{align*}
        \sum_{h\in \bar{\mathcal C}_i(j)}\bar \Delta_{ih}[t_j,\left(t_j+d_{ij}\right)\wedge\left(t_{h}+d_{ih}\right)\}) \ge \mathcal L_{ij}^{\mathrm{GR},\gamma}.
    \end{align*}
\end{lemma}

The proof of Lemma~\ref{le:app_pf_dp_charging_refine_realval} is the same as that of Lemma~\ref{le:app_pf_dp_charging_refine} by changing $\mathcal T$ to $\mathcal T_{\gamma}$, so we omit it here. Applying Lemma~\ref{le:app_pf_dp_charging_refine_realval} to every $j\in S$ and summing over $j\in S$ yields 
\begin{align}
    \sum_{j\in S}\mathcal L_{ij}^{\mathrm{GR},\gamma}&\le\sum_{j\in S}\sum_{h\in\bar{\mathcal C}_i(j)}\bar\Delta_{ih}\left[t_j,\left(t_j+d_{ij}\right)\wedge\left(t_{h}+d_{ih}\right)\right)\nonumber \\
    &\le\sum_{h=1}^m\sum_{j\in S:\bar{\mathcal C}_i(j)\ni h} \bar\Delta_{ih}\left[t_j,\left(t_j+d_{ij}\right)\wedge\left(t_{h}+d_{ih}\right)\right). \label{eq:local-charging-sum_refine_realval}
\end{align}

\noindent\underline{\bf Step 2: Aggregation via sub-additivity.}
As discussed in Step~2 in Section~\ref{subse:pf_pf_value_LP_2_refine}, given Eq.~\eqref{eq:local-charging-sum_refine_realval}, the dual-feasibility constraints follow once we establish that
\begin{equation}\label{eq:pf_dp_additive_111_refine_realval}
    \sum_{j\in S:\bar{\mathcal C}_i(j)\ni h} \bar\Delta_{ih}\left[t_j,\left(t_j+d_{ij}\right)\wedge\left(t_{h}+d_{ih}\right)\right)\le \bar\vartheta_{ih},\qquad \forall h\in[m].
\end{equation}

It therefore remains to establish Eq.~\eqref{eq:pf_dp_additive_111_refine_realval}, for which we introduce the following weak sub-additivity property. Its proof is provided in Section~\ref{se:app_pf_dp_additive_refine_realval}.
\begin{lemma}[Weak sub-additivity of refined-loss increments]\label{le:pf_dp_additive_refine_realval}
    Fix job $h$ and any sequence \[
        t_h\le \xi_1<\xi_2\le \xi_3<\xi_4\le\dots\le\xi_{2k-1}<\xi_{2k}\le t_h+d_{ih}.
    \]
    Suppose that $\gamma\ge 2$ and $\xi_{2u}-\xi_{2u-1}\ge 1$ for any $u\in[k-1]$, we have that \begin{align*}
        \bar\vartheta_{ih}=\frac{\gamma}{\gamma-1}\bar\Delta_{ih}[t_h,t_h+d_{ih}) \ge \sum_{u=1}^{k}\bar\Delta_{ih}[\xi_{2u-1},\xi_{2u}).
    \end{align*}
\end{lemma}

Next, we apply Lemma~\ref{le:pf_dp_additive_refine_realval} to prove Eq.~\eqref{eq:pf_dp_additive_111_refine_realval}. For notational convenience, we fix $h\in[m]$, and write \[
        \{j_1,\dots,j_w\}=\{j\in S:h\in\bar{\mathcal C}_i(j)\},\qquad j_1<j_2<\dots<j_w.
\]
Noting that when $w=0$, Eq.~\eqref{eq:pf_dp_additive_111_refine_realval} holds directly. Suppose that $w\ge 1$, we define $b_u\coloneqq \left(t_{j_u}+d_{ij_u}\right)\wedge\left(t_h+d_{ih}\right)$ for $u\in[w]$, and show that $\{t_{j_1},b_{j_1},\dots,t_{j_w},b_{j_w}\}$ satisfies the conditions in Lemma~\ref{le:pf_dp_additive_refine_realval}.
As in Step~2 in Section~\ref{subse:pf_pf_value_LP_2}, it holds that $b_{j_u}=t_{j_u}+d_{i{j_u}}$ for any $u\in[w-1]$, and we have that $
    t_h\le t_{j_1}<b_{j_1}\le\dots\le t_{j_w}<b_{j_w}\le t_h+d_{ih}$.
Since $d_{ij}\ge 1$ for any $i,j$, for each $u\in[w-1]$, it holds that $    b_{j_u}-t_{j_u}=t_{j_u}+d_{ij_{u}}-t_{j_u}=d_{ij_{u}}\ge 1$.
Therefore, we can apply Lemma~\ref{le:pf_dp_additive_refine_realval} to the time points $\{t_{j_1},b_{j_1},\dots,t_{j_w},b_{j_w}\}$,
which yields that \begin{align*}
    \bar\vartheta_{ih}=\frac{\gamma}{\gamma-1}\bar\Delta_{ih}[t_h,t_h+d_{ih})\ge \sum_{j\in S:\bar{\mathcal C}_i(j)\ni h}\bar\Delta_{ih}[t_j,\left(t_j+d_{ij}\right)\wedge\left(t_{h}+d_{ih}\right)).
\end{align*}
This proves Eq.~\eqref{eq:pf_dp_additive_111_refine_realval}, and we finish the proof of Proposition~\ref{prop:pf_value_LP_dual_fea_refine_realval}. \Halmos

\subsection{Putting Things Together}\label{se:pf_choice_beta_eta_realvalue}
We now complete the proof of Theorem~\ref{thm:real-duration}. When $\ln (\max\{\delta,D\})\ge e-1$, we define $\bar L\coloneqq 1+\ln(\delta\vee D)$ and take
\[
    \gamma=\gamma^*= \bar L,\qquad
    \eta=\eta^*=\bar L^{-2},\qquad
    \beta=\beta^*=1+\frac{\delta D}{\eta^*},
\]
so that \((\beta,\eta)\in\Theta\) and \(\gamma\ge 2\). Hence, Proposition~\ref{prop:capa_fea_realvalue} guarantees that Algorithm~\ref{alg:main_d_realvalue} is capacity feasible. Since \[
    \ln\beta=\ln(\delta D)+2\ln\bar L+\mathcal O(1), \qquad \text{and} \qquad \frac{\gamma}{\gamma-1}\left(1+\eta(\gamma+1)\right)
    =\frac{\bar L}{\bar L-1}+\frac{\bar L+1}{\bar L(\bar L-1)}
    =1+\mathcal O\left(\frac1{\bar L}\right),
\]
by Proposition~\ref{prop:pf_value_LP_realvallue}, as $c_{\min}\to\infty$, the competitive ratio of Algorithm~\ref{alg:main_d_realvalue} is upper bounded by \[
    1+\frac{\gamma}{\gamma-1}(1+\eta(\gamma+1))\beta^{1/c_{\min}}\ln \beta
    \le\ln(\delta D)+2\ln\bar L+\mathcal O(1)=\ln(\delta D)+2\ln\ln(\delta\vee D)+\mathcal O(1),
\]
where the first inequality due to $\ln(\delta D)\le 2\bar L$.

We next consider the remaining case where $\bar L=1+\ln(\delta\vee D)<e$. In this regime, we use a constant-regime parameter choice. We set
\[ \gamma=2,\qquad
\beta=\beta^*=2(\delta D+1),
\qquad \text{and}\qquad 
\eta=\eta^*=1,
\]
so that $(\beta,\eta)\in\Theta$.
By Proposition~\ref{prop:pf_value_LP_realvallue}, the competitive ratio of Algorithm~\ref{alg:main_d_realvalue} is upper bounded by
$1+2(1+3\eta)\beta^{1/c_{\min}}\ln\beta$.
Letting $c_{\min}\to\infty$, since $\ln(\delta\vee D)<e-1$, we can conclude that Algorithm~\ref{alg:main_d_realvalue} is asymptotically $\mathcal O(1)$-competitive in this case.
\Halmos\endproof

\subsection{Proof of Lemma~\ref{le:pf_dp_additive_refine_realval}}\label{se:app_pf_dp_additive_refine_realval}
To proof of Lemma~\ref{le:pf_dp_additive_refine_realval} can be derived by combining Lemma~\ref{le:pf_dp_additive_refine_realval_1} and Lemma~\ref{le:app_realval_duration_proof_feasibility_smalllemma}, which we present in the following part.
\begin{lemma}\label{le:app_realval_duration_proof_feasibility_smalllemma}
    For server $i$, job $h$, and any time interval satisfying $[a,b)\subseteq[t_h,t_h+d_{ih})$ and $b-a\ge 1$, we have \[
        \bar\Delta_{ih}[a,b)\le \frac{\gamma}{\gamma-1}\bar\Delta_{ih}[a,a+\lfloor (b-a)\gamma\rfloor/\gamma)
    \]
\end{lemma}
\proof{Proof.}
    We first present a decreasing property of $\bar\Delta_{ih}[a,b)$. Since its proof is the same as that in Section~\ref{se:app_pf_dp_additive_refine_delta_de}, we omit it here.
    \begin{lemma}\label{le:pf_dp_additive_refine_delta_de_realval}
        Fix any server $i$, job $h$ and a time interval length $d$. For any time interval satisfying $t_2\ge t_1\ge t_h$, we have that $
        \bar\Delta_{ih}[t_1,t_1+d)\ge \bar\Delta_{ih}[t_2,t_2+d)$.
    \end{lemma}
    For simplicity, denote $\bar\Delta_{i,h,\tau}^{\operatorname{Re}}\coloneqq \bar\Delta_{ih}[\tau,\tau+1/\gamma)=$
    \begin{align*}
        \max_{l \in \{0\} \cup [|\tilde{\boldsymbol{r}}_{i,h,\tau}|]}\left\{\left(\tilde{r}_{i,h,\tau}^{(\ell)} \wedge r_{ih}\right)\left[\Psi\!\left(1-\frac{\ell}{c_i}-(\alpha_{i,h-1\to \tau}-\alpha_{i,h\to \tau})\right)-\Psi\!\left(1-\frac{\ell}{c_i}\right)\right]\right\}.
    \end{align*}
    Then, for every $T_2\ge T_1\ge t_h$, set $d=1/\gamma$, we obtain the following corollary:\begin{equation}\label{eq:le:app_realval_duration_proof_feasibility_smalllemma_1}
        \bar\Delta_{i,h,T_1}^{\operatorname{Re}}=\bar\Delta_{ih}[T_1,T_1+1/\gamma)\ge \bar\Delta_{ih}[T_2,T_2+1/\gamma)=\bar\Delta_{i,h,T_2}^{\operatorname{Re}}.
    \end{equation}
    Combining Eq.~\eqref{eq:le:app_realval_duration_proof_feasibility_smalllemma_1} with the fact that $\mathcal T_{\gamma}[a,b)\subseteq\mathcal T_{\gamma}[a,a+\lfloor (b-a)\gamma\rfloor/\gamma)\cup\{a+\lfloor (b-a)\gamma\rfloor/\gamma\}$, we obtain that \begin{align*}
        \bar\Delta_{ih}[a,b)&=\sum_{\tau\in\mathcal{T}_{\gamma}[a,b)} \bar\Delta_{i,h,\tau}^{\operatorname{Re}}\le\sum_{\tau\in\mathcal{T}_{\gamma}[a,a+\lfloor (b-a)\gamma\rfloor/\gamma)} \bar\Delta_{i,h,\tau}^{\operatorname{Re}}+\bar\Delta_{i,h,a+\lfloor (b-a)\gamma\rfloor/\gamma}^{\operatorname{Re}}\\
        &\le  \sum_{\tau\in\mathcal{T}_{\gamma}[a,a+\lfloor (b-a)\gamma\rfloor/\gamma)} \bar\Delta_{i,h,\tau}^{\operatorname{Re}}+\frac{1}{\lfloor (b-a)\gamma\rfloor}\sum_{\tau\in\mathcal{T}_{\gamma}[a,a+\lfloor (b-a)\gamma\rfloor/\gamma)} \bar\Delta_{i,h,\tau}^{\operatorname{Re}} \\
        &\le \frac{\lfloor (b-a)\gamma\rfloor+1}{\lfloor (b-a)\gamma\rfloor}\sum_{\tau\in\mathcal{T}_{\gamma}[a,a+\lfloor (b-a)\gamma\rfloor/\gamma)} \bar\Delta_{i,h,\tau}^{\operatorname{Re}} \\
        &\le \frac{\gamma}{\gamma-1}\bar\Delta_{ih}[a,a+\lfloor (b-a)\gamma\rfloor/\gamma),
    \end{align*}
    where the second inequality due to Eq.~\eqref{eq:le:app_realval_duration_proof_feasibility_smalllemma_1} and $\tau\le a+\lfloor (b-a)\gamma\rfloor/\gamma$ for any $\tau\in\mathcal{T}_{\gamma}[a,a+\lfloor (b-a)\gamma\rfloor/\gamma)$, while the last inequality due to $\lfloor (b-a)\gamma\rfloor\ge \gamma-1$ and $\gamma\ge 2$.
\Halmos\endproof
We next present a technical lemma.
\begin{lemma}[Sub-additivity of refined-loss increments under regular conditions]\label{le:pf_dp_additive_refine_realval_1}
    Fix job $h$ and any sequence \[
        t_h\le \zeta_1<\zeta_2\le \zeta_3<\zeta_4\le\dots\le\zeta_{2k-1}<\zeta_{2k}\le t_h+d_{ih}.
    \]
    Suppose that $\zeta_{2u}-\zeta_{2u-1}\in (1/\gamma) \mathbb Z_{\ge 1}$ for any $u\in[k-1]$, we have that \begin{align*}
        \bar\Delta_{ih}[t_h,t_h+d_{ih}) \ge \sum_{u=1}^{k}\bar\Delta_{ih}[\zeta_{2u-1},\zeta_{2u}).
    \end{align*}
\end{lemma}
The proof of Lemma~\ref{le:pf_dp_additive_refine_realval_1} is deferred to Section~\ref{subse:pf_dp_additive_refine_realval_1}.

Next we apply Lemma~\ref{le:pf_dp_additive_refine_realval_1} and Lemma~\ref{le:app_realval_duration_proof_feasibility_smalllemma} to prove Lemma~\ref{le:pf_dp_additive_refine_realval}. Take $[\zeta_{2u-1},\zeta_{2u})=[\xi_{2u-1},\xi_{2u-1}+\lfloor (\xi_{2u}-\xi_{2u-1})\gamma\rfloor/\gamma)$ for each $u\in[k-1]$ and take $[\zeta_{2k-1},\zeta_{2k})=[\xi_{2k-1},\xi_{2k})$, by Lemma~\ref{le:pf_dp_additive_refine_realval_1}, we obtain that \begin{equation}\label{eq:pf_dp_additive_refine_realval_1}
    \bar\Delta_{ih}[t_h,t_h+d_{ih})\ge \sum_{u=1}^{k-1}\bar\Delta_{ih}[\xi_{2u-1},\xi_{2u-1}+\lfloor (\xi_{2u}-\xi_{2u-1})\gamma\rfloor/\gamma)+ \bar\Delta_{ih}[\xi_{2k-1},\xi_{2k}).
\end{equation}
By Lemma~\ref{le:app_realval_duration_proof_feasibility_smalllemma}, for each $u\in[k-1]$, we have that \begin{equation}\label{eq:pf_dp_additive_refine_realval_2}
    \bar\Delta_{ih}[\xi_{2u-1},\xi_{2u-1}+\lfloor (\xi_{2u}-\xi_{2u-1})\gamma\rfloor/\gamma)\ge \frac{\gamma-1}{\gamma}\bar\Delta_{ih}[\xi_{2u-1},\xi_{2u}).
\end{equation}
Combining Eq.~\eqref{eq:pf_dp_additive_refine_realval_1} and Eq.~\eqref{eq:pf_dp_additive_refine_realval_2}, we can conclude that  \begin{align*}
    \frac{\gamma}{\gamma-1}\bar\Delta_{ih}[t_h,t_h+d_{ih})&\ge \frac{\gamma}{\gamma-1}\sum_{u=1}^{k-1}\bar\Delta_{ih}[\xi_{2u-1},\xi_{2u-1}+\lfloor (\xi_{2u}-\xi_{2u-1})\gamma\rfloor/\gamma)+ \bar\Delta_{ih}[\xi_{2k-1},\xi_{2k})\\
    &\ge \sum_{u=1}^{k}\bar\Delta_{ih}[\xi_{2u-1},\xi_{2u}),
\end{align*}
which proves Lemma~\ref{le:pf_dp_additive_refine_realval}. \Halmos

\subsection{Proof of Lemma~\ref{le:pf_dp_additive_refine_realval_1}}\label{subse:pf_dp_additive_refine_realval_1}

The proof of Lemma~\ref{le:pf_dp_additive_refine_realval_1} follows the same stream line as that in Section~\ref{se:app_pf_dp_additive_refine} except for some small details. We present it here.

Recall that in the proof of Lemma~\ref{le:app_realval_duration_proof_feasibility_smalllemma}, we define
\begin{align*}
    \bar\Delta_{i,h,\tau}^{\operatorname{Re}}&= \bar\Delta_{ih}[\tau,\tau+1/\gamma)\\
    &=\max_{l \in \{0\} \cup [|\tilde{\boldsymbol{r}}_{i,h,\tau}|]}\left\{\left(\tilde{r}_{i,h,\tau}^{(\ell)} \wedge r_{ih}\right)\left[\Psi\!\left(1-\frac{\ell}{c_i}-(\alpha_{i,h-1\to \tau}-\alpha_{i,h\to \tau})\right)-\Psi\!\left(1-\frac{\ell}{c_i}\right)\right]\right\},
\end{align*}
then for any $T_1,T_2$ such that $t_h\le T_1< T_2$, we have \begin{align*}
    \bar\Delta_{ih}[T_1,T_2)=\sum_{\tau\in\mathcal{T}_{\gamma}[T_1,T_2)}\bar\Delta_{i,h,\tau}^{\operatorname{Re}}.
\end{align*}

Let $\rho_v=t_h+\lfloor (\zeta_v-t_h)\gamma\rfloor/\gamma$ for any $v\in[2k-1]$, let $\rho_{2k}=\rho_{2k-1}+\zeta_{2k}-\zeta_{2k-1}$. Since for $u\in[k-1]$, $\gamma(\zeta_{2u}-\zeta_{2u-1})$ are positive integers, we have that $\rho_{2u}-\rho_{2u-1}=\zeta_{2u}-\zeta_{2u-1}$. Noting that $t_h\le \rho_{2u-1}\le \zeta_{2u-1}$, by Lemma~\ref{le:pf_dp_additive_refine_delta_de_realval}, it follows that 
\begin{align}\label{eq:pf_dp_additive_refine_4_realval}
    \sum_{u=1}^{k}\bar\Delta_{ih}[\zeta_{2u-1},\zeta_{2u})
    \le \sum_{u=1}^{k}\bar\Delta_{ih}[\rho_{2u-1},\rho_{2u})= \sum_{u=1}^k \sum_{\tau\in\mathcal T_{\gamma}[\rho_{2u-1},\rho_{2u})}\bar\Delta_{i,h,\tau}^{\operatorname{Re}}.
\end{align}
Since $\rho_{2u}-\rho_{2u-1}=\zeta_{2u}-\zeta_{2u-1}\in(1/\gamma)\mathbb Z_{\ge 1}$ for any $u\in[k-1]$, we have that
\begin{equation}\label{eq:pf_dp_additive_refine_1_realval}
    \mathcal T_{\gamma}[\rho_{2u-1},\rho_{2u})=\{\rho_{2u-1},\rho_{2u-1}+1/\gamma,\dots,\rho_{2u}-1/\gamma\},\qquad \forall u\in[k-1].
\end{equation}
Since $\rho_{2u}\le\rho_{2u+1}$, we have that
\begin{equation}\label{eq:pf_dp_additive_refine_2_realval}
   \mathcal T_{\gamma}[\rho_{2u-1},\rho_{2u}) \cap \mathcal T_{\gamma}[\rho_{2u^{\prime}-1},\rho_{2u^{\prime}})=\emptyset,\qquad \forall u\neq u^{\prime},\qquad u,u^{\prime}\in[k].
\end{equation}

By the definition of $\rho_{u}$, we have $\rho_{2u-1}\in\mathcal T_{\gamma}[t_h,t_h+d_{ih}),\rho_{2u}\le t_h+d_{ih}$ for any $u\in [k]$, which yields that 
\begin{equation}\label{eq:pf_dp_additive_refine_3_realval}
    \mathcal T_{\gamma}[\rho_{2u-1},\rho_{2u})\subseteq\mathcal T_{\gamma}[t_h,t_h+d_{ih}),\qquad \forall u\in[k].
\end{equation}
Recall that $\rho_{2u+1}\ge\rho_{2u}$, combining Eq.~\eqref{eq:pf_dp_additive_refine_1_realval}, Eq.~\eqref{eq:pf_dp_additive_refine_2_realval}, and Eq.~\eqref{eq:pf_dp_additive_refine_3_realval}, we obtain that \begin{align*}
    \underset{u\in[k]}{\bigsqcup} \mathcal T_{\gamma}[\rho_{2u-1},\rho_{2u})\subseteq \mathcal T_{\gamma}[t_h,t_h+d_{ih}).
\end{align*}
Since $\bar\Delta_{i,h,\tau}^{\operatorname{Re}}\ge 0$ for any $\tau\in\mathcal T_{\gamma}[t_h,t_h+d_{ih})$, combining this result with Eq.~\eqref{eq:pf_dp_additive_refine_4_realval}, we can conclude that \begin{align*}
    \sum_{u=1}^{k}\bar\Delta_{ih}[\zeta_{2u-1},\zeta_{2u})
    \le& \sum_{u=1}^k \sum_{\tau\in\mathcal T_{\gamma}[\rho_{2u-1},\rho_{2u})}\bar\Delta_{i,h,\tau}^{\operatorname{Re}}=\underset{\tau\in \underset{u\in[k]}{\bigsqcup} \mathcal T_{\gamma}[\rho_{2u-1},\rho_{2u})}{\sum}\bar\Delta_{i,h,\tau}^{\operatorname{Re}}\\
    \le& \underset{\tau\in\mathcal T_{\gamma}[t_h,t_h+d_{ih})}{\sum}\bar\Delta_{i,h,\tau}^{\operatorname{Re}}=\bar\Delta_{ih}[t_h,t_h+d_{ih}),
\end{align*}
which finishes the proof of Lemma~\ref{le:pf_dp_additive_refine_realval_1}. \Halmos

\section{Detailed Experimental Settings and Results}\label{app:experiments-details}

We conduct three sets of numerical experiments to evaluate the empirical performance of TS-BAL and GR-BAL. Section~\ref{subse:NE_hard_ins} considers two deliberately challenging families of instances: one targets FLB and Greedy, whereas the other is motivated by the lower-bound construction in Section~\ref{sec:lower-bound} and targets TS-BAL and GR-BAL. Section~\ref{subse:NE_random} evaluates the algorithms in randomly generated non-stationary environments. Section~\ref{subse:NE_smallcapacity} examines their performance in the small-capacity regime.

Throughout this section, performance is measured by the instance-wise competitive ratio $\mathrm{OPT}/\mathrm{ALG}$, where $\mathrm{OPT}$ is the offline optimal reward and $\mathrm{ALG}$ is the reward obtained by the corresponding online algorithm. Thus, a smaller ratio indicates better performance. We compare the following four algorithms:
\begin{enumerate}[nolistsep]
    \item TS-BAL: This algorithm is Algorithm~\ref{alg:main_d_integer_accurate} in our paper.
    \item GR-BAL:
    This algorithm is Algorithm~\ref{alg:main_d_integer_active} in our paper.
    \item FLB  (Forward-Looking Balance) in \cite{feng2025online}: As discussed in Remark~\ref{rem:comparison_feng}, we adapt the algorithm to locally bounded rewards by setting the assignment-specific loss to be \(\mathcal{L}^{\mathrm{FLB}}_{ij}=\sum_{\tau\in\{t_j+k:k\in\mathbb N,\ k<d_{ij}\}}\Psi\!\left(\alpha_{i,j-1\to\tau}\right)\). 
    \item Greedy: This algorithm chooses \(i_j^* \in \underset{i\in\mathcal N(j),~\text{$i$ is available}}{\arg\max}r_{ij}d_{ij}\).
\end{enumerate}

For TS-BAL and GR-BAL, we use the parameters specified in Sections~\ref{se:pf_choice_beta_eta} and~\ref{se:pf_choice_beta_eta_refine}, respectively. For FLB, we use the parameter calibration prescribed in \cite{feng2025online}. We use $\mathcal{N}_{[a,b]}(\mu,\sigma)$ to denote a normal distribution with mean $\mu$ and standard deviation $\sigma$, truncated to $[a,b]$.

\subsection{Artificially Designed Hard Instances}\label{subse:NE_hard_ins}
\noindent\underline{\bf Settings.} We first consider two families of hard instances artificially designed against the candidate algorithms. The first family is designed to expose the limitations of FLB and Greedy under globally drifting rewards. The second family is motivated by the lower-bound construction in Section~\ref{sec:lower-bound} and is designed to challenge TS-BAL and GR-BAL.

\begin{enumerate}[label=(\alph*),nolistsep]
    \item Consider a single server with capacity $c=5000$ and set $D=1$. For each tested value of $\delta$, let $M=\lfloor\delta^{8}\rfloor$. Each batch $m\in[ M ]$ contains $c$ low-reward jobs with reward rate $1$ and $c$ high-reward jobs with reward rate $\delta$. The low-reward jobs
    arrive at time $2m$, and the high-reward jobs arrive at time $2m+1/2$. All jobs in these batches have duration one. After all $M$ batches, one additional job with reward rate $\delta^{8}$ and duration one arrives at time $3M$.
    \item Consider a single server and $M=1000$ batches. Batch $k$ contains $c=200$ identical jobs, all arriving at $t_k=\frac{k-1}{M}$,  with reward rate $r_k=\delta^{t_k}$ and duration $d_k=\left\lfloor D^{t_k}\right\rfloor$. We set $\delta=D=10$. These batches generate a family of prefix instances: instance $m$ consists of batches $1,\dots,m$. We evaluate each algorithm on every prefix and report its competitive ratio as a function of $m$.
\end{enumerate}

\noindent\underline{\bf Results.} Figure~\ref{fig:Com_rab_hard_ins} shows that our algorithms are robust across both stress tests. In instance family (a), the competitive ratios of FLB and Greedy deteriorate rapidly as the reward scale increases, whereas TS-BAL and GR-BAL remain substantially more stable. In instance family (b), TS-BAL and GR-BAL remain comparable to FLB on the family constructed to challenge the proposed algorithms, while Greedy deteriorates sharply as the number of batches increases. Thus, the gains of TS-BAL and GR-BAL on the first hard family do not come at the cost of poor performance on the second.

\begin{figure}[t]
    \begin{center}
        \includegraphics[width =0.48\textwidth]{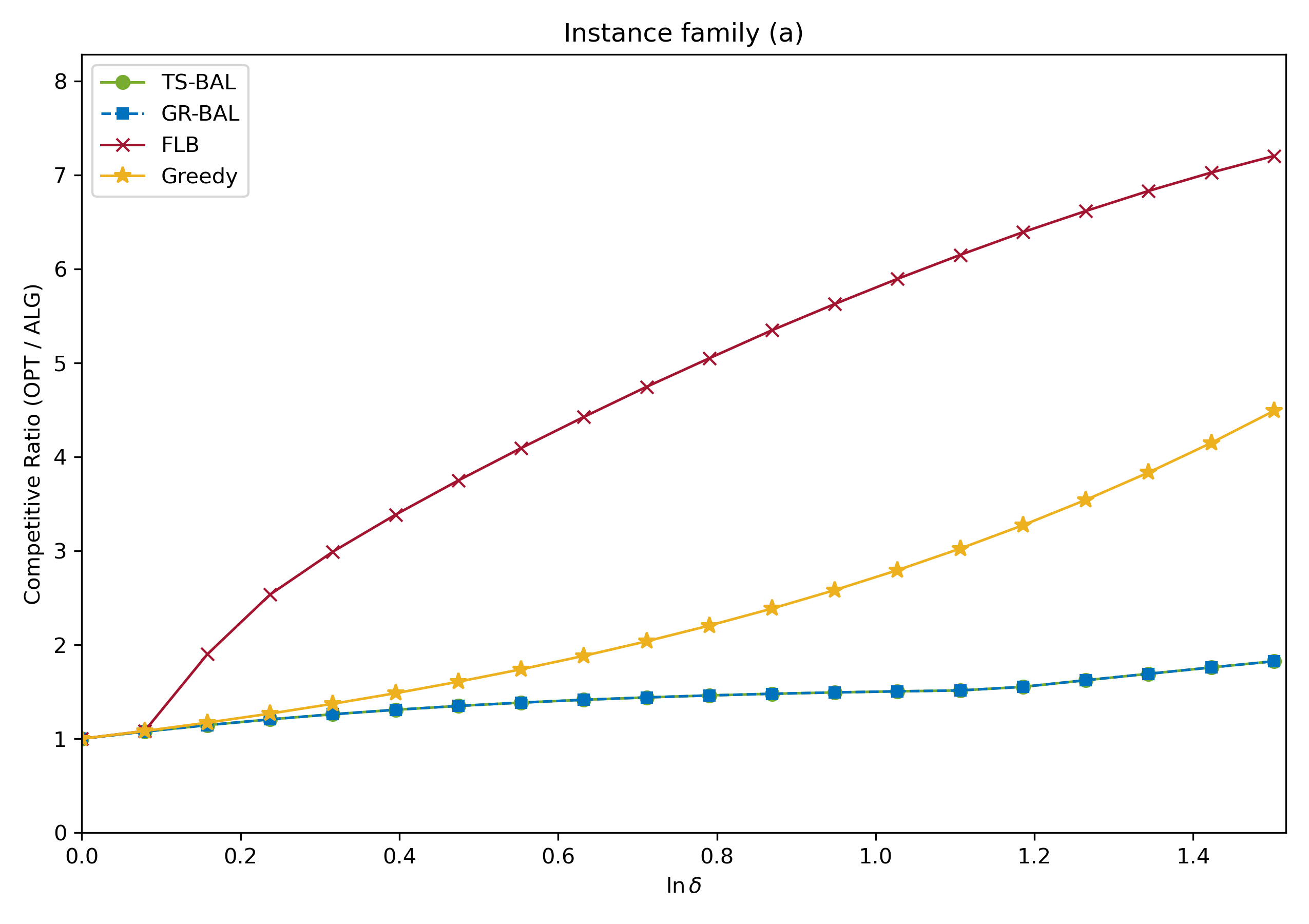}
        \includegraphics[width =0.48\textwidth]{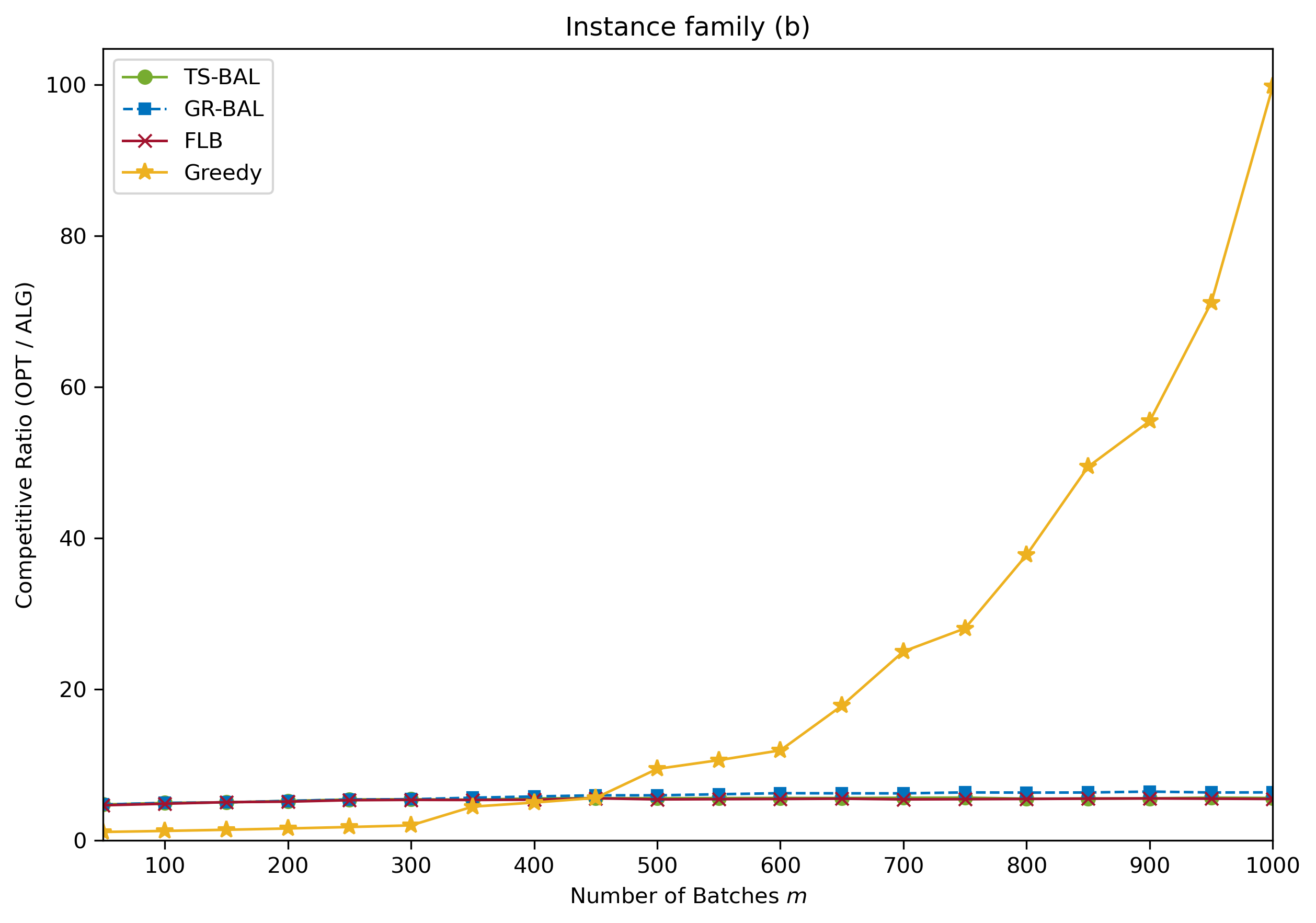}
    \end{center}
    \caption{Competitive-ratio comparisons on artificially designed hard instances. TS-BAL performs almost identically to GR-BAL, and its curves are therefore obscured by the GR-BAL curves.}\label{fig:Com_rab_hard_ins}
\end{figure}

\subsection{Randomly Generated Environments}\label{subse:NE_random}
We next compare TS-BAL, GR-BAL, FLB, and Greedy in randomly generated non-stationary environments. Similar with the design of the numerical experiments in \cite{feng2025online}, in this part, arrival times, reward rates, and assignment-specific processing durations are generated stochastically, while the compatibility graph is constructed to create local competition for reusable capacity.

\noindent\underline{\bf Setting.} We consider $n=24$ servers, each with capacity $c_i \equiv c\in\{40,60\}$. For each $T\in\{50,100,\ldots,1000\}$ and $\delta\in\{2,3\}$,
we generate $100$ independent instances and report the average instance-wise competitive ratio $\mathrm{OPT}/\mathrm{ALG}$. For each realization, $\mathrm{OPT}$ is computed by solving the exact mixed-integer programming formulation of the offline assignment problem.

Reward rates are generated from a non-stationary stochastic process on the
uniform time grid \(t_k=k\Delta\), where \(\Delta=0.05\). Let \(W_k\) denote
the reward level at grid point \(t_k\). We initialize \(W_0=1\). For each
\(k\ge 1\), we first generate the unconstrained candidate
\[
    Y_k
    =
    \max\left\{
        1,\,
        W_{k-1}+0.7\sqrt{\Delta}\,Z_k
    \right\},
    \qquad
    Z_k
    \overset{\mathrm{i.i.d.}}{\sim}
    \mathcal{N}(0.15,1.2),
\]
where \(1.2\) is the standard deviation. We then project \(Y_k\) onto an
interval that enforces the local reward condition over the preceding
\(D=10\) time units. Specifically, define
\[
    \mathcal{H}_k
    =
    \left\{
        \ell<k:
        t_\ell\ge \max\{0,t_k-D\}
    \right\},
\]
and let
\[
    L_k
    =
    \max\left\{
        1,\,
        \frac{1}{\delta}
        \max_{\ell\in\mathcal{H}_k}W_\ell
    \right\},
    \qquad
    U_k
    =
    \delta
    \min_{\ell\in\mathcal{H}_k}W_\ell.
\]
The projected reward level is $W_k=\min\left\{U_k,\,\max\{L_k,Y_k\}\right\}$.
Thus, when \(t_k<D\), the history window is truncated at time zero, and no
pre-zero reward history is used. By induction, the projection interval is
nonempty, \(W_k\ge 1\), and the grid-point reward levels satisfy
\[
    \frac{1}{\delta}
    \le
    \frac{W_k}{W_\ell}
    \le
    \delta
    \qquad
    \text{whenever }
    |t_k-t_\ell|\le D.
\]

We extend the grid-point process to continuous time by linear interpolation.
For \(t\in[t_k,t_{k+1}]\), define
\[
    W(t)
    =
    (1-\theta)W_k+\theta W_{k+1},
    \qquad
    \theta
    =
    \frac{t-t_k}{\Delta}.
\]
Because \(D/\Delta=200\) is an integer, a direct endpoint argument shows that
the grid-point local reward condition is preserved under linear
interpolation. In particular,
\[
    \frac{1}{\delta}
    \le
    \frac{W(s)}{W(t)}
    \le
    \delta
    \qquad
    \text{for all }s,t\ge 0
    \text{ such that }
    |s-t|\le D.
\]
Consequently, the generated continuous-time reward trajectory satisfies
Assumption~\ref{assumption:local-bounded-heterogeneity}.

Jobs arrive according to a reward-dependent non-homogeneous Poisson process
with intensity $\lambda(t)=\frac{292.5}{W(t)^2}$. We simulate this process using midpoint discretization with interval length \(h=0.05\). For each interval $I_m=[mh,(m+1)h)$, let $\bar{t}_m=\left(m+\frac{1}{2}\right)h$ denote its midpoint. We sample the number of arrivals in \(I_m\) according to
$N_m\sim\operatorname{Poisson}\left(h\lambda(\bar{t}_m)\right)$.
Conditional on \(N_m\), the \(N_m\) arrival times are sampled independently
and uniformly from \(I_m\). The arrival times generated across all intervals
are then pooled and sorted. Hence, although the reward process is initially
constructed on a discrete grid, job arrival times are continuous. For each job $j$ arriving at time $t_j$, its reward rate is identical across all servers and is given by $r_{ij} \equiv r_j = W(t_j)$. The processing durations are generated independently across jobs according to $d_{ij}=d_j=\left\lceil X_{j}\right\rceil$ with $X_{j}\sim\mathcal{N}_{[1,10]}(4.5,2)$ independently.

To construct compatible edges, for each $i\in[n]$, we assign server $i$ an ability level $g_i=0.25+0.65(1-x_i)^2$, where $x_i=\frac{i-1}{n-1} = \frac{i-1}{23}$, so that server abilities decrease from $0.90$ to $0.25$. We define the baseline job-difficulty level at time $t$ by the periodic function $s(t)=0.25 + 0.65 \times (1-\cos(2\pi (t-3\lfloor t/3\rfloor)/(2P)))/2$, where the period parameter is $P=3$. Figure~\ref{fig:s-t-plot} illustrates this function. Each job $j$ arriving at time $t_j$ is assigned a requirement level
$q_j=s(t_j)+\varepsilon_j$, where $\varepsilon_j\sim\mathcal{N}(0,0.12)$.
A server-job pair $(i,j)$ is compatible if and only if $g_i\ge q_j$.

\begin{figure}[h]
    \begin{center}
        \includegraphics[width =0.5\textwidth]{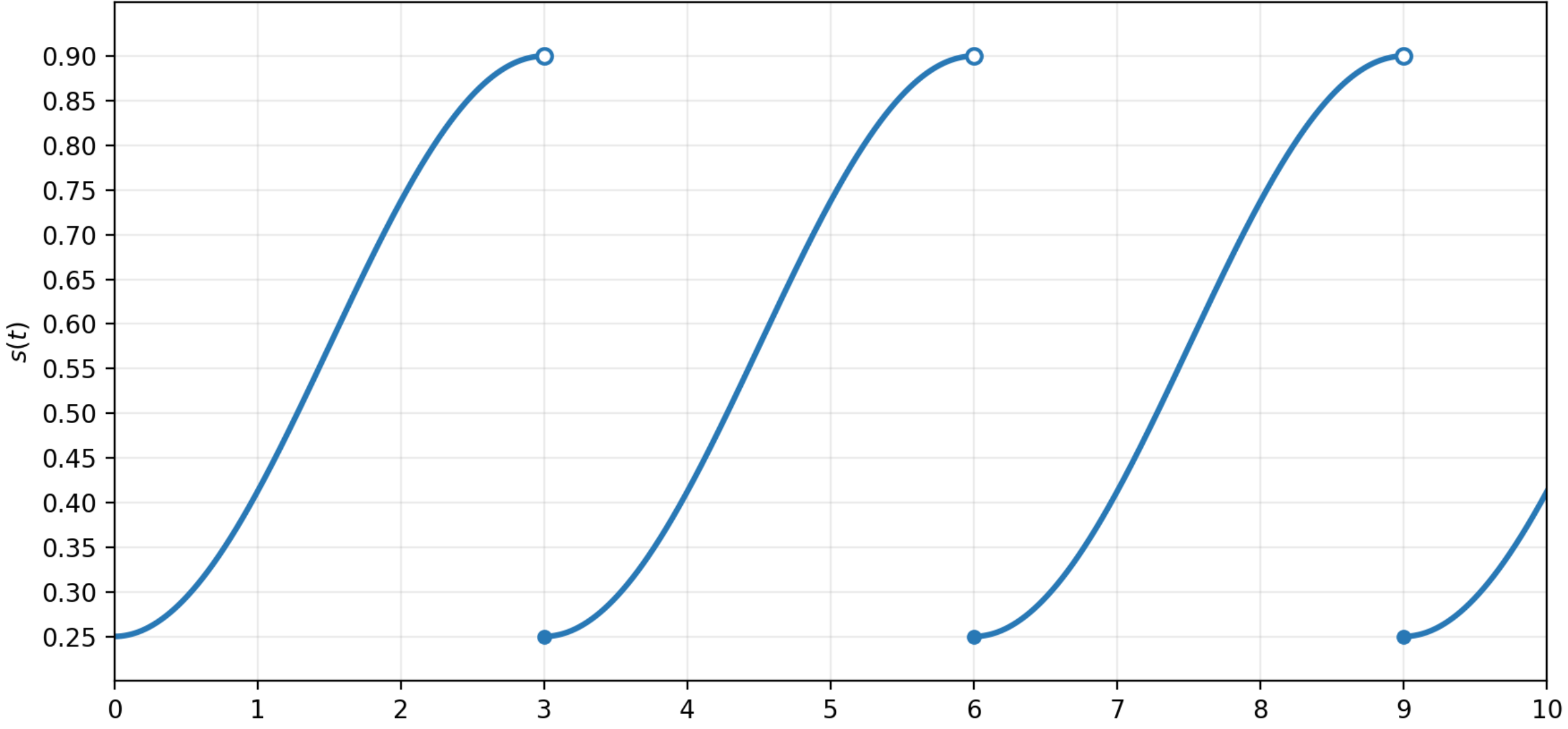}
    \end{center}
    \caption{The baseline job-difficulty level as a function of time $t$.}\label{fig:s-t-plot}
\end{figure}

\noindent\underline{\bf Results.} Figure~\ref{fig:Com_rab_random} shows that TS-BAL and GR-BAL consistently attain the two lowest average competitive ratios across the tested horizons, capacities, and values of $\delta$. Their performance also remains stable as the horizon grows. In contrast, FLB and Greedy generally incur larger ratios. We also note that, although the current analysis yields a competitive-ratio bound for TS-BAL whose leading term is twice that of GR-BAL, TS-BAL performs slightly better than GR-BAL in our experiments. This observation further underscores the question raised at the end of Section~\ref{sec:algorithm_main_theorem}: whether the factor of $2$ in the leading term of TS-BAL’s competitive-ratio bound can be eliminated.

\begin{figure}[h]
    \begin{center}
        \includegraphics[width =0.48\textwidth]{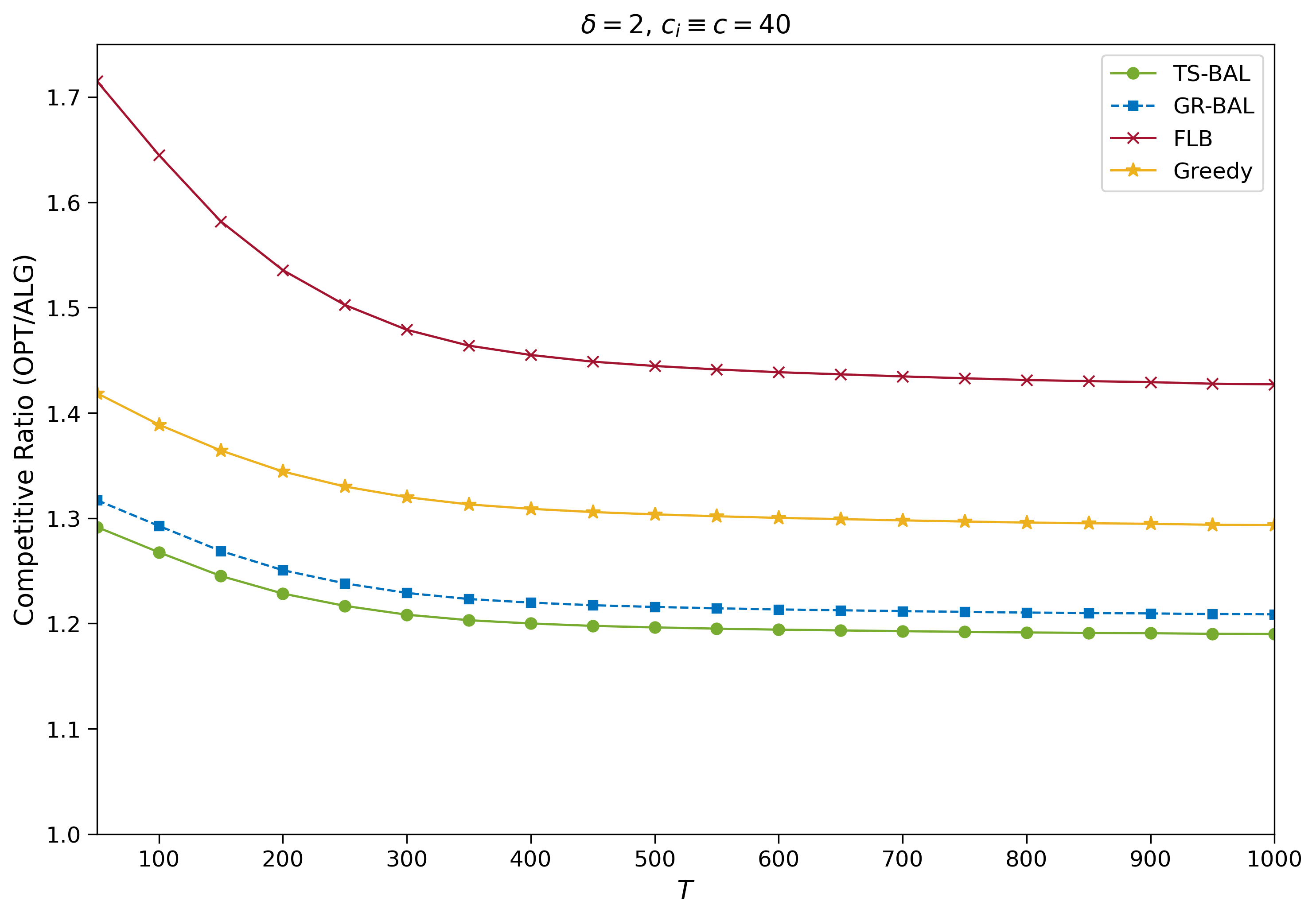}
        \includegraphics[width =0.48\textwidth]{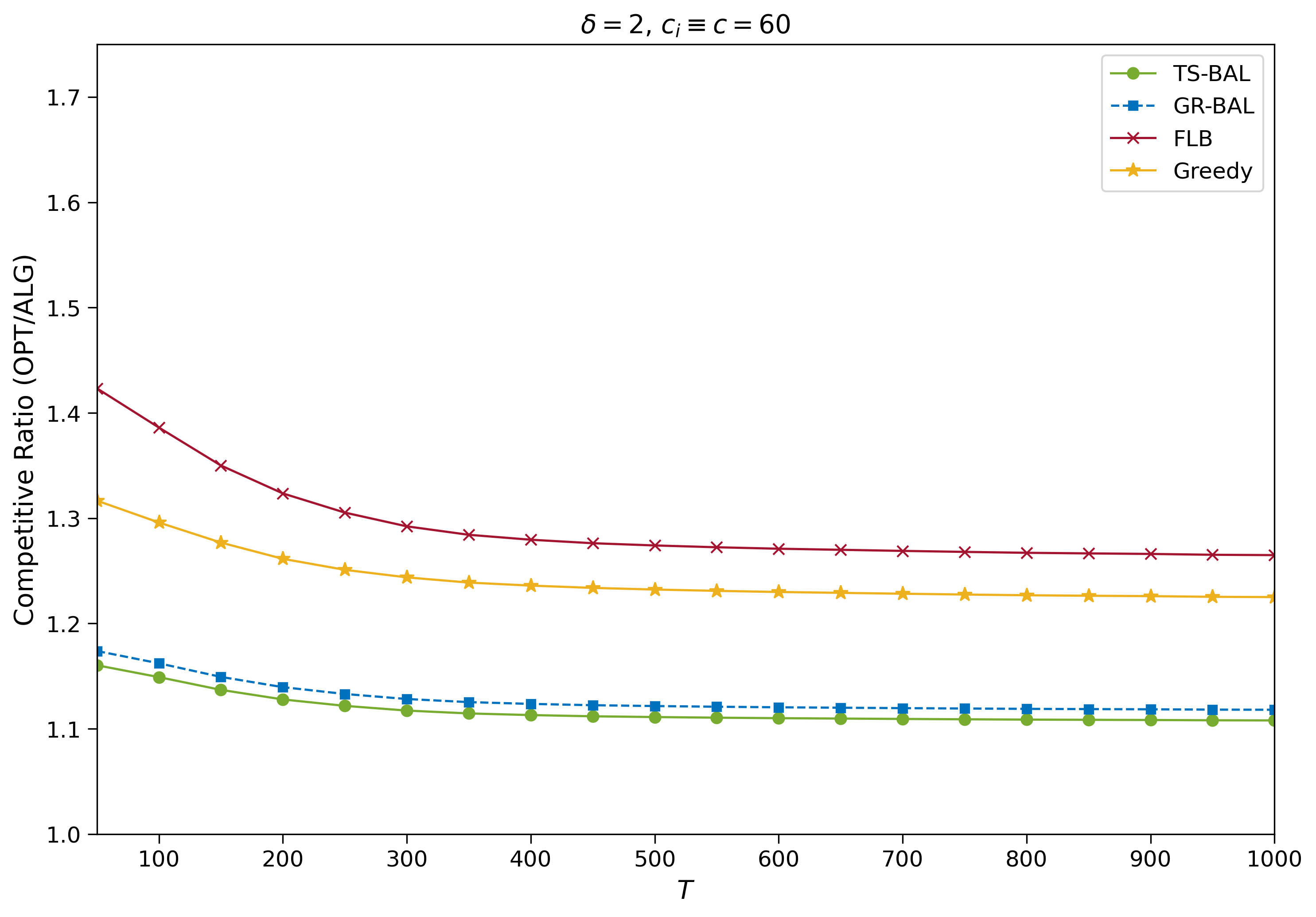}
        \includegraphics[width =0.48\textwidth]{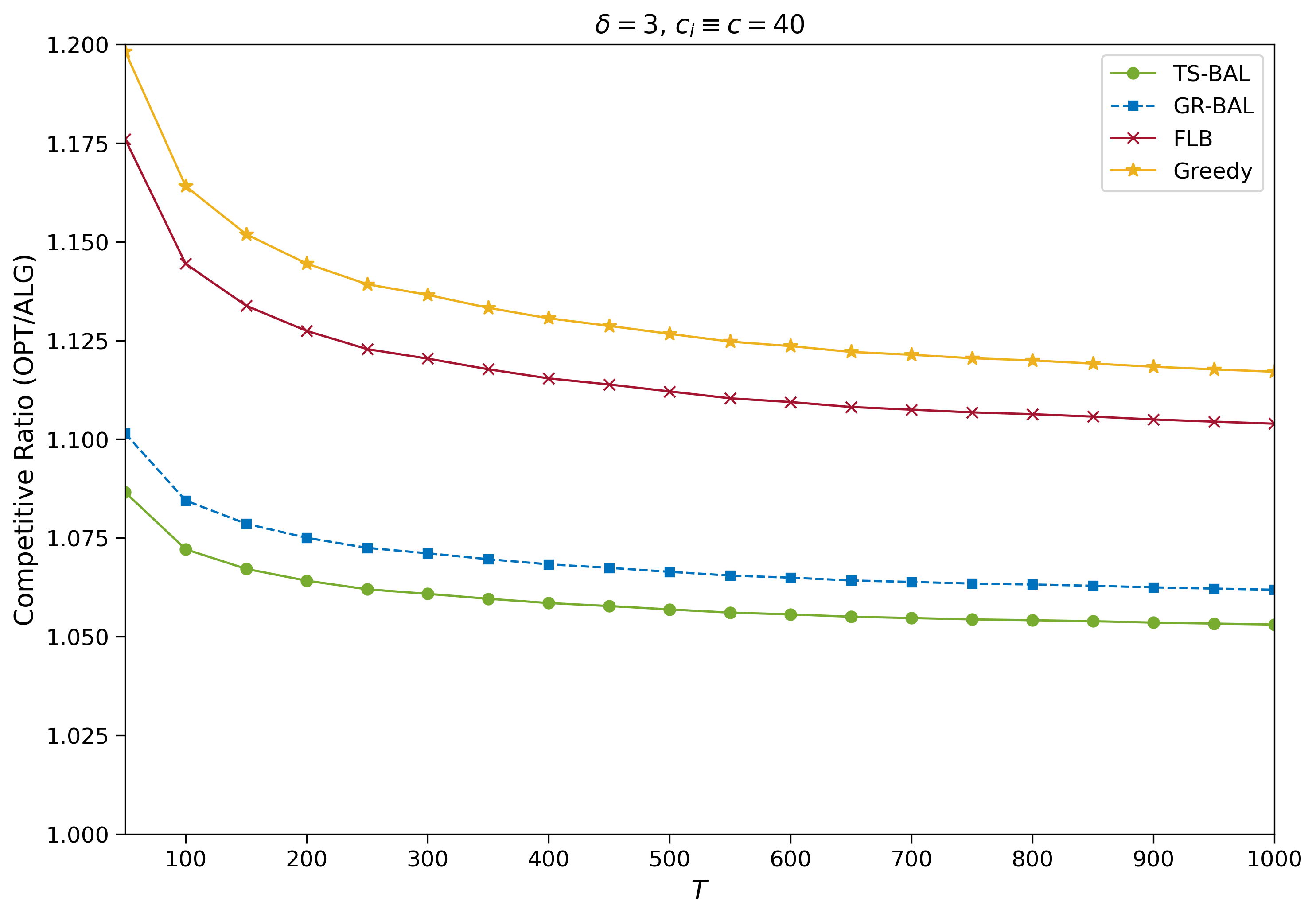}
        \includegraphics[width =0.48\textwidth]{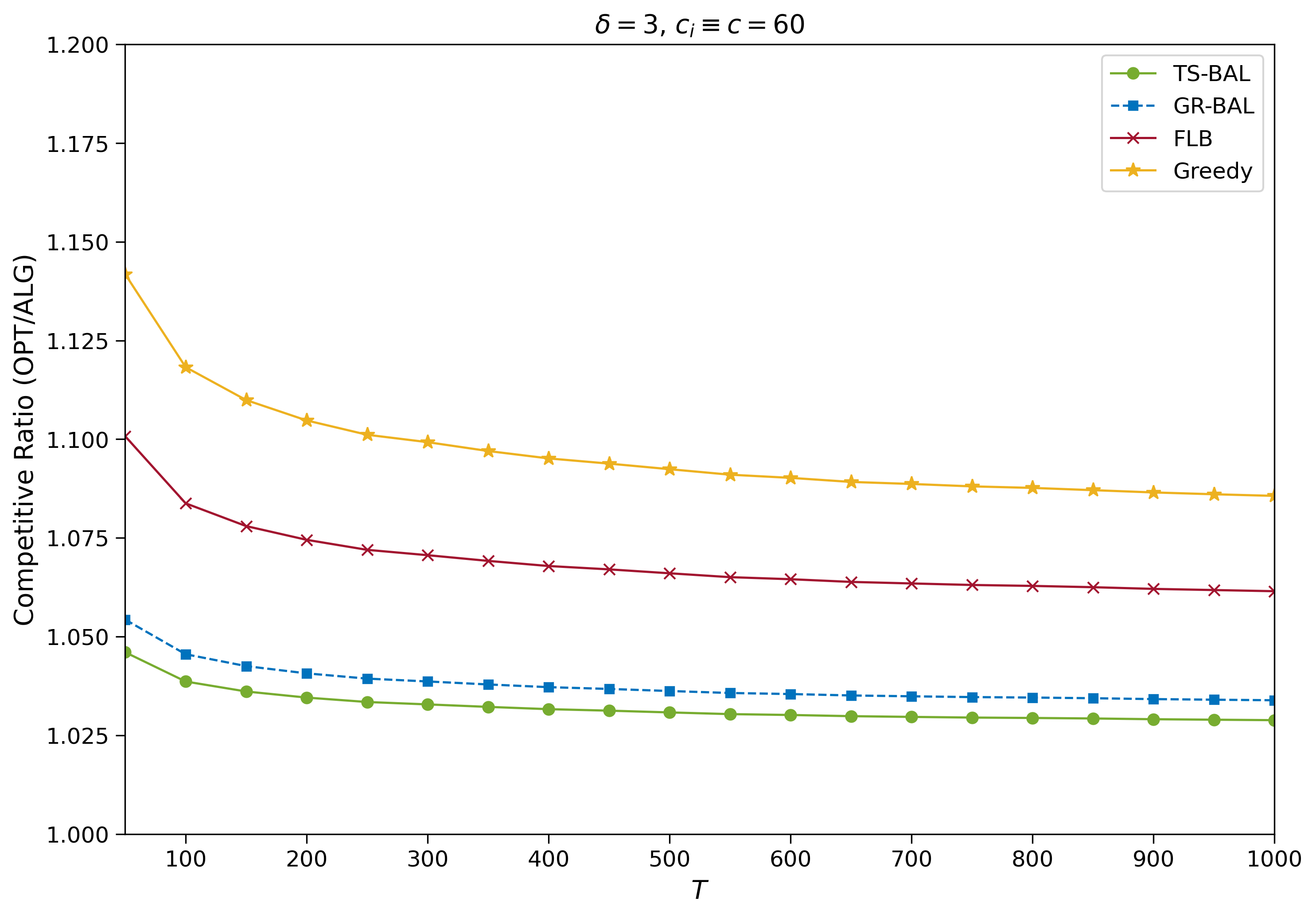}
    \end{center}
    \caption{Competitive-ratio comparisons in randomly generated environments.}
    \label{fig:Com_rab_random}
\end{figure}

\subsection{The Small-Capacity Regime}\label{subse:NE_smallcapacity}
In this part, we examine whether the proposed algorithms remain effective when server capacity is small. We use the same instance generator as in Section~\ref{subse:NE_random}, fix $T=500$, and vary the capacity of each server over $c\in\{2,4,6,\ldots,40\}$.
We conduct the experiment separately for $\delta=2$ and $\delta=3$, again reporting the average instance-wise competitive ratio over $100$ independent instances for each capacity level.

\noindent\underline{\bf Results.} Figure~\ref{fig:Com_rab_random_small_capa} reveals a clear capacity effect. At the smallest capacities, the relative performance of the algorithms depends on $\delta$: Greedy is particularly competitive when $\delta=2$, whereas TS-BAL and GR-BAL already perform favorably when $\delta=3$. For both values of $\delta$, after a brief initial non-monotonic phase, the competitive ratios of TS-BAL and GR-BAL decrease steadily as capacity increases. The proposed algorithms eventually outperform Greedy for both values of $\delta$ and substantially outperform FLB over the tested range. These results suggest that TS-BAL and GR-BAL retain meaningful finite-capacity performance, while their advantage becomes more pronounced as capacity grows, consistent with the large-capacity focus of our theoretical analysis.

\begin{figure}[t]
    \begin{center}
        \includegraphics[width =0.48\textwidth]{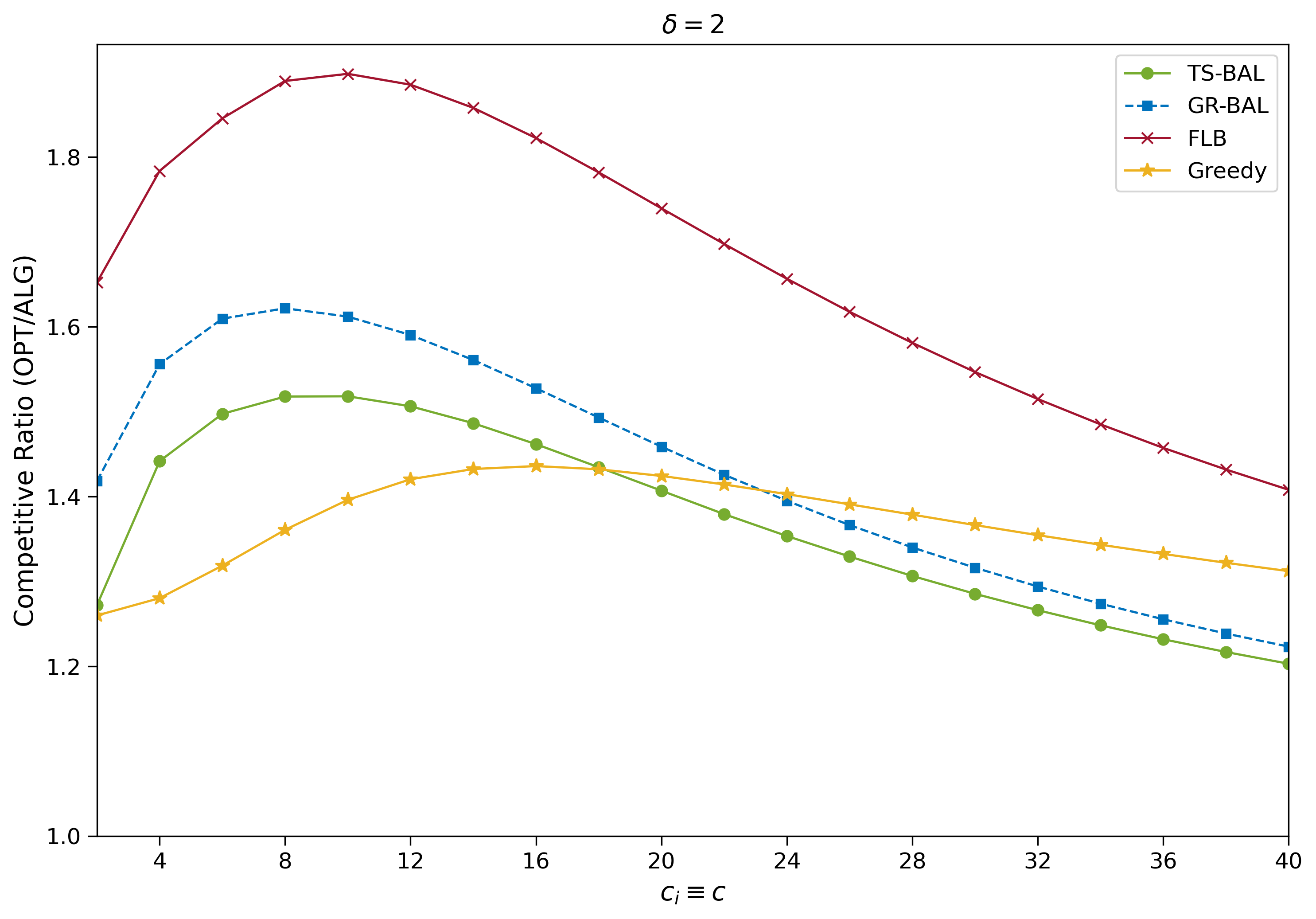}
        \includegraphics[width =0.48\textwidth]{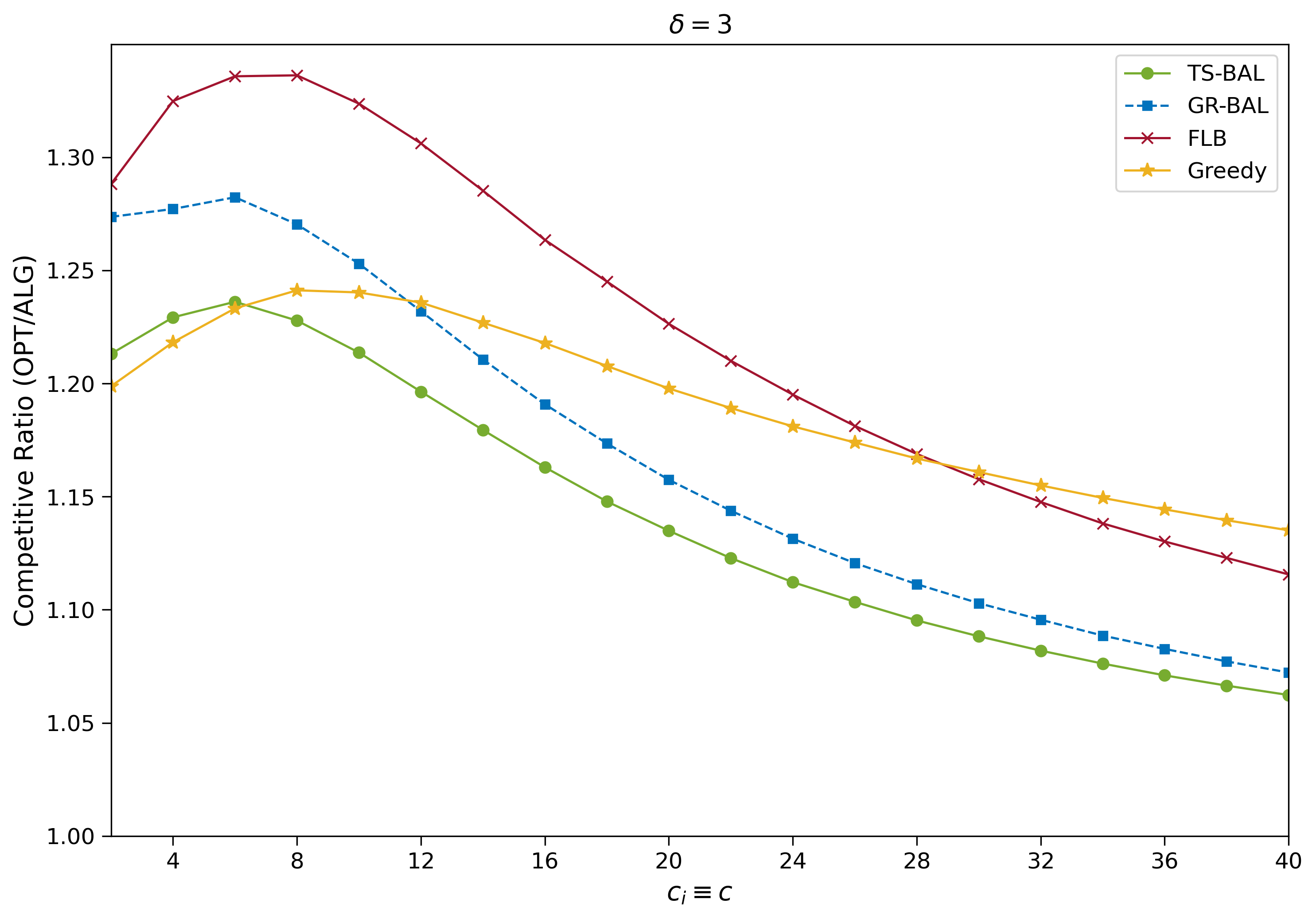}
    \end{center}
    \caption{Competitive-ratio comparisons when varying the minimum capacity.}\label{fig:Com_rab_random_small_capa}
\end{figure}

\end{APPENDICES}



\end{document}